\documentclass{article}
\usepackage{geometry}
\usepackage{graphicx}  
\usepackage[dvipsnames]{xcolor}
\usepackage{amsmath,amsfonts,amssymb,amsthm}
\usepackage{dsfont}
\usepackage{mathrsfs}
\usepackage{mathabx}
\usepackage{comment}
\usepackage{authblk}
\usepackage[normalem]{ulem}
\usepackage{longtable}

\definecolor{royalblue}{rgb}{0.25, 0.41, 0.88}
\usepackage[pdftitle={},
pdfauthor={},
colorlinks=true,linkcolor=MidnightBlue,urlcolor=RoyalBlue,citecolor=royalblue,bookmarks=true,bookmarksopen=true,bookmarksopenlevel=2,unicode=true,linktocpage]{hyperref}
\usepackage[capitalise,noabbrev]{cleveref}

\usepackage[labelfont={bf,sf}, textfont=sf]{caption}
\usepackage{subcaption}

\numberwithin{equation}{section}

\theoremstyle{plain}
\newtheorem{Thm}{Theorem}[section]

\newtheorem{Prop}[Thm]{Proposition}

\newtheorem{lemma}[Thm]{Lemma}

\newtheorem{Cor}[Thm]{Corollary}

\theoremstyle{definition}

\newtheorem{remark}[Thm]{Remark}

\title{\textsc{Scaling limits of critical FK-decorated random planar maps with $q=4$}}

\date{\today}

\newcommand{\R}{\mathbb{R}}
\newcommand{\Z}{\mathbb{Z}}
\newcommand{\looop}{\mathfrak{l}}

\newcommand{\Tb}{\mathbb{T}}
\newcommand{\pfrak}{\mathfrak{p}}
\newcommand{\tfrak}{\mathfrak{t}}
\newcommand{\Pmap}{P}

\newcommand{\rc}{\mathsf{c}}
\newcommand{\rC}{\mathsf{C}}
\newcommand{\rh}{\mathsf{h}}
\newcommand{\rH}{\mathsf{H}}
\newcommand{\rF}{\mathsf{F}}

\newcommand{\Pb}{\mathbb{P}}
\newcommand{\Eb}{\mathbb{E}}
\newcommand{\Var}{\mathrm{Var}}
\newcommand{\Gcal}{\mathcal{G}}

\newcommand{\hright}{h^{\rightarrow}}
\newcommand{\cright}{c^{\rightarrow}}
\newcommand{\hleft}{h^{\leftarrow}}
\newcommand{\cleft}{c^{\leftarrow}}
\newcommand{\Hright}{H^{\rightarrow}}
\newcommand{\Cright}{C^{\rightarrow}}
\newcommand{\Hleft}{H^{\leftarrow}}
\newcommand{\Cleft}{C^{\leftarrow}}
\newcommand{\tleft}{\tau^{\leftarrow}}
\newcommand{\thleft}{\tau^{\leftarrow,\rh}}
\newcommand{\tcleft}{\tau^{\leftarrow,\rc}}
\renewcommand{\th}{\mathtt{h}}
\newcommand{\tc}{\mathtt{c}}
\newcommand{\tG}{{\mathtt{G}}}
\newcommand{\tN}{{\mathtt{N}}}
\newcommand{\tauh}{\tau^{\th}}
\newcommand{\tauc}{\tau^{\tc}}

\newcommand{\cC}{\mathcal{C}}
\newcommand{\cS}{\mathcal{S}}
\newcommand{\cH}{\mathcal{H}}
\newcommand{\cD}{\mathcal{D}}

\newcommand{\Nleft}{N^{\leftarrow}} 
\newcommand{\Nright}{N^{\rightarrow}} 
\newcommand{\Xn}{Z^{(n)}}

\def\Xint#1{\mathchoice
{\XXint\displaystyle\textstyle{#1}}%
{\XXint\textstyle\scriptstyle{#1}}%
{\XXint\scriptstyle\scriptscriptstyle{#1}}%
{\XXint\scriptscriptstyle\scriptscriptstyle{#1}}%
\!\int}
\def\XXint#1#2#3{{\setbox0=\hbox{$#1{#2#3}{\int}$ }
\vcenter{\hbox{$#2#3$ }}\kern-.6\wd0}}
\def\dashint{\Xint-}

\begin{document}

\author{William Da Silva\thanks{University of Vienna, Austria, \texttt{william.da.silva@univie.ac.at}} \quad \quad
Xingjian Hu\thanks{Fudan University, Shanghai, China, \texttt{22110180020@m.fudan.edu.cn}} \quad \quad  Ellen Powell\thanks{Durham University, UK, \texttt{ellen.g.powell@durham.ac.uk}} \quad \quad Mo Dick Wong\thanks{The University of Hong Kong, Hong Kong,
\texttt{mdwong@hku.hk}}}
\date{}

\maketitle

\vspace{-0.4cm}
\begin{abstract}
We establish the first scaling limit for FK($q$)-weighted planar maps in the critical case $q=4$, resolving a problem that has remained open since Sheffield's seminal work~\cite{SheffieldScott2016QGAI}. In that work, Sheffield proved a scaling limit for $q<4$ via the celebrated hamburger-cheeseburger bijection, which initiated the peanosphere (mating-of-trees) approach to Liouville quantum gravity. We prove that, at criticality, the associated burger count $\cS$ and discrepancy $\cD$ satisfy
\[
\left(\frac{\mathcal{S}_{\lfloor nt \rfloor}}{\sqrt{n}}, \frac{\log(n)}{{2\pi }\sqrt{n}} \mathcal{D}_{\lfloor nt \rfloor}\right)_{t\in\mathbb{R}}
\stackrel{\text{d}}{\longrightarrow}
(B^1_t, B^2_{t})_{t\in\mathbb{R}},
\]
where $B^1$ and $B^2$ are independent two-sided Brownian motions. To the best of our knowledge, no conjecture for the correct discrepancy scaling factor had previously been formulated. Matching the limiting process with the critical mating of trees~\cite{AHPS}, we establish the first rigorous planar map convergence towards CLE$_4$ and critical ($\gamma=2$) Liouville quantum gravity, in the peanosphere sense. Our proof is based on a novel approach that reveals the exactly solvable nature of the model through a correspondence with the (bicoloured) fully packed loop-$O(2)$ model on triangulations, and
yields critical geometric exponents matching the predictions of conformal field theory.
\end{abstract}

\begin{figure}[h!]
	\vspace{-0.33cm}
	\centering
	\includegraphics[width=.31\textwidth]{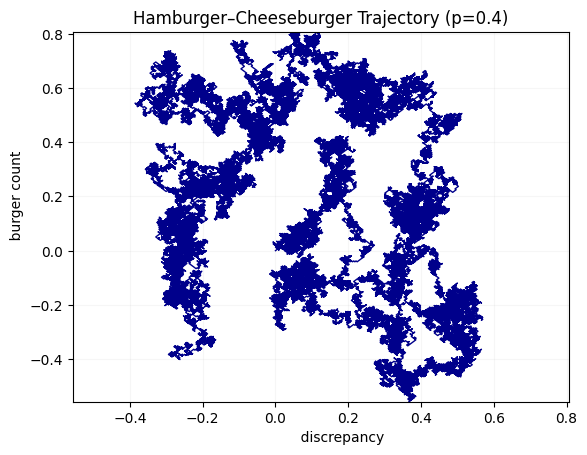} 
	\includegraphics[width=.31\textwidth]{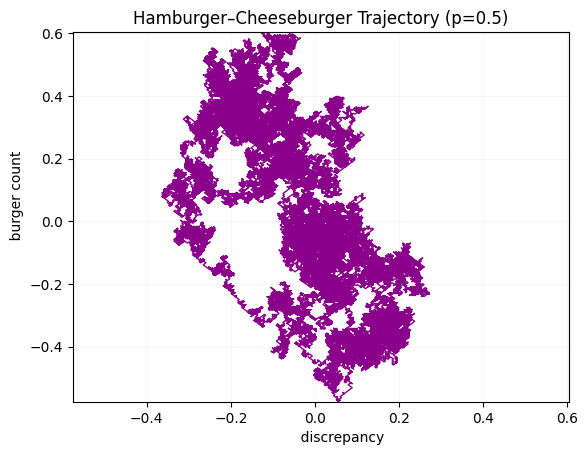}
	\includegraphics[width=.31\textwidth]{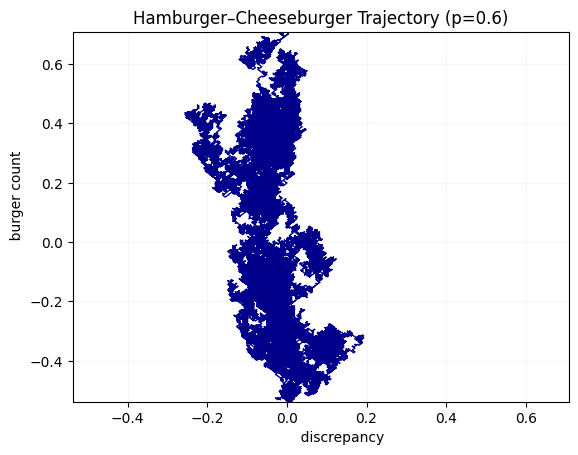}  
	\caption{Numerical simulations of the hamburger-cheeseburger trajectory for different values of $q$. We took $N=10^6$ in the inventory model and represented the burger count and discrepancy, both rescaled by $\sqrt{N}$. For $q>4$ (right), the discrepancy collapses to $0$, marking a phase transition. This paper is concerned with the critical case $q=4$ (middle), where we show that a logarithmic correction emerges.}
    \label{fig:intro_sim}
\end{figure}

\tableofcontents

%
%

\section{Introduction}

Liouville quantum gravity (LQG) surfaces are canonical models of random planar geometry, originally motivated by conformal field theory and string theory \cite{polyakovstrings}. 
They possess a rich mathematical structure, have deep connections to probability, geometry, combinatorics and theoretical physics \cite{knizhnik1988fractal,DuplantierSheffield,DKRV}, and remain the subject of many open problems{; see} \cite{BerestyckiNathanaël2024Gffa} {for an extensive exposition}. 
A central conjecture is that LQG surfaces, {which form a one-parameter family indexed} by $\gamma\in (0,2]$,  arise as scaling limits of discrete random surfaces known as random planar maps.

The best known results in this direction {concern models where face degrees are suitably controlled. Progress culminated in the works of} Le Gall \cite{le2013uniqueness}, Miermont \cite{miermont2013brownian}, Miller--Sheffield \cite{LQGTBM} and Holden--Sun \cite{HoldenSun}, {which establish convergence in} the uniform {bounded-degree} case ($\gamma=\sqrt{8/3}$). 
It has been predicted in physics, see for example \cite{DKRV} for a precise conjecture, that a natural way to escape from the \emph{pure-gravity} universality class is to couple the maps with critical statistical mechanics models. Proving rigorous convergence results in this general setting remains {a} major open {problem, although} spectacular progress has been made recently for the related model of Boltzmann planar maps with heavy-tailed face degree distributions \cite{curien2025scaling}.

In this work, we consider the infinite-volume Fortuin--Kasteleyn (FK) decorated planar map model, which provides a natural setting for studying scaling limits outside of the pure-gravity regime.  For $q\ge 0$, a finite {(self-dual)} FK$(q)$-decorated planar map of size $k$ is a random planar map $M_k$ with $k$ edges together with a distinguished subset of edges $\Pmap_k$. 
{The probability of sampling a given consistent pair $(\mathfrak{m},\pfrak)$ satisfies
\[
\Pb((M_k,\Pmap_k) = (\mathfrak{m},\pfrak)) 
\; \propto \; 
q^{\#\text{K}(\pfrak)+ \frac12\#\text{E}(\pfrak)-\frac12 \#\text{V}(\mathfrak{m})},
\]
where $\#\text{K}(\pfrak)$, $\#\text{E}(\pfrak)$ and $\#\text{V}(\mathfrak{m})$ denote, respectively, the number of connected components, edges and vertices of $\pfrak$ or $\mathfrak{m}$.}
Equivalently, the model can be constructed by first sampling a planar map weighted by the self-dual FK($q$)-partition function, and then drawing an FK($q$)-percolation configuration on top of it \cite{duminil2013parafermionic}.
The infinite-volume model is obtained as the local limit of such maps as the size tends to infinity; see \cite{SheffieldScott2016QGAI,ChenLinxiao2017Bpot}.

This planar map model is intrinsically related to the $q$-state Potts, Ising, and loop-$O(n)$ models.
On a fixed lattice, it has been shown in the landmark paper \cite{duminil2021discontinuity} that the model displays a phase transition at $q=4$, while it is known that on planar maps some of the observables degenerate at this critical threshold, see \cite{SheffieldScott2016QGAI,feng2023triviality}. This makes the boundary case $q=4$ particularly challenging and significant, as  fundamentally new geometric and probabilistic behaviour emerges.

FK$(q)$-decorated maps {with $q\in [0,4]$} are believed to lie in the universality class of $\gamma$-LQG, where $\gamma\in [\sqrt{2},2]$ is determined by $q$ via the relation
\[
q = 2 + 2\cos(\pi\gamma^2/2).
\]
In particular, after conformally embedding the map into $\mathbb{C}$ {in a canonical way}, and scaling so that $N$ faces are mapped to the unit disc, the embedded metric measure space is conjectured to converge under suitable rescaling to a $\gamma$-LQG surface as $N\to \infty$. Moreover, the loops separating primal and dual FK clusters are expected to converge jointly with the surface to an independent conformal loop ensemble $\mathrm{CLE}_{16/\gamma^2}$. 

A powerful approach to studying these maps comes from bijections due to Mullin \cite{mullin67} (for $q=0$), and Bernardi \cite{bernardi08} followed by Sheffield \cite{SheffieldScott2016QGAI} for general $q \ge 0$. These bijections encode (infinite) FK-decorated planar maps by (bi-infinite) words in a finite alphabet, interpreted as an inventory process of two item types (``hamburgers'' and ``cheeseburgers'' in the terminology of \cite{SheffieldScott2016QGAI}). The resulting word defines a (bi-infinite) two-dimensional lattice path, {which} encodes the geometry of the underlying planar map.

In \cite{SheffieldScott2016QGAI}, Sheffield proved that these {non-Markovian} bi-infinite lattice paths, when rescaled by $n^{-1/2}$, converge to a correlated planar Brownian motion when $q<4$.  Via the remarkable mating-of-trees theory \cite{DuplantierMillerSheffield},
this proves convergence of FK-decorated maps to $\gamma$-LQG surfaces decorated by $\mathrm{CLE}_{16/\gamma^2}$ in the {so-called} peanosphere topology. This is a powerful result in its own right, and an important step towards resolving the conjecture for conformally embedded maps, see e.g.~\cite{tutteLQG}. 

On the other hand, at the critical value $q=4$ the picture degenerates: the rescaled lattice path converges to a one-dimensional Brownian motion, leaving the conjectural connection to critical LQG ($\gamma=2$) and $\mathrm{CLE}_4$ unresolved.

To investigate this critical regime, write $\cS_n = \cH_n + \cC_n$ and $\cD_n = \cH_n - \cC_n$ for the sum and discrepancy of the inventory path for an infinite FK($4$)-decorated map, see \cref{sss: counts}. Sheffield's result implies that
\[
(n^{-1/2}\cS_{nt},\, n^{-1/2}\cD_{nt}) \overset{\text{d}}{\longrightarrow} (B_t, 0),
\]
as $n\to \infty$, so a different rescaling is needed to observe non-trivial behaviour in $\cD_n$. Identifying the correct scaling and obtaining a limit was posed as an open problem in \cite{SheffieldScott2016QGAI}. {Our main result fully answers this question.}

\begin{Thm}[Main result: scaling limit of critical hamburger-cheeseburger walk]
\label{thm: main intro} 
We have the convergence in distribution
\begin{equation}\label{eq: main}
\left( \frac{\mathcal{S}_{\lfloor nt\rfloor}}{\sqrt{n}},\ \frac{\log n}{2\pi\sqrt{n}}\,\mathcal{D}_{\lfloor nt\rfloor} \right)_{t\in \R}
\overset{\textnormal{d}}{\longrightarrow}
\left(B^1_t,\, B^2_{t}\right)_{t\in \R}
\end{equation}
in the space of c\`{a}dl\`{a}g functions with the local Skorohod $J_1$ topology as $n\to \infty$, where $B^1$ and $B^2$ are independent standard two-sided Brownian motions.  
\end{Thm}

We also determine the asymptotic variance of $\cD_n$, which was only shown to be $o(n)$ in \cite{SheffieldScott2016QGAI}. {We emphasise that our proof of \cref{thm: main intro} does not rely on this variance estimate.}

\begin{Thm}[Variance estimate]
\label{thm:variance}
We have that 
\[
 {\Var(\cD_n)} \asymp \frac{n}{\log^2(n)}
\quad \text{as } n\to \infty.
\]
More precisely, we have $\liminf_n \frac{\log^2(n)}{n}\Var(\cD_n)\ge 4\pi^2$ and $\limsup_n \frac{\log^2(n)}{n}\Var(\cD_n)\le 8\pi^2$.
\end{Thm}

We believe that $\frac{\log^2(n)}{n}\Var(\cD_n)\to 4\pi^2$ (that is, the covariance of the pair on the left-hand side of \eqref{eq: main} converges to the covariance of the limit), but we have not yet pursued this direction since it is not needed for our main result.

Our results  gain further significance in light of recent work showing that critical LQG surfaces and $\mathrm{CLE}_4$ can be represented via mating-of-trees encodings using planar Brownian motion: \cite{AHPS}. Moreover, our techniques allow us to obtain scaling exponents for geometric observables in the FK$(4)$-planar map model. 

For this, we need the notion of a typical loop $\mathfrak{L}$,  a typical cluster $\mathfrak{c}$ and a typical envelope $\mathfrak{e}$ in the infinite FK$(4)$-planar map, as studied for $q<4$ in \cite{berestycki2017critical,GwynneEwain2019Slft} (the envelope is also referred to as a bubble in \cite{berestycki2017critical}). These are defined precisely in \cref{sec:loops-clusters-env} and correspond to choosing a loop,  cluster or envelope uniformly at random from an FK$(4)$-planar map of finite size $k$, and letting $k\to \infty$. 
We give natural definitions of the length $|\mathfrak{L}|$ of a typical loop, the boundary length $|\partial \mathfrak{c}|$ of a typical cluster and the boundary length $|\partial \mathfrak{e}|$ of a typical envelope in \cref{sec:loops-clusters-env}.

\begin{Thm}[Loop and cluster exponents]
\label{thm:main_exponents}
The following tail asymptotics hold: as $\ell\to \infty$, 
\[ \Pb(|\mathfrak{L}|=\ell) \sim \frac{256}{\pi^4} \frac{\log^2 \ell}{\ell^3},  \quad \Pb(|\partial\mathfrak{c}|=\ell) \sim \frac{32}{\pi^4} \frac{\log^2 \ell}{\ell^3} \quad \text{and} \quad \mathbb{P}(|\partial \mathfrak{e}|\ge \ell)\sim \frac{\pi^2}{\ell \log^3(\ell)}.  \]
\end{Thm}

The exponents are consistent with the results of \cite{berestycki2017critical,GwynneEwain2019Slft}, taking a limit as $q\to 4$, although no conjecture was made for the explicit slowly varying corrections. We point out that in contrast to the previous literature \cite{berestycki2017critical,GwynneEwain2019Slft}, our strategy involves exact calculations (see below) that allow us to identify these exponents, which we do before establishing the scaling limit of Theorem \ref{thm: main intro}. In the subcritical case, the scaling limit of $(\cC_n,\cH_n)$ is the main tool used to obtain the exponents. {We also stress that, contrary to \cite{berestycki2017critical,GwynneEwain2019Slft}, our asymptotics are exact and do not involve any $\ell^{o(1)}$ correction.}

It is the observable $|\partial \mathfrak{e}|$ that can be identified most easily with a continuum counterpart. Namely, in the work \cite{lqgcle4}, it is shown that exploring a critical Liouville quantum gravity surface and independent conformal loop ensemble with parameter $4$ in an appropriate way, the {components separated from a quantum typical point are conditionally independent critical quantum discs given their boundary lengths. Moreover, it is a consequence of the critical mating of trees \cite{AHPS} that these boundary lengths can be described as} the jumps of a $1$-stable process, i.e.\ a process whose intensity scales the same way as the density of $|\partial \mathfrak{e}|$. 
We give a more precise interpretation in \cref{sec: scaling limit of rw}, which requires a little more background.

The first key step in our strategy is to use a correspondence between the FK$(4)$ decorated map model and the fully packed loop-$O(2)$ model on triangulations, see \cref{sec:O(2)_model}. For the latter model, using powerful tools from analytic combinatorics developed by \cite{borot2012recursive}, we obtain an exact expression and asymptotic identity for the partition function, which is a novel result of independent interest, making rigorous the physics prediction \cite{gaudin1989n}.

Let $F_\ell$ be the partition function for rooted triangulations with boundary length $\ell$, under the $O(n)$ model with weights $n^{\# \text{loops}} x^{\# \text{triangles}}$ at $n=2, x=(4\sqrt{2})^{-1}$, see \cref{sec:O(2)_model}.
\begin{Thm}[Exact expression and asymptotics of $F_\ell$]
\label{thm:asympt_F_ell_intro}
The partition function satisfies
\[
F_{\ell} = 2(2\sqrt2)^{\ell}\int^1_0 u\bigg(1-\frac{\pi}{2}u\bigg)^{\ell}\log\bigg(\frac{1 + \sqrt{1-u^2}}{u}\bigg)\mathrm{d}u. 
\]
In particular, we have
\begin{equation*}
F_{\ell} \sim \frac{8}{\pi^2}(2\sqrt2)^{\ell}\frac{\log\ell}{\ell^2}
\quad \text{as } \ell\to\infty.
\end{equation*}
\end{Thm}

\noindent 
Note that the form of the asymptotic behaviour features the exponent $2$, which is indicative of criticality. Previously, the asymptotics of the $O(2)$ partition function were obtained only for the ``rigid'' model \cite{aidekon2024scaling,kammerer2024gaskets}, using a different probabilistic argument.
{We also mention that it is partly possible to extend the results to FK$(q)$ planar maps with $q\in(0,4)$. In a concurrent joint work with Berestycki \cite{berestycki2025critical}, the first-named author pursues this approach to derive exact tail exponents (analogous to \cref{thm:main_exponents}) and an explicit expression and asymptotic for the partition function of the associated $O(n)$ model (as in \cref{thm:main_exponents}). However, we note that this analysis builds on Sheffield’s convergence result \cite{SheffieldScott2016QGAI}.}

Our strategy to establish the scaling limit result (\cref{thm: main intro}) proceeds as follows. We begin by mapping the FK$(4)$-weighted planar map model to the fully-packed $O(2)$ model on triangulations. Within this framework, we derive an exact expression for the partition function $F_\ell$ using tools from analytic combinatorics (\cref{thm:asympt_F_ell_intro}). 
We then reveal a natural probabilistic interpretation of this partition function in terms of two random-walk-type observables that encode complementary peeling explorations on the map, which we refer to as the \emph{exploration into the past} and the \emph{exploration into the future}. 
We rigorously establish the scaling limits of these two explorations, and, in the course of this analysis, derive \cref{thm:main_exponents} from the study of the former.
Finally, we use ideas from excursion theory to ``squeeze'' the discrepancy $\cD_n$ between the two explorations and thereby deduce the full scaling limit. {We refer to \cref{ssec: outline}} for a more detailed overview of the proof outline.

The paper is organised as follows. In \cref{sec:prelim}, we provide some background on the FK-decorated planar map model, the Mullin--Bernardi--Sheffield bijection and the correspondence with the loop-$O(n)$ model. 
We also introduce the two key processes -- the \emph{exploration into the past} and the \emph{exploration into the future} -- which play a central role throughout this work.
\cref{sec:resolvent_F_ell} uses analytic combinatorics arguments to derive the exact expression and asymptotic behaviour of the partition function $F_\ell$. In \cref{sec: geometric exponents}, we connect this partition function to geometric observables, including the loop and cluster exponents. 
The subsequent sections, \cref{sec:backward,sec:explo_future}, focus respectively on the analysis of the explorations into the past and future, proving their respective scaling limits. 
In \cref{sec:tightness_Dn}, we combine these two explorations to show that the discrepancy $\cD_n$ is tight at scale $\frac{\sqrt{n}}{\log n}$. This is the crucial step that allows us to conclude the proof of \cref{thm: main intro} in \cref{sec:final_cvg}.

The paper also contains \cref{appendix_asymp} that is devoted to various integral expansions. Simulations of hamburger-cheeseburger trajectories for several values of $q$ are presented in \cref{fig:intro_sim} and in \cref{appendix:sim}.

\paragraph{Acknowledgements.} 

We are grateful to Xin Sun for suggesting the article \cite{borot2012recursive} as a starting point for deriving exact partition functions. We thank \'{E}lie A\"{i}dékon, Nathana\"{e}l Berestycki, Yuyang Feng, Ewain Gwynne and Julie Tourniaire for insightful discussions.
W.D.S.\ acknowledges the support of the Austrian Science Fund (FWF) grant on “Emergent branching structures in random geometry” (DOI: 10.55776/ESP534).
X.H.\ is supported by the China Scholarship Council. 
The research of E.P.\ is supported by UKRI Future Leader's Fellowship MR/W008513.  
M.D.W.\ is supported by a start-up allowance from Croucher Foundation and a start-up fund from Faculty of Science, The University of Hong Kong.

\paragraph{Notation.}
We use the notation $\mathbb{N} = \{0,1,2,\ldots\}$ for the set of natural numbers, including $0$.
For real-valued functions $f$ and $g$ defined on some domain $D$ (possibly sequences), we write
$f(x)=o(g(x))$
if
${f(x)}/{g(x)}\rightarrow 0$ as $x\to \infty$, $x\in D$.
We also write
$f(x)=O(g(x))$
if there exists a constant $C>0$ such that
$|f(x)|\leq C\, |g(x)|$
for all $x\in D$, and 
$f(x) \sim g(x)$
if 
${f(x)}/{g(x)} \to 1$
as $x\to\infty$, $x\in D$.
For a sequence of random variables, we write $Y_n \overset{\mathrm{d}}{\longrightarrow} Y$ (resp.\ $Y_n\overset{\mathbb{P}}{\longrightarrow} Y$) as $n\to\infty$ to denote convergence in distribution (resp.\ in probability) to the random variable $Y$. Likewise, we write $Y\stackrel{\mathrm{d}}{=}Z$ (resp.\ $Y=Z$ a.s.)  when two random variables $Y$ and $Z$ have the same distribution (resp.\ are almost surely equal).
Finally, we also use the notation $\mathrm{d}z$ for real integrals and $dz$ for complex integrals, and $\textstyle{\oint}_{\mathscr{C}}$ for integration along the oriented contour $\mathscr{C}$. 

\section*{Index of notation}
\label{appendix:notation}

\begin{flushleft}
  \begin{longtable}{cl}
    $(\mathfrak{m}, \pfrak)$ & planar map $\mathfrak{m}$ decorated with a subset of edges $\pfrak$ \\
    $\Theta = \{\rh,\rc,\rH,\rC,\rF\}$ & alphabet of hamburger-cheeseburger letters \\
    $\overline{w}$ & reduction of the word $w$, see \cref{sec:mullin-bernardi-sheffield} \\
    $\mathsf{w}(\theta)$ & weight of the letter or word $\theta$, see \eqref{eq: symbol weights} \\
    $(X(k), k\in\Z)$ & sequence of i.i.d.\ letters in $\Theta$ with law $\mathsf{w}$ \\
    $\varphi(k)$ & match of $k\in\Z$, see \cref{par:inf_fk_map} \\
    $\rF$-excursion & word of the form $x(\varphi(0))\cdots x(0)$ with $x(0)=\rF$ \\
    $X(i,j)$ & portion $X(i)\cdots X(j)$ where $i,j\in\Z$, $i<j$ \\
    $\cS(i,j)$, $\cD(i,j)$ & burger count and discrepancy of $X(i,j)$, see \cref{sss: counts} \\
    $\cS_n$, $\cD_n$ & burger count and discrepancy at time $n\in\mathbb{Z}$, see \eqref{eq:def_calS_calD} \\
    $\cH(i,j)$, $\cC(i,j)$ & hamburger and cheeseburger counts in $X(i,j)$, see \cref{sss: counts} \\
    $\cH_n$, $\cC_n$ & hamburger and cheeseburger counts at time $n\in\mathbb{Z}$, see \eqref{eq:cH_cC_def} \\
    $\hleft_n$, $\cleft_n$ & exploration into the past or (lazy) reduced walks (\cref{sec: walk definitions}) \\
    $(Y(i)$, $i\in \mathbb{N})$ & see maximal excursion decomposition \eqref{eq:dec_max_exc} \\
    $\thleft$, $\tcleft$ & hitting times of $-1$ by $\hleft$, $\cleft$ \\
    $\tleft$ & minimum of these hitting times \\
    $\th$, $\tc$ & non-lazy versions of $\hleft$ and $\cleft$, see \cref{sec: walk definitions} \\
    $\tauh$, $\tauc$ & hitting times of $-1$ by $\th$ and $\tc$ \\
    $\mathscr{A}_{\th}$, $\mathscr{A}_{\tc}$ & alphabet corresponding to $h$-steps and $c$-steps, see \cref{sec: walk definitions} \\
    $\tG_k$, $\tN_k$ & times in the coupling \eqref{eq:(h,c)_geo_construc} \\
    $\tau_{\rF}$ & first time $k$ s.t.\ $\overline{X(1,k)}$ contains an $\rF$ symbol, see \cref{par:future walk} \\
    $P_{\rF}$ & 
    {law of the} word $X(1, \tau_{\rF})$, see \cref{par:future walk} \\
    $\mathcal{H}^*(P_{\rF})$, $\mathcal{C}^*(P_{\rF})$, $\mathcal{D}^*(P_{\rF})$ & burger counts and discrepancy of $P_{\rF}$ (without the $\rF$) -- \cref{par:future walk} \\
    $\hright_n$, $\cright_n$ & exploration into the future, see \eqref{eq:def_hcright} \\
    $\mathfrak{L}(0)$, $\mathfrak{c}(0)$, $\mathfrak{e}(0)$  & typical loop, cluster and envelope, defined in \cref{sec:loops-clusters-env-def} \\
    $\mathsf{sk}(e)$ & skeleton of the $\rF$-excursion $e$, see \eqref{eq:def_sk(e)} \\
    $Z(\tfrak,\boldsymbol{\ell},x,n)$ & weight \eqref{eq: def weight triangulation} of the loop-$O(n)$ model with edge parameter $x$ \\
    $F_\ell$ & loop-$O(n)$ partition function \eqref{eq: def partition triangulation} with boundary length $\ell$ \\
    $x_c(n)$ & critical edge parameter $x_c(n) = \frac{1}{\sqrt{8(n+2)}}$, see \eqref{eq:x_c_formula} \\
    $W(z)$ & resolvent function \eqref{eq: def resolvent} \\
    $\rho(z)$ & spectral density \eqref{def: spectral density} \\
    $\zeta$ & limiting $1$-stable random variable, with law  \eqref{eq:Lapl_zeta} \\
    $(\xi,\eta)$ & step distributions of the displacement and duration, see \cref{sec: joint laws} \\
    $\sigma(n)$ & time-change \eqref{eq:time_change_cal}, i.i.d.\ sum of $\eta$ variables \\
    $\sigma$ & Brownian hitting time, see \cref{lem:sc_duration_length} \\
    $\Nleft_n$ & number of reduced steps before time $n$, see \eqref{eq:def_N_red} \\
    $r(e)$ & excursion length of $e$, defined in \cref{sec:forward_step} \\
    $\mathbb{Q}$ & change of measure w.r.t.\ $r(e)$, see \eqref{eq:law_Q} \\
    $U$ & uniform random variable sampled under $\mathbb{Q}$, see \cref{lem: law of first F-symbol} \\
    $v_n$ & scaling factor for the discrepancy, $v_n=\sqrt{n}/\log(n)$ \\
    $\Nright_n$ & number of unmatched $\rF$ symbols at time $n$, see \eqref{eq:def_NF} \\
    $u_n$, $w_n$ & scaling factors for $\Nleft_n$ and $\Nright_n$ \\
    $\Gcal_0$ & $\sigma$-field generated by $(X(-k), k<i))$, see \cref{par:approx_markov}
  \end{longtable}
\end{flushleft}

\section{Background and preliminary observations}
\label{sec:prelim}

We provide some background and definitions for the FK-weighted planar map model. We also comment on the existing literature, although some of our preliminary observations and constructions seem to be new.

\subsection{Finite FK decorated random planar maps}

\subsubsection{Definition of the model} 
\label{sss:def_FK}

We start with some notation related to planar maps. Let $\mathcal{M}_k$ be the set of rooted planar maps  with $k>0$ edges, that is, proper embeddings on the sphere of finite connected graphs with $k$ edges and one distinguished, oriented \emph{root} edge,
defined up to continuous deformation. For $\mathfrak{m} \in \mathcal{M}_k$, let $V(\mathfrak{m})$, $E(\mathfrak{m})$ and $F(\mathfrak{m})$ be the set of vertices, edges and faces of $\mathfrak{m}$ respectively. We declare the \emph{root vertex} of $\mathfrak{m}$ to be the initial vertex of the oriented root edge, and the \emph{root face} to be the face of $\mathfrak{m}$ to the right of the oriented root edge.  
{The degree of a face $\mathfrak{f}\in F(\mathfrak{m})$ is the number of edges incident to this face, counting twice any cut edge lying strictly inside $\mathfrak{f}$.}

In this subsection we will describe the law of an FK planar map (of size $k$), which is a random pair $(\mathfrak{m}, \pfrak)$ where $\mathfrak{m}\in\mathcal{M}_k$ and $\pfrak$ is a subset of $E(\mathfrak{m})$ called open edges. To do this, we will first explain how to construct a loop-decorated triangulation $T(\mathfrak{m},\pfrak)$ from $(\mathfrak{m},\pfrak)$, {called the \textbf{Tutte map}. See \cref{fig:Tutte} for an illustration.} 

Given a pair $(\mathfrak{m}, \pfrak)$ as above, we let $\mathfrak{m}^{\dagger}$ be the \textbf{dual map} of $\mathfrak{m}$, so that each vertex ${v}^{\dagger} \in V(\mathfrak{m}^{\dagger})$ corresponds to a face $\mathfrak{f}\in F(\mathfrak{m})$ {and two dual vertices are connected by a dual edge if the faces are adjacent}. The root edge of $\mathfrak{m}^{\dagger}$ is, by convention, the oriented edge that crosses the root edge of $\mathfrak{m}$ from right to left, which means that the root vertex of $\mathfrak{m}^\dagger$ corresponds to the root face of $\mathfrak{m}$. The edge set $\pfrak$ induces a subset $\pfrak^{\dagger}$ of $E(\mathfrak{m}^{\dagger})$ by declaring that ${e}^{\dagger}\in \pfrak^{\dagger}$ if and only if the primal edge ${e}\notin \pfrak$.  We build a new planar map $Q(\mathfrak{m})$ with vertex set $V(\mathfrak{m})\cup V(\mathfrak{m}^{\dagger})$ by  connecting an edge between ${v}^{\dagger}$ and every vertex ${v}$ that is adjacent to $\mathfrak{f}$ in $\mathfrak{m}$. Note that for each edge ${e}\in E(\mathfrak{m})$ (and the corresponding ${e}^{\dagger}\in E(\mathfrak{m}^{\dagger})$), there are four edges connecting the endpoints of ${e}$ and ${e}^{\dagger}$. Hence $Q(\mathfrak{m})$ is a {rooted, planar \textbf{quadrangulation} (i.e.\ all faces have four incident edges)}, with the convention that the root edge points from the root vertex of $\mathfrak{m}^\dagger$ to the root vertex of $\mathfrak{m}$. Adding edges in $\pfrak$ and $\pfrak^{\dagger}$ to $Q(\mathfrak{m})$, we build a planar rooted \textbf{triangulation} $T(\mathfrak{m}, \pfrak)$, since every quadrangle is split into two triangles by exactly one edge in $\pfrak\cup \pfrak^{\dagger}$. These triangles will be referred to as primal (resp.\ dual) triangles if they are produced by edges in $\pfrak$ (resp.\ $\pfrak^{\dagger}$). For visualisation purposes we think of (primal) edges in $\pfrak$ as being coloured \emph{blue} and (dual) edges in $\pfrak^\dagger$ as being coloured \emph{red}. The root edge of $T(\mathfrak{m}, \pfrak)$ is the same as that of $Q(\mathfrak{m})$.

On $T(\mathfrak{m}, \pfrak)$, we can start from any triangle and draw a unique path that only crosses the triangles through edges in $Q(\mathfrak{m})$. This path will finally return to the starting triangle and turn out to be a loop. Repeating the procedure (starting from another arbitrary uncrossed triangle each time) until every triangle is crossed by a loop, we obtain a \textbf{fully packed ensemble of loops} $L(\mathfrak{m}, \pfrak)$ that separate the primal and the dual connected components {created by $(\pfrak,\pfrak^\dagger)$}. These loops are self-avoiding and mutually avoiding.

For $q\geq 0$, an \textbf{FK$(q)$ planar map with $k$ edges}, is a random configuration $(M_k, \Pmap_k)$ with $M_k\in \mathcal{M}_k$ and $\Pmap_k\subset E(M_k)$ sampled according to law $\mathbb{P}$ satisfying
\begin{equation}\label{def: FK sphere}
    \mathbb{P}((M_k, \Pmap_k) = (\mathfrak{m}, \pfrak)) 
    \, \propto \, \sqrt{q}^{\#L(\mathfrak{m}, \pfrak)}. 
\end{equation}
Notice that the weight on the right hand side does not depend on the choice of root edge in $M_k$ and thus, conditionally on the map, the oriented root edge is chosen uniformly at random. 
{Moreover, conditionally on the map $M_k$, the edge configuration $\Pmap_k$ is sampled as a self-dual FK$(q)$-percolation (also called random cluster) configuration on $M_k$, see e.g.\ \cite{duminil2013parafermionic}.}

\subsubsection{The Mullin--Bernardi--Sheffield bijection}
\label{sec:mullin-bernardi-sheffield}

In this section we recall the bijective frameworks of Mullin \cite{mullin67}, Bernardi \cite{bernardi08} and Sheffield \cite{SheffieldScott2016QGAI}, which relate random planar maps decorated by edge sets to words in a certain alphabet $\Theta$ (see below), which can be viewed as inventory accumulation processes.

\paragraph{Inventory accumulation model.} Let $\Theta = \{\rc, \rh, \rC, \rH, \rF\}$, called the alphabet of letters. A word $w = \theta_1 \theta_2\ldots \theta_k$ ($\theta_i\in\Theta$) is a concatenation of letters. Denote by $\emptyset$ the empty word. In the sequel, we will call the letter $\rc$ (respectively $\rh$) a \emph{cheeseburger} (respectively \emph{hamburger}), $\rC$ (respectively $\rH$) a \emph{cheeseburger order} (respectively \emph{hamburger order}) and $\rF$ a \emph{flexible/freshest order}. A word $w$ is then viewed as an (ordered) sequence of events {in a restaurant selling either hamburgers or cheeseburgers. 
For example, the word $w=\rh\rc\rF\rC$ corresponds to the restaurant producing first a hamburger, then a cheeseburger, and finally two customers ordering a ``freshest'' burger and then a cheeseburger. It is conceptually useful to view this string of letters as describing, from left to right, the evolution of a \emph{stack} of burgers: new burgers are pushed onto the stack, and (fulfilled) orders remove a burger from it.}
Thus, the model can be seen as an \textbf{inventory model} at a LIFO (last-in, first-out) retailer with two products and three types of orders.

The reduced word $\overline{w}$ is the word that results after matching burgers and orders in $w$, following the \textbf{reduction rule} that $\rc\rC = \rh\rH = \rc\rF = \rh\rF = \emptyset$ and $\rc\rH = \rH\rc$, $\rh\rC = \rC\rh$.\footnote{More rigorously we should say that the reduced word is an equivalence class of words with equivalence given by these relations. We will abuse the notation and use a candidate of the class instead.} 
The rule is interpreted {by fulfilling the orders -- whenever possible -- in the way they were placed (i.e.\ from left to right)}, as follows. A hamburger order (cheeseburger order) is matched with the first {remaining} hamburger (cheeseburger) discovered when tracing back along the sequence (reading the word from right to left), while a freshest order is matched with the first {remaining} burger (no restriction on type) discovered when tracing back. 
{In other words, $\rH$ and $\rC$ orders consume the topmost burger of the corresponding type $\rh$ and $\rc$ in the current stack (if any), while a flexible order $\rF$ consumes the topmost (i.e.\ freshest) burger in the current stack, regardless of its type (if any).}
{For instance, the word $w=\rh\rc\rH\rF\rc\rH$ reduces to $\overline{w}=\rc \rH = \rH\rc$ and we have, in particular, that: the first customer (third letter) orders and gets a hamburger, the next customer orders fresh and gets a cheeseburger, but the last customer's hamburger order cannot be fulfilled.}

We will now describe the bijection between inventory accumulation words and planar maps decorated {with a percolation configuration}. Roughly speaking, given $\mathfrak{m}\in \mathcal{M}_k$ and $\pfrak\subset E(\mathfrak{m})$, the bijection is a fixed way to explore the triangles in the Tutte map $T(\mathfrak{m}, \pfrak)$ and assign each triangle a letter from the alphabet $\Theta$; that is, to produce a word $w$. In the rest of the paper, we use the convention that $\rh, \rH$ ($\rc, \rC$) corresponds to primal triangles (dual triangles).

\begin{figure*}[t]
    \medskip
    \centering
    \begin{subfigure}[t]{0.45\textwidth}
        \centering
        \includegraphics[page=1,width=0.9\textwidth]{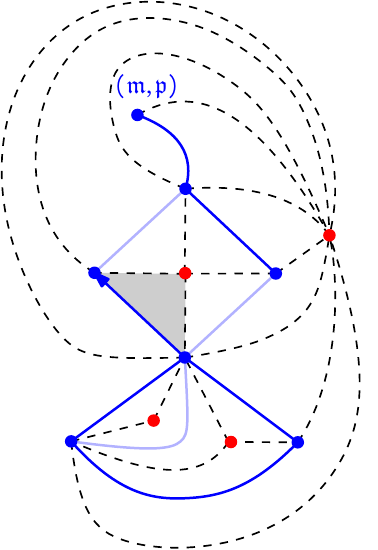}
        \caption{}
    \end{subfigure}%
    ~ 
    \begin{subfigure}[t]{0.45\textwidth}
        \centering
        \includegraphics[page=2,width=0.9\textwidth]{Figures/example_sheffield.pdf}
        \caption{}
    \end{subfigure}
    \caption{Illustration of the Tutte map. \textbf{(a)} A planar map $\mathfrak{m}$ together with an FK percolation configuration $\pfrak$ (blue), with an oriented root edge. Closed edges (i.e.\ edges in $\mathfrak{m}\setminus \pfrak$) are in  pale blue. We draw the dual vertices in red, each corresponding to a face of $\mathfrak{m}$, and connect each dual vertex to its adjacent primal vertices (dashed). The root triangle is the triangle to the right of the root edge (grey). \textbf{(b)} We draw the dual edges of $\pfrak^\dagger$ (red) between dual vertices. The Tutte map $T(\mathfrak{m},\pfrak)$ is the triangulation that arises from considering all the edges (blue, red and dashed). The interface between primal and dual components corresponds to loops (purple).
    \label{fig:Tutte}
    }
\end{figure*}

\paragraph{Map-to-word.} \label{para:mapword}
We begin by discussing, given $(\mathfrak{m},\pfrak)$, how to explore (i.e.\ give an order to) the triangles in $T(\mathfrak{m}, \pfrak)$, where $T(\mathfrak{m}, \pfrak)$ is constructed as described in \cref{sss:def_FK}. The construction can be seen from \cref{fig:Tutte,fig:sheffield_exploration}. For the loops in $L(\mathfrak{m}, \pfrak)$, there is a unique loop $\looop_0$ that crosses the root edge. We explore $\looop_0$ in the direction that crosses the oriented root edge from the left to the right. This gives an order for the triangles visited by $\looop_0$. To explore triangles crossed by other loops in $L(\mathfrak{m},\pfrak)$, we use a recursive argument. Note that each triangle has a companion triangle with which it forms a quadrangle in $Q(\mathfrak{m}, \pfrak)$. We look for the \emph{last} triangle $t$ visited by $\looop_0$ whose companion $t'$ is not visited by $\looop_0$, but by another loop, say $\looop_1$. For the quadrangle consisting of $t$ and $t'$, we remove the diagonal edge and connect the other diagonal instead {(this has the effect that the corresponding two primal or dual triangles then become triangles of the opposite type)}.
{We will sometimes refer to the new diagonals as \textbf{fictional edges}.}
Then the two loops $\looop_0$ and $\looop_1$ form one larger loop, from which we continue the procedure until all the loops in $L(\mathfrak{m}, \pfrak)$ are combined into one large loop $\looop$. 
Such a modification will change the structure of $T(\mathfrak{m}, \pfrak)$, and each flipped diagonal will be recorded in the corresponding word (as constructed below) by an $\rF$ symbol. 

We now explain how to construct the word corresponding to $(\mathfrak{m},\pfrak)$. The ultimate loop $\looop$ from the previous paragraph visits the triangles in a certain order. Moreover, each quadrangle in $Q(\mathfrak{m})$ consists of two companion triangles: to the first one visited by $\looop$ we associate a symbol {$\rh$} (for primal triangles) or {$\rc$} (for dual triangles), and to the second one we associate a symbol $\rC$ or $\rH$ in the same manner. We therefore obtain an intermediary word consisting of letters in $\{\rc, \rh, \rC, \rH\}$ by writing them in the order assigned by $\looop$. However, recall that when combining the loops, some quadrangles were modified by reversing the diagonal edges. These quadrangles corresponds to matches of $\rc, \rC$ or $\rh, \rH$ in the intermediary word. We therefore change the letters $\rC$ and $\rH$ in such situations to $\rF$, resulting in a word $w$ made of letters in $\Theta$ as the end product. 

To summarise, we have explained how, starting with $\mathfrak{m}\in \mathcal{M}_k$ and $\pfrak\subset E(\mathfrak{m})$, we can produce a word $w$ in the alphabet $\Theta$. Before discussing the inverse of this operation, we make some remarks on the word $w$. The length {(i.e.\ number of letters)} of $w$ is $2k$ since there are $2k$ triangles. 
{It can be seen \cite[Section 4.1]{SheffieldScott2016QGAI} that the matches of burgers and orders correspond to the completion of a quadrangle, in the sense that two symbols corresponding to companion triangles are actually matched under the aforementioned reduction rule.}\footnote{In particular, we claim that symbols corresponding to fictional triangles are properly matched in $w$. Visually, this means that every burger produced inside $\looop_1$ (say) will be consumed inside $\looop_1$. For this property to hold, it is important that we wait for the \emph{last} pair of triangles $(t,t')$ before entering $\looop_1$ in our construction of the space-filling loop $\looop$.}
As a consequence, the reduced word $\overline{w}$ of $w$ is equal to $\emptyset$. 
By construction, the number of loops $\#L(\mathfrak{m}, \pfrak)$ is the number of $\rF$ letters in $w$ plus 1.

\begin{figure}[t]
  \bigskip
  \begin{center}
    \includegraphics[page=3,scale=0.9]{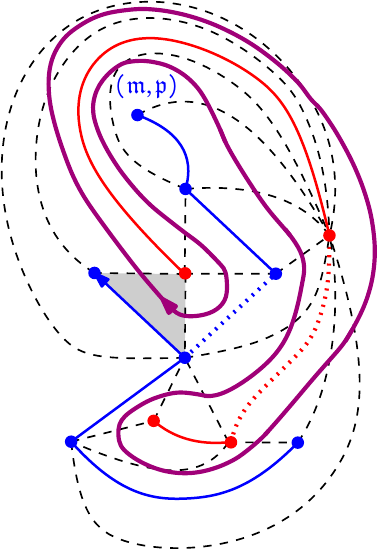}
  \end{center}
  \caption{The Mullin--Bernardi--Sheffield bijection, applied to the map in \cref{fig:Tutte}. We construct the space-filling exploration (purple) by starting from the loop crossing the root triangle and opening up the other connected components -- the corresponding ``fictional'' edges that were flipped along the procedure are dotted. The precise word that corresponds to the pair $(\mathfrak{m},\pfrak)$ is $w = \rh \rc \rH \rh \rh \rc \rH \rc \rH \rC \rF \rh \rh \rh \rH \rC \rH \rF$.} 
  \label{fig:sheffield_exploration}
\end{figure}

\paragraph{Word-to-map.} With these observations in hand, we now discuss the inverse operation. That is, how to construct a planar map $\mathfrak{m}\in \mathcal{M}_k$ and a subset of edges $\pfrak\subset E(\mathfrak{m})$, starting from a word $w$ of length $2k$ with $\overline{w} = \emptyset$.

The first step is to construct a triangulation made up of primal and dual triangles (which will be equal to $T(\mathfrak{m},\pfrak)$) together with a loop that visits all the triangles. To do this, we begin by changing every letter $\rF$ in $w$ to $\rC$ (respectively $\rH$) if its match is $\rc$ (respectively $\rh$). Given this new word, we start with a root edge {(oriented from a dual to a primal vertex)} and add triangles consecutively, together with a path passing through the triangles, using the following rule. 
{When we have a letter $\rh$ (respectively $\rc$), we glue a primal (respectively dual) triangle to the edge we just crossed, in the only way that ensures that the path keeps primal edges to its left and dual edges to its right. When we have a letter $\rH$ (respectively $\rC$), we do the same but identify the primal (respectively dual) edge with that of the triangle corresponding to its matching $\rh$ (respectively $\rc$). This identification is made so that the two triangles are companion triangles and thus corresponds to completing a quadrangle.}
The condition $\overline{w} = \emptyset$ ensures that after we add the last triangle to the configuration, the path exits through an edge that connects the endpoints of the root edge. Hence we could glue these two edges together and make the whole map into a rooted planar triangulation (each triangle coming with a \emph{primal} or \emph{dual} label) with one space-filling loop crossing every triangle. Since the triangles are matched in pairs, we also have an associated rooted planar quadrangulation $Q$ {by simply removing the primal and dual edges}. However, recall that we replaced $\rF$ with $\rC$ or $\rH$ at the beginning of this paragraph. Hence, after carrying out the construction above, we need an additional step. For each quadrangle that corresponds to $\rc\rF$ or $\rh\rF$, we ``flip'' (change) its diagonal edge. This results in a new triangulation $T$ (again made up of primal and dual triangles). 

Finally, from here we can {construct} the map $\mathfrak{m}$ and the subset of edges $\pfrak$. From the quadrangulation $Q$, we can recover $\mathfrak{m}$ directly: since $Q$ is bipartite, $V(Q)$ can be divided into two subsets such that each edge has one vertex in each subset as an endpoint. We declare the subset containing the end vertex of the root edge of $Q$ to be the (primal) vertex set $V(\mathfrak{m})$ of $\mathfrak{m}$. We connect two vertices in $V(\mathfrak{m})$ by an edge if this edge is a diagonal of some quadrangle in $Q$ (although not necessarily an edge of $T$), which produces the planar map $\mathfrak{m}$. Declaring the edges of $\mathfrak{m}$ that lie in $T$ to be the edge set $\pfrak$, we produce the pair $(\mathfrak{m},\pfrak)$, and it is straightforward to check that $T = T(\mathfrak{m}, \pfrak)$.  

We have now described how to produce $w$ from $(\mathfrak{m},\pfrak)$ and vice versa. It can moreover be checked that these two operations are inverse to one another; thus we have a bijection between pairs $\mathfrak{m}\in \mathcal{M}_k$ together with $\pfrak\subset E(\mathfrak{m})$, and words $w$ in $\Theta$ of length $2k$ satisfying $\overline{w}=\emptyset$. 

\subsubsection{Random inventory accumulation}

Through the bijection discussed above, a random decorated FK planar map with $k$ edges, i.e.,~a random pair ${M_k}\in \mathcal{M}_k$ and ${\Pmap_k}\subset E(M_k)$, corresponds to a random word $W$ in $\Theta$ of length $2k$ satisfying $\overline{W}=\emptyset$. The law of this random word can be described as follows. 

First, let $p=p(q) {\in [0,1)}$ satisfy
\begin{equation}\label{eq: p}
\sqrt{q}=\frac{2p}{1-p}
\end{equation}
and given $p$, assign the following weights ${\mathsf{w}}(\theta)$ to letters $\theta\in \Theta$:
\begin{equation}\label{eq: symbol weights}
    \mathsf{w}(\rc) = \mathsf{w}(\rh) = \frac14, \quad \mathsf{w}(\rC) = \mathsf{w}(\rH) = \frac{1-p}{4}, \quad \mathsf{w}(\rF) = \frac{p}2. 
\end{equation}
We will also use the notation $\mathsf{w}(w)=\mathsf{w}(\theta_1)\dots \mathsf{w}(\theta_n)$ for a word $w=\theta_1\dots \theta_n$ in the alphabet $\Theta$.
{Note that the weights \eqref{eq: symbol weights} also make sense for $p=1$. The main results in this paper concern the case $q=4$, which means $p=1/2$.}

Now consider a random word $W$ with $2k$ letters, each sampled independently with distribution given by \eqref{eq: symbol weights}, \emph{but conditioned on}  $\overline{W} = \emptyset$. Then the probability of choosing a given word $w$ with $\overline{w}=\emptyset$ satisfies
\begin{equation}\label{eq: word law}
\mathbb{P}(W = w\, \left| \, |W| = 2k, \overline{W} = \emptyset \right.)\, \propto \,  \Big(\frac{1}{4}\Big)^{\#\rc + \#\rh} \Big(\frac{1-p}{4}\Big)^{\#\rC+\#\rH} \Big(\frac{p}{2}\Big)^{\#\rF} = \Big(\frac{1-p}{16}\Big)^{k} \Big(\frac{2p}{1-p}\Big)^{\#\rF},
\end{equation}
where $\#\rc, \#\rh, \#\rC, \#\rH, \#\rF $ are the number of symbols of each type in $w$. 
To see the equality in \eqref{eq: word law}, one simply observes that $\#\rc + \#\rh = \#\rC + \#\rH + \#\rF = k$ because of the condition that $\overline{w} = \emptyset$. 

Since the number of edges in $\mathfrak{m}$ is {fixed and} equal to $k$ and the number of loops in $L(\mathfrak{m},\pfrak)$ is equal to the number of $\rF$ symbols plus one, this law on $W$ is the law of a word corresponding to $(\mathfrak{m},\pfrak)$ sampled from \eqref{def: FK sphere} with {$p$ and $q$ related as in \eqref{eq: symbol weights}}.

\subsection{Infinite FK decorated random planar maps and scaling limits}
\label{par:inf_fk_map}
\subsubsection{Definition of the model} 
\label{sss:infinite_fk}
In the preceding section, we have seen that a (random) FK planar map $(M_k,\Pmap_k)$ with $k$ edges and parameter $q$ corresponds to a word $W$ sampled according to \eqref{eq: word law}.  When $k\to\infty$, it is then proved in \cite{ChenLinxiao2017Bpot} that $(M_k, \Pmap_k)$ converges in distribution in the local topology to an object $(M_{\infty}, \Pmap_{\infty})$, called the \textbf{infinite FK$(q)$ planar map}. {More precisely, the convergence takes place on the completion of the space of finite rooted planar maps with a distinguished subgraph, equipped with the local metric $d_{\mathrm{loc}}((\mathfrak{m},\pfrak),(\mathfrak{m}',\pfrak'))=1/\sup\{R: B_R((\mathfrak{m},\pfrak))=B_R((\mathfrak{m}',\pfrak'))\}$ where $B_R(\mathfrak{m},\pfrak)$ is everything in $(\mathfrak{m},\pfrak)$ at graph distance less than $R$ from the root edge}. 

Informally speaking, the law of the infinite FK$(q)$ planar map can be sampled from by taking a bi-infinite word $W = \ldots X(-1) X(0) X(1) \ldots$ with $(X(i))_{i\in\mathbb{Z}}$ i.i.d.\,copies distributed according to \eqref{eq: symbol weights} {(where $p$ and $q$ are related through \eqref{eq: symbol weights})}, and then applying the word-to-map direction of Sheffield's bijection. Let us now explain more precisely what this means.

In such a bi-infinite word $W$, it is the case that with probability one, there is a \textbf{match} for every symbol (see \cite[Proposition 2.2]{SheffieldScott2016QGAI}). That is, any letter $X(i)$ in $W$ corresponding to a burger $(\rc$ or $\rh$) will be consumed by some order ($\rC,\rH,\rF)$ in the future (i.e., after applying the reduction rules to $W$) and any order will be fulfilled by some burger in the past. The unique match of $X(i)$ is denoted by $X(\varphi(i))$. 

{Now let us explain how, given the bi-infinite word $W$, we can construct a sequence of growing neighbourhoods of the root triangle in the corresponding infinite FK planar map. The root triangle will correspond to the time $0$ (or the letter $X(0)$). We call an {\bf $\rF$-excursion} a word {$e$} of the form $\rh \cdots \rF$ (type $\rh$) or $\rc \cdots \rF$ (type $\rc$) where the final $\rF$ has the first letter as its match. It is not hard to see that there will be an infinite sequence of nested $\rF$-excursions containing $X(0)$, and these will correspond in the map to a particular growing sequence of neighbourhoods of the origin.}

Namely, take such an $\rF$-excursion {$e$} and assume without loss of generality that it is of type $\rh$, so ${e}=\rh \cdots X(0) \cdots \rF$. It follows from the definition of an $\rF$-excursion that $\overline{{e}}=\rC\cdots \rC$ is made up of only $\rC$ symbols {(if any)}. We consider the word $e'$ defined by removing all these remaining $\rC$'s, and apply the word-to-map direction of the Mullin--Bernardi--Sheffield bijection to $e'$.

Repeating this procedure for increasing $\rF$-excursions containing $X(0)$ builds a sequence of increasing envelopes containing the root triangle. This defines the ball of any given radius (with respect to graph distance) centred at the root triangle. To define the infinite FK map itself (and not the associated infinite triangulation) we proceed as in the finite volume case: to define the map {$M_\infty$} we keep only vertices incident to primal/blue edges and join
two vertices by an edge if and only if they are opposites on some quadrangle; the marked subset of edges {$\Pmap_\infty$} is made up of exactly the primal/blue edges in the infinite triangulation.

\subsubsection{Burger counts and scaling limits} \label{sss: counts}
For a fixed word $w$, it is possible that some of the letters remain unmatched, that is, some burgers are not consumed and/or some orders are not fulfilled. For any word $w$, if we write $\# \theta({w})$ for the number of letters of type $\theta$ in ${w}$, the total burger count $\cS(w)$ is defined to be $\#\rh(w)+\#\rc(w)-\#\rH(w)-\#\rC(w)-\#\rF(w)$.

In \cite[Proposition 2.2]{SheffieldScott2016QGAI}, it is proved that for every $p\in [0, 1]$, for the bi-infinite word $\ldots X(-1)X(0)X(1)\ldots$ with $(X(i))_{i\in\mathbb{Z}}$ i.i.d.~copies of $X(1)$ with the law in \eqref{eq: symbol weights}, almost surely every letter has a match. For any portion $W$ of the word we define the burger discrepancy $\cD(W)$ to be equal to $\#\rh(W)- \#\rC(W)-\#\rc(W)+\#\rH(W)$ plus one for every $\rF$ symbol in $W$ whose match in the bi-infinite word is a $\rc$, minus one for every $\rF$ symbol in $W$ whose match in the bi-infinite word is an $\rh$. Note that this match may not be in $W$, so the definition $\cD(W)$ only makes sense for a portion $W$ of $ \ldots X(-1)X(0)X(1)\ldots$. 

We denote the portion of the bi-infinite word between symbols $i$ and $j\ge i$ as $X(i,j)=X(i)X(i+1)\dots X(j-1)X(j)$, {and extend the notation to the semi-infinite words $X(-\infty,j)$ and $X(i,+\infty)$ in the obvious way.} Then we can define the total \textbf{burger count} $\cS(i,j)=\cS(X(i,j))$, and also the \textbf{burger discrepancy} $\cD(i,j)=\cD(X(i,j))$. 
We also use the notation $\cH(i,j)$ for $\#\rh(X(i,j))- \#\rH(X(i,j))$ minus one for every $\rF$ symbol in $X(i,j)$ whose match in $X(-\infty,j)$ is an $\rh$, and $\cC(i,j)$ for $\#\rc(X(i,j))- \#\rC(X(i,j))$ minus one for every $\rF$ symbol in $X(i,j)$ whose match in $X(-\infty,j)$ is an $\rc$. Then by definition we have $\cS(i,j)=\cH(i,j)+\cC(i,j)$ and $\cD(i,j)=\cH(i,j)-\cC(i,j)$ for each $i,j$.

Sheffield \cite{SheffieldScott2016QGAI} proved the following scaling limit result for the process {defined by $(\mathcal{S}_0, \mathcal{D}_0)=(0,0)$ and}
\begin{equation}
\label{eq:def_calS_calD}
(\mathcal{S}_n, \mathcal{D}_n):=(\cS(1,n),\cD(1,n))
\quad
\text{and } 
\quad
(\mathcal{S}_{-n}, \mathcal{D}_{-n}):=(-\cS(-n,-1),-\cD(-n,-1)),
\quad n\geq 0,
\end{equation}
when $q<4$, equivalently $p<1/2$. 
For future reference, we also set
\begin{equation} \label{eq:cH_cC_def}
\cH_n=\mathcal{H}({1},n), \, \cC_n=\mathcal{C}({1},n) \text{ for } n\ge 0; \text{ and } \cH_n= - \mathcal{H}(-n,{-1}), \, \cC_n=-\mathcal{C}(-n,{-1}) \text{ for } n\le 0. 
\end{equation}
Observe that, because burger productions and orders are symmetric in \eqref{eq: symbol weights}, $\cS$ is always a simple symmetric random walk (regardless of $p$) but, due to the contribution of $\rF$ symbols, for $q>0$ the discrepancy process $\cD$ is \emph{not} a random walk {(and in fact not even Markov)}. 

\begin{Thm}[Sheffield, \cite{SheffieldScott2016QGAI}]\label{thm: sheffield's convergence}
Let $p\in [0,1]$ and $\alpha = \max\{1-2p, 0\}$. Let $B^1_t$ and $B^2_t$ be independent standard one-dimensional two-sided Brownian motions. Then
\begin{equation}\label{eq: sheffield's convergence}
\bigg( \frac{\mathcal{S}_{\lfloor nt\rfloor}}{\sqrt{n}}, \frac{\mathcal{D}_{\lfloor nt \rfloor}}{\sqrt{n}}\bigg)_{t\in \R} \stackrel{\textnormal{d}}{\longrightarrow} (B^1_t, B^2_{\alpha t})_{t\in \R} 
\end{equation}
in the space of c\`{a}dl\`{a}g functions with the local Skorohod J$_1$ topology. 
\end{Thm}

{Observe that, when $p\geq 1/2$, the discrepancy goes to $0$ at the scale $\sqrt{n}$. The case when $p>1/2$ is supposedly related to supercritical Liouville gravity, as it has been proved that the planar maps are tree-like \cite{feng2023triviality}.}
Our main result, \cref{thm: main intro}, establishes a non-trivial limit in the \emph{critical} case $(q=4,p=1/2)$, when $\cD_n$ is rescaled in the correct way.

\subsubsection{Auxilliary walks: past and future}\label{sec: walk definitions}

{As already pointed out, the discrepancy $\cD$ is not a Markov process. A key step in the proof of our scaling limit result (\cref{thm: main intro}) is to replace it with two explorations that \emph{do} have a nice Markov structure, and which we will show provide effective control of 
$\cD$.}

\paragraph{Exploration into the past.} \label{par:red_walk}
To analyse the word $X(\varphi(0))\cdots X(0)$, we introduce a pair of random walks $(\hleft_n, \cleft_n, n\geq 0)$ called the \textbf{reduced walk} (this terminology first appeared in \cite{berestycki2017critical}), and sometimes referred to as the \textbf{exploration into the past}.\footnote{These walks are denoted by $h$ and $c$ in \cite{berestycki2017critical}, but we use the left arrow to distinguish them from the exploration into the future, that we will define in the next paragraph and use in \cref{sec:explo_future}.}

Recall the notion of $\rF$-excursion introduced in \cref{sss:infinite_fk}. An $\rF$-excursion is said to be \textbf{maximal} if it is not contained in any other $\rF$-excursion inside $X(-\infty, 0)$. Given this definition, we can decompose $X(-\infty, 0)$ over its maximal $\rF$-excursions, writing it as a sequence 
\begin{equation}\label{eq:dec_max_exc}
X(-\infty,0) = \cdots Y(2)Y(1)X(0),
\end{equation}
where for each $i\ge 1$, $Y(i)$ is either a single letter among $\rh$, $\rc$, $\rH$ or $\rC$, or  a maximal $\rF$-excursion. We stress that this decomposition is unique.

We can now define $(\hleft_n,\cleft_n,n\ge 0)$. We first set $\hleft_0=0$ and $\cleft_0=0$. Then we define $(\hleft_n,\cleft_n)$ recursively for $n\ge 1$ as follows:
\begin{itemize}
    \item if $Y(n) = \rH$ (resp.\ $\rC$) then $(\hleft_n,\cleft_n) = (\hleft_{n-1}+1,\cleft_{n-1})$ (resp.\ $(\hleft_n,\cleft_n) = (\hleft_{n-1},\cleft_{n-1}+1)$);
    \item if $Y(n) = \rh$ (resp.\ $\rc$)  then $(\hleft_n,\cleft_n) = (\hleft_{n-1}-1,\cleft_{n-1})$ (resp.\ $(\hleft_n,\cleft_n) = (\hleft_{n-1},\cleft_{n-1}-1)$); 
    \item if $Y(n)$ is an $\rF$ excursion $E$ of type $\rc$ (resp.\ type $\rh$), then $(\hleft_n,\cleft_n) = (\hleft_{n-1}+|\overline{E}|,\cleft_{n-1})$ (resp.\ $(\hleft_n,\cleft_n) = (\hleft_{n-1},\cleft_{n-1}+|\overline{E}|)$) where $|\overline{E}|$ is the length of the reduced word.\footnote{Note that if $E$ is an $\rF$-excursion of type $\rc$, the reduced word $\overline{E}$ can only consist of $\rH$ orders. This explains why we consider these steps as increments for $\hleft$.} 
\end{itemize}
We call any such step an \textbf{$h$-step} (resp.\ \textbf{$c$-step}).

The process
$(\hleft_n,\cleft_n)$ is therefore equal to a time change of the process $(\cH_{-m}, \cC_{-m})_{m\ge 0}$ defined in \eqref{eq:cH_cC_def}.
In this time change, all intervals of time between a burger -- $\rc$ or $\rh$ -- being produced and ordered as ``fresh'' -- matched to an $\rF$ -- are skipped. The advantage of working with $(\hleft_n,\cleft_n)$ is that, by definition of the symbols in $X$, it is really a (Markovian) random walk.

Let $\tleft=\min\{\thleft,\tcleft\}$ where $\thleft$ is the first time that $\hleft$  hits $-1$ and $\tcleft$ is the first time that $\cleft$ hits $-1$. {Given the above definitions, it will also be convenient to introduce the set $\mathscr{A}_{\rh}$ (resp.\ $\mathscr{A}_{\rc}$) of all words made of $\rh$, $\rH$ and $\rF$-excursions of type $\rc$ (resp.\ $\rc$, $\rC$ and $\rF$-excursions of type $\rh$). Then the decomposition of $X$ into maximal $\rF$-excursions gives a unique way of writing $X(-\infty,0)$ as a concatenation of words in $\mathscr{A}_{\rc}$ and $\mathscr{A}_{\rh}$, and $h$-steps correspond to words in $\mathscr{A}_{\rh}$, while $c$-steps correspond to words in $\mathscr{A}_{\rc}$.

It is clear from the definition that the increments of the reduced walk are always one-dimensional, so that it will often make sense to work with the following one-dimensional versions of the walk. Given $(\hleft,\cleft)$ we define $(\th_k)_{k\ge 0}$  (respectively $(\tc_k)_{k\ge 0}$) to be the time change of $\hleft$ (respectively $\cleft$) where we skip past times corresponding to $c$-steps (respectively $h$-steps). These are random walks which are independent and have the same step distribution.
Let $\tauh$ (resp.\ $\tauc$) be the hitting times of $-1$ by $\th$ (resp.\ $\tc$).

We note for future reference that one can construct the ``lazy'' reduced walk $(\hleft, \cleft)$ from the ``non-lazy'' walk $(\th,\tc)$ as follows.
Let $\tG_i$ for ${i\geq 1}$ be the number of $c$-steps between the $(i-1)$th $h$-step and the $i$th $h$-step. Then $(\tG_i)_{i\ge 1}$ is a sequence of geometric random variables with parameter $1/2$ (i.e.\ $\mathbb{P}(\tG_i=j) = 2^{-j-1}$ for all $j\ge 0$), independent of $(\th,\tc)$. Let $\tN_{k} = \sum^k_{i=1} \tG_i$ with the convention that $\tN_0 = 0$; it follows that
\begin{equation} \label{eq:(h,c)_geo_construc}
(\hleft_n, \cleft_n) = (\th_{k}, \tc_{n-k}), \quad \text{ for } k\geq 0  \text{ such that } \tN_{k} + k\leq n \leq \tN_{k+1} + k. 
\end{equation}
\noindent In words, $\hleft$ stays constant (and $\cleft$ moves according to $\tc$) in each interval $n\in \{\tN_k+k+1, \dots, \tN_{k+1}+k\}$, and moves according to $\th$ at times of the form $\tN_k+k$. Note that if $\tG_{k+1}=0$ (which happens with probability $1/2$), the interval $\{\tN_k+k+1, \dots, \tN_{k+1}+k\}$ is empty.

\paragraph{Exploration into the future.}\label{par:future walk}
{The second exploration that will play a crucial role in the sequel will be looking into the future, i.e.\ to the right of $0$. As far as we know, this exploration has not appeared in the literature before.}
Let $\tau_{\rF}$ be the first time $k$ such that the reduced word $\overline{X(1,k)}$ contains an $\rF$ symbol. In other words, $\tau_{\rF}$ is the first time that we see an $\rF$ symbol whose match lie at or to the left of position $0$. From \cite[{Section 3.3}]{SheffieldScott2016QGAI}, we know that $\tau_{\rF}<\infty$ almost surely. Denote by $P_{\rF} = X(1,\tau_{\rF})$ the corresponding random word. 

In a similar manner, for $n\ge 0$ let $\tau_{\rF}^{n}$ be the first time $k$ such that $\overline{X(1,k)}$ has $n$ symbols of type ${\rF}$. Set $\tau^{0}_{\rF} = 0$ by convention. Alternatively, it is useful to see the times $\tau_{\rF}^{n}$ as defined recursively by
\begin{equation} \label{eq:tau_F^n}
\tau_{\rF}^{n+1}
=
\inf\Big\{k>\tau_{\rF}^{n}, \; \overline{X(\tau^{n}_{\rF}+1,k)} \text{ contains an $\rF$}\Big\}.
\end{equation}
This shows that the words $P_{\rF}^n = X(\tau^{n-1}_{\rF}+1,\tau^n_{\rF})$, $n \ge 1$, are i.i.d.\ copies of $P_{\rF}$, {and in particular} $(\tau^{n}_{\rF} - \tau^{n-1}_{\rF})$ are i.i.d.\ copies of $\tau_{\rF} = \tau_{\rF}^{1}$.
This provides a ``random walk exploration'' of the word $X(1,+\infty)$, which is obtained by concatenating the $P_{\rF}^i$ and plays the same role as that of the reduced walk in the previous paragraph. We shall refer to this i.i.d.\ concatenation as the \textbf{exploration into the future}.

{We will be interested in the burger count and discrepancy along this exploration. An important remark is that, by definition, the reduced word $\overline{P}_{\rF}$ can only consist of a string of $\rH$ or $\rC$ orders, followed by an $\rF$.
Let $\mathcal{C}^*(P_{\rF})=\cC(1,\tau_\rF-1)$ and $\mathcal{H}^*(P_{\rF})=\cH(1,\tau_\rF-1)$ be the number of cheeseburger orders and hamburger orders in $\overline{P}_{\rF}$, where the final $\rF$ symbol is \emph{not} counted.\footnote{The reason why we are counting out the final $\rF$ symbols is because they would introduce some correlation between the steps of the concatenation. We use the notation $*$ to keep in mind that the $\rF$ symbol was taken out.} 
We also denote by $\mathcal{D}^*(P_{\rF}) = \mathcal{H}^*(P_{\rF}) - \mathcal{C}^*(P_{\rF})$ the discrepancy of $\overline{P}_{\rF}$ (again, taking away the final $\rF$). As for the exploration into the past, we introduce 
\begin{equation} \label{eq:def_hcright}
\hright_n := \sum_{i=1}^n \mathcal{H}^*(P_{\rF}^{i})
\quad 
\text{and}
\quad
\cright_n := \sum_{i=1}^n \mathcal{C}^*(P_{\rF}^{i}),
\end{equation}
where the $P_{\rF}^i$, $i\geq 1$, are the i.i.d.\ words defined below \eqref{eq:tau_F^n}. Again, these processes are true random walks.}
We will derive the scaling limit of these walks in \cref{sec:scal_lim_forw}.

\subsection{Loops, clusters and envelopes}
\label{sec:loops-clusters-env}

\subsubsection{Definitions}
\label{sec:loops-clusters-env-def}
\paragraph{Loops.} {Recall {from \cref{par:inf_fk_map}} that the infinite FK planar map model is encoded by a bi-infinite word formed of symbols in $\{\rc,\rh,\rC,\rH,\rF\}$, where each symbol $\rF$ corresponds to a loop forming the interface between a primal and a dual connected component of the coloured (loop-decorated) triangulation associated to the map. We define a notion of \textbf{typical loop} by conditioning on the event $X(0)=\rF$ and looking at the loop $\mathfrak{L}(0)$ corresponding to that $\rF$ symbol. The loop is \emph{typical} in the sense that a loop picked uniformly at random from a  FK planar map of size $k$ converges in law to $\mathfrak{L}(0)$ as $k\to \infty$ (see \cite[Proposition 3.5]{berestycki2017critical}).
There are also several equivalent definitions of a typical loop, as shown in \cite[Proposition 3.4]{berestycki2017critical}. For example, one could instead consider the first $\rF$-excursion to the left of $0$ (without any conditioning). The length of $\mathfrak{L}(0)$, denoted by $|\mathfrak{L}(0)|$, is defined as the number of triangles it crosses.

In \cite{berestycki2017critical} and \cite{GwynneEwain2019Slft}, it is shown that for $q<4$ and  
$  q=(2-4\cos^2(\frac{\pi}{4p_0}))^2$, 
\begin{equation}\label{eq: loop exponent q<4}
\mathbb{P}(|\mathfrak{L}(0)| > k) = k^{-\frac{1}{p_0} + o(1)}
\quad \text{as } k\to \infty.
\end{equation}
In \cite{GwynneEwain2019Slft}, a slightly stronger result involving the regularity of the tails is obtained. In \cref{thm:main_exponents} we prove the analogous result for $q=4$ ($p_0=1/2$) and also identify the slowly varying correction exactly.

\paragraph{Clusters.} The loop $\mathfrak{L}(0)$ splits the triangulation into a ring of triangles (those which intersect it), together with one infinite and one finite connected component. We define the \textbf{typical cluster} $\mathfrak{c}(0)$ to be the finite component, and collapse everything in the ring of triangles and infinite connected component to a single face, which we declare to be the root (or external) face of $\mathfrak{c}(0)$. 

Given the map-to-word construction in the Mullin--Bernardi--Sheffield bijection, it is natural to draw a fictional ``flipped'' edge (diagonal) in the quadrangle corresponding to $X(0)$ its match, see Figure \ref{fig:outer-bdry}.  
We declare the root vertex of the map $\mathfrak{c}(0)$ to be the unique vertex in $\mathfrak{c}(0)$ that is adjacent to this fictional edge. 
The outer boundary of $\mathfrak{c}(0)$ is denoted by $\partial \mathfrak{c}(0)$. 
The perimeter of $\mathfrak{c}(0)$ is defined as the degree of the root face of $\mathfrak{c}(0)$, and is denoted by $|\partial \mathfrak{c}(0)|$. The typical cluster $\mathfrak{c}(0)$ is then a triangulation with fully packed loop configuration, and we write $|\partial \mathfrak{c}(0)|$ for its boundary length. Again, see Figure \ref{fig:outer-bdry}.}

\paragraph{Envelopes.} {When $X(0)=\rF$, we also have the notion of \textbf{envelope} associated with $X(0)$ (the terminology ``bubble'' is also used interchangeably with envelope in \cite{berestycki2017critical}).} Recall that we call a word {$e$} an \textbf{$\rF$-excursion} if it is of the form $\rh\cdots\rF$ (type $\rh$) or $\rc\cdots\rF$ (type $\rc$), {where the final $\rF$ is matched to the first letter} {(\cref{par:inf_fk_map})}. The associated envelope is then the (loop-decorated) submap of the infinite FK map encoded by this $\rF$-excursion. {We declare the root face of this submap to be one that is \emph{not} crossed by a loop encoded by an $\rF$ inside {$e$}.} It also has simple boundary, and we denote by $|\partial \mathfrak{e}(0)|$ its boundary length. Conditioned on $X(0)=\rF$, we write $\mathfrak{e}(0)$ for the envelope encoded by the $\rF$-excursion $X(\varphi(0))\cdots X(0)$. Note that, conditional on $X(0)=\rF$, both the typical loop and the typical cluster are included in the envelope $\mathfrak{e}(0)$. More precisely, since the match $\rc\rF$ (resp.\ $\rh\rF$) corresponds to a fictional edge in the encoding, $\rF$-excursions of type {$\rh$} (resp.\ {$\rc$}) correspond to envelopes with red outer boundary that surround a blue (resp.\ red) component, which is our typical cluster.

\begin{figure}[h] 
  \bigskip
  \begin{center}
    \includegraphics[page=1,scale=0.65]{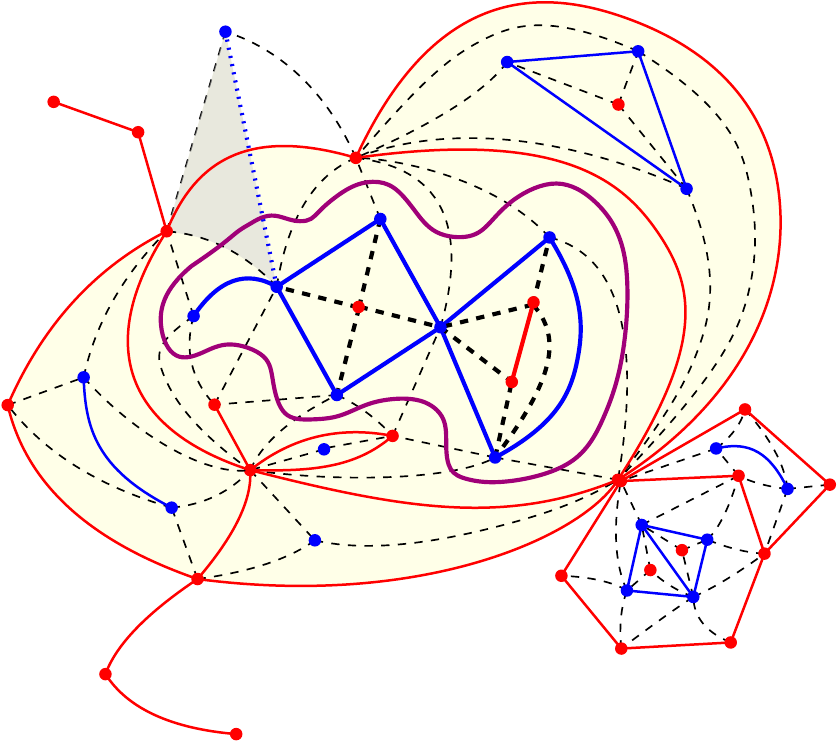}
  \end{center}
  \caption{A portion of the infinite FK planar map. Primal (resp.\ dual) elements are represented in blue (resp.\ red). In grey is the triangle corresponding to time $0$ in the hamburger-cheeseburger encoding, which is here assumed to correspond to an $\rF$ (we drew the corresponding fictional edge in dotted blue). This $\rF$ symbol determines a \emph{typical loop} $\mathfrak{L}(0)$ (purple). We denote by $\mathfrak{c}(0)$ the connected component (in bold) inside the loop and refer to it as the \emph{typical cluster}, in this case a primal (blue) component. The outside of $\mathfrak{c}(0)$ in this drawing is identified as a single root face for $\mathfrak{c}(0)$. It is readily checked that the perimeter of $\mathfrak{c}(0)$, i.e.\ the degree of this root face, is $|\partial \mathfrak{c}(0)| = 9$. {The \emph{envelope} $\mathfrak{e}(0)$ is the planar map in shaded yellow (together with the portion of the grey triangle that it includes), with its root face lying outside the yellow region. Note that it has simple boundary.}} 
  \label{fig:outer-bdry}
\end{figure}

\subsubsection{Dictionary: from words and walks to loops and clusters } 
\label{sss:dictionary}
In this subsection we show how several properties of loops, clusters and envelopes are encoded via the Mullin--Bernardi--Sheffield bijection.  

First, one can express both the perimeter of the typical cluster $|\partial \mathfrak{c}(0)|$ and the length of the typical loop $|\mathfrak{L}(0)|$ using the reduced walk defined in \cref{par:red_walk}.

\begin{Prop}[{Reduced walk expressions of $|\mathfrak{L}(0)|$ and $|\partial\mathfrak{c}(0)|$}]
\label{prop: loop_bdry_length}
On the event that $X(0) = \rF$, we have $|\mathfrak{L}(0)| = \tleft$. Furthermore, if $X(\varphi(0))=\rh$ 
we have $|\partial\mathfrak{c}(0)| {=}  \tauh-1$ and if $X(\varphi(0))=\rc$ 
we have $|\partial\mathfrak{c}(0)| {=}  \tauc-1$.
\end{Prop}

\begin{proof}
{We assume that $X(0)=\rF$.}
{Recall from \eqref{eq:dec_max_exc} the unique decomposition
$X(-\infty,0) = \cdots Y(2)Y(1)X(0)$, into maximal $\rF$-excursions.}
We will analyse the contribution of each $Y(i)$ ($1\leq i\leq \tleft$) to {the loop length $|\mathfrak{L}(0)|$ and the perimeter $|\partial \mathfrak{c}(0)|$ of the typical cluster. See \cref{fig:outer-bdry-sheffield} for an illustration.} 

First, we note that $Y(\tleft)$ must be the $\rc$ or $\rh$ symbol that is matched with $X(0)=\rF$. Indeed, suppose that the match $X(\varphi(0))$ of $X(0)$ is a $\rc$ (the $\rh$ case being symmetric). Then  $\overline{X(\varphi(0)+1,-1)}$ cannot contain a $\rC$ or $\rF$ symbol, meaning that $\mathcal{C}(\varphi(0)+1,-1)\le 0$ and $\mathcal{C}(\varphi(0),-1)\le -1$. In other words, if $Y(j)$ is the unique word in the maximal excursion decomposition including $X(\varphi(0))$, then we have $j\ge \tleft$. On the other hand, $Y(\tleft)$ must correspond to a step of $(\hleft,\cleft)$ from $(a,0)$ to $(a,-1)$ or $(0,a)$ to $(-1,a)$ for some $a\ge 0$. In the first case, this means that every $\rc$ in $Y(\tleft-1)\ldots Y(1)$ has a match (within $Y(\tleft-1)\dots Y(1)$), and that $Y(\tleft)=\rc$ which must therefore get matched with $X(0)=\rF$. 
Likewise, the second case implies that $Y(\tleft)=\rh$ is matched with $X(0)=\rF$, which is not possible since we assumed $X(\varphi(0)) = \rc$.

\begin{figure}[t] 
  \bigskip
  \begin{center}
    \includegraphics[page=2,scale=0.65]{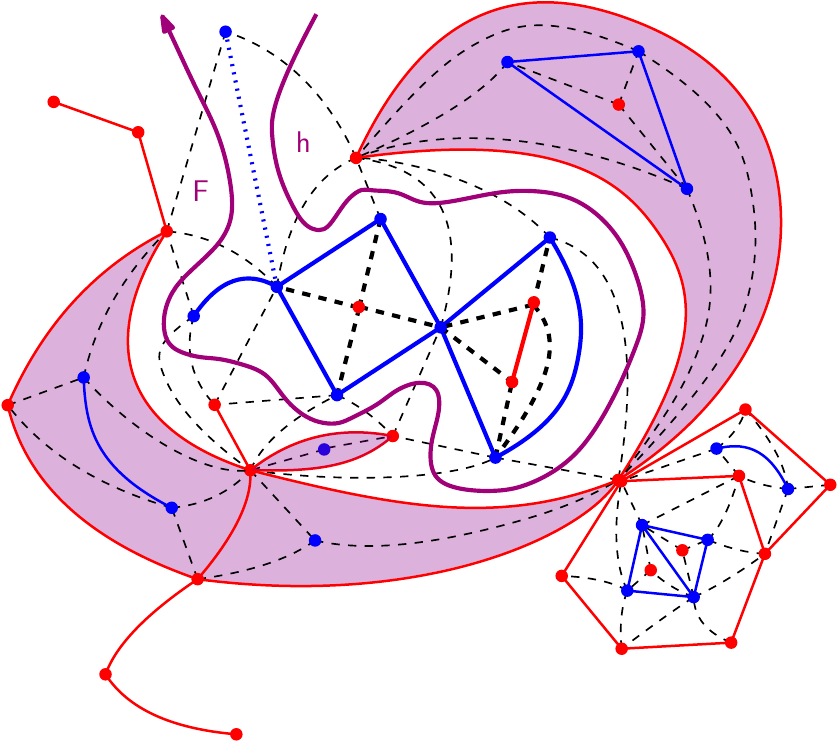}
  \end{center}
  \caption{A glimpse of the hamburger-cheeseburger exploration (purple) of the planar map in Figure \ref{fig:outer-bdry}. The exploration enters through a fictional triangle corresponding to an $\rF$-excursion of type $\rh$ (the fictional edge is drawn in dotted blue). When drawing the contour of the blue component, the purple exploration gets diverted and {fills in} some pockets (shaded purple regions) -- we did not represent these pocket diversions.} 
  \label{fig:outer-bdry-sheffield}
\end{figure}

{Next, we claim that any $i\in \{ 1,\ldots, \tleft\}$ such that $Y(i)\in \{\rh,\rc,\rH,\rC \}$ corresponds -- through Sheffield's encoding -- to a triangle that is crossed by the loop $\mathfrak{L}(0)$. Indeed, consider such an $i$. The triangle corresponding to it must be crossed by some loop $\mathfrak{L}_i$ in the infinite FK map (since the loop configuration is fully packed). Since that triangle is in the envelope $\mathfrak{e}(0)$ of $0$, the loop $\mathfrak{L}_i$ must also lie inside $\mathfrak{e}(0)$. Therefore, if $\mathfrak{L}_i \ne \mathfrak{L}(0)$, $Y(i)$ would be straddled by an $\rF$-excursion inside $Y(\tleft)\cdots Y(1)$, which contradicts the maximality of the decomposition.} Now let us see what happens if $i\in \{ 1,\ldots, \tleft\}$ is such that $Y(i)$ is an $\rF$-excursion. In that case, $Y(i)$ corresponds to an envelope $\mathfrak{e}_i$ contained inside $\mathfrak{e}(0)$. Such an envelope contains exactly one symbol corresponding to a triangle intersecting $\mathfrak{L}(0)$. This is the one that has its associated quadrangle flipped when constructing the word from the map.

From the above two paragraphs, we deduce that each $Y(i)$ for $1\le i\le \tleft$ corresponds to exactly one (distinct) triangle crossed by $\mathfrak{L}(0)$. On the other hand, any triangle crossed by $\mathfrak{L}(0)$ must correspond to a symbol in the $\rF$-excursion at $0$; that is, in the word $Y(\tleft)\dots Y(1)$. As a consequence, $|\mathfrak{L}(0)| = \tleft$.

We get the expression for $|\partial\mathfrak{c}(0)|$ in a similar way, but keeping track of the colours. By symmetry, we may restrict to the case when $\tleft=\thleft$, in which case we note that the match of $X(0) = \rF$ is an $\rh$. We first claim that any $i\in \{1,\ldots,\tleft-1\}$ such that $Y(i) = \rh$ or $\rH$ corresponds to a primal (blue) triangle whose blue edge lies on the boundary $\partial\mathfrak{c}(0)$. {Indeed, it corresponds to a triangle crossed by $\mathfrak{L}(0)$ by the above discussion, and it is primal because $Y(i)=\rh$ or $\rH$. Hence the only possibility is that the triangle has its blue edge along the boundary $\partial\mathfrak{c}(0)$ of the cluster.} Secondly, if $Y(i)$ is an $\rF$-excursion of type $\rc$, it encodes an envelope with primal (blue) outer boundary, which lies in $\mathfrak{c}(0)$. {As in the previous part of the proof, such an $\rF$-excursion  has exactly one symbol corresponding to a triangle lying on $\partial\mathfrak{c}(0)$.} In conclusion, in both cases above, the contribution of such $Y(i)$ to $|\partial\mathfrak{c}(0)|$ is 1. Finally, all the other triangles corresponding to symbols in $\mathfrak{e}(0)$ do not contribute to $|\partial\mathfrak{c}(0)|$, since they are either dual triangles, or primal triangles lying inside an envelope with dual (red) boundary. We conclude that $|\partial\mathfrak{c}(0)| = \tauh-1$ {(note the $-1$ coming from the fact that $Y(\tleft)=\rh$ is ``fictional'' and does not correspond to a triangle with an edge on $\partial \mathfrak{c}(0)$).} 
\end{proof}

We will now explain more precisely how clusters are encoded through the Mullin--Bernardi--Sheffield bijection. This observation seems to be new and will allow us, in \cref{ss:cluster}, to calculate the probability that $\mathfrak{c}(0) = \tfrak$ for any fixed rooted triangulation $\tfrak$ with fully packed loop configuration. We call a word $w$ in the alphabet $\{\rc,\rC,\rh,\rH,\rF\}$ a \textbf{skeleton word} of type $\rh$, if $\overline{w}= \emptyset$ and if $w$ can be written
\begin{equation} \label{eq:skeleton_fact}
     w= A_1 E_1\cdots A_n E_n A_{n+1},
\end{equation}
where the $E_i$ are $\rF$-excursions of type $\rc$  and the $A_i$ are words in the alphabet $\{\rh,\rH\}$. Equivalently, $\overline{w}=\emptyset$ and $w$ is a concatenation of {letters in the alphabet} $\mathscr{A}_{\rh}$ introduced in \cref{sec: walk definitions}. Note that if $w$ is a skeleton word, the factorisation \eqref{eq:skeleton_fact} is unique. By swapping $\rc$ and $\rh$ in the previous definition, one may likewise define  skeleton words of type $\rc$. 
 
If ${e}$ is an $\rF$-excursion of type $\rh$ (say), one can construct a skeleton word $\mathsf{sk}({e})$ by simply forgetting inside ${e}$ all the sub-$\rF$-excursions of type $\rh$ that are not contained inside an $\rF$-excursion of type $\rc$, as well as all letters $\rc$ and $\rC$ that do not lie inside an $\rF$-excursion of type $\rc$. For instance, if
\[{e}=\rh \rC \rh \rh\rC\rF \rH \rc \rh\rc\rC\rH \rF \rC \rF,\]
then 
\begin{equation} \label{eq:def_sk(e)}
\mathsf{sk}({e}) 
=
\rh\rH\rc\rh\rc\rC\rH\rF. 
\end{equation}
The motivation for introducing skeleton words comes from the following geometric interpretation. See Figures \ref{fig:outer-bdry-sheffield} and \ref{fig:ring-config}.

\begin{figure}[t] 
  \bigskip
  \begin{center}
    \includegraphics[scale=0.9]{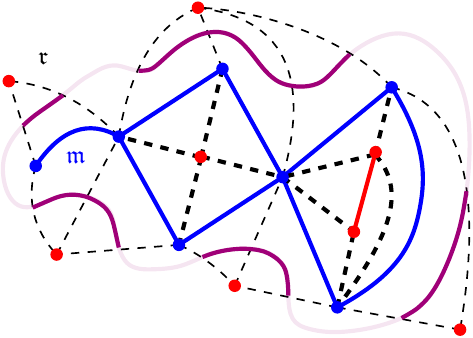}
  \end{center}
  \caption{The planar map corresponding to $\textsf{sk}(E)$, scooped out of Figure \ref{fig:outer-bdry}. Symbols $R$ in the decomposition \eqref{eq:dec_skeleton} correspond to collections of triangles that form an outer \emph{ring} $\mathfrak{r}$ {(not represented)}. We drew in purple the same loop as in Figure \ref{fig:outer-bdry}, with shaded lines crossing the ring triangles in a single loop outside of the picture.} 
  \label{fig:ring-config}
\end{figure}

\begin{lemma}[Skeleton words are clusters] 
    On the event that $X(0)=\rF$, let $E=X(\varphi(0))\cdots X(0)$ the $\rF$-excursion encoding the envelope $\mathfrak{e}(0)$. Then the triangles in the infinite FK map that have an edge in $\mathfrak{c}(0)$ are in one-to-one correspondence with symbols in $\mathsf{sk}(E)$. {Moreover, under this correspondence, the triangles are explored consecutively in Sheffield's bijection when reading $\mathsf{sk}(E)$ from left to right. }
    \label{lem:skeleton=clusters}
\end{lemma}
\begin{proof}
Without loss of generality, we may assume that $X(\varphi(0))=\rh$, {so that $\mathfrak{c}(0)$ is a primal (blue) component}. Recall the decomposition into maximal excursions of \eqref{eq:dec_max_exc}.
We can further decompose the word $E = X(\varphi(0))\cdots Y(2)Y(1)X(0)$ in a unique way into
\begin{equation}\label{eq:dec_skeleton}
    E = \rh R(\ell)S(\ell)\cdots R(1)S(1)R(0)\rF, 
\end{equation}
\noindent  where the $S(i)$ are either $\rh$, $\rH$ or a maximal $\rF$-excursion of type $\rc$, the $R(i)$'s are (possibly empty) subwords in the $Y$'s, which are in $\mathscr{A}_{\rc}$, and by (the proof of) \cref{prop: loop_bdry_length}, $\ell=\tauh-1=|\partial \mathfrak{c}(0)|$. 

Moreover, we observe that in the decomposition \eqref{eq:dec_skeleton}, we have $\mathsf{sk}(E) = S(\ell)\cdots S(1)$. It remains to see that the triangles in the infinite FK map with an edge in $\mathfrak{c}(0)$ are in one-to-one correspondence with the symbols in the $S(i)$ for $1\le i\le \ell$. We break these triangles up into two sets: the triangles not included in $\mathfrak{c}(0)$ but with an edge in $\mathfrak{c}(0)$, and the triangles in the submap $\mathfrak{c}(0).$

By the proof of \cref{prop: loop_bdry_length} (see penultimate paragraph), each triangle of the first kind corresponds to a distinct $i\in\{1,\ldots,\ell\}$. More precisely, each such triangle either corresponds to one of the $S(i)$  with $S(i)\in\{\rh,\rH\}$, or is the unique triangle on the exterior of the envelope corresponding to an $\rF$-excursion $S(i)$, whose quadrangle has its diagonal flipped when exploring that envelope. Furthermore, each triangle in $\mathfrak{c}(0)$, i.e.\ of the second kind, is contained in the envelope corresponding to some {$\rF$-excursion} $S(i)$ for $1\le i\le \ell$. Indeed, such a triangle must lie in an {outermost} envelope contained in $\mathfrak{c}(0)$, which must have primal (blue) boundary since $\partial \mathfrak{c}(0)$ is blue. But the maximal envelopes with blue boundary contained in $\mathfrak{e}(0)$ are precisely  encoded by the $S(i)$ for $1\le i\le \ell$ which are $\rF$-excursions. This proves the correspondence. The ordering of triangles follows directly from the decomposition \eqref{eq:dec_skeleton}. 
\end{proof}

\subsection{Fully packed loop-$O(n)$ model on triangulations}
\label{sec:O(2)_model}

In this section we recall the fully packed loop-$O(n)$ model on planar triangulations and its connection to FK$(q)$-decorated planar maps. 
{The key feature of this model is that, via the \textbf{gasket decomposition} of \cite{borot2012recursive}, it satisfies a certain natural functional equation, which we will solve to compute the partition function at $n=2$.}

\paragraph{Definition of the model and correspondence with FK-decorated maps.}
In \cref{sec:mullin-bernardi-sheffield}, we described how to build planar rooted triangulations (the so-called Tutte map) from FK-decorated planar maps, where all faces (including the root face) are triangles. {Because we want to make use of the spatial Markov property of the model,} it is natural to extend it to triangulations with boundary length $\ell$, where the root face has degree $\ell$ not necessarily equal to $3$.

Let $\Tb_{\ell}$ denote the set of rooted planar triangulations $\tfrak$ with boundary length $\ell$, together with a configuration $\boldsymbol{\ell}$ of disjoint, simple loops drawn on the dual map. We will always take the loop configuration to be \textbf{fully packed} in the sense that every internal face of $\tfrak$ is visited by a loop. For parameters $x,n>0$, we assign to $(\tfrak,\boldsymbol{\ell})\in \Tb_\ell$ the weight
\begin{equation}\label{eq: def weight triangulation}
Z(\tfrak,\boldsymbol{\ell},x,n)=n^{\# \boldsymbol{\ell}} \, x^{\#  {F(\tfrak)-1}},
\end{equation}
i.e.\ each internal face receives weight $x$, while each loop receives the global weight $n$.
The corresponding \textbf{partition function} is
\begin{equation}\label{eq: def partition triangulation}
F_{\ell}=\sum_{(\tfrak,\boldsymbol{\ell})\in \Tb_{\ell}} Z(\tfrak,\boldsymbol{\ell},x,n),
\end{equation}
and, when $F_\ell<\infty$, this defines a probability measure on $\Tb_{\ell}$, {called the \textbf{fully packed loop-$O(n)$ model on triangulations}}.

This model can be viewed as a symmetric case of the twofold loop model in \cite{borot2012loop}, where regions of {the map} separated by loops are assigned a colour \emph{red} or \emph{blue} {depending only on the parity of the number of surrounding loops}, and face/vertex weights may depend on the colour. In this picture, the boundary is monochromatic (since there is no loop visiting the root face) and the colouring procedure is completely determined by the colour of the boundary. We do not need the detailed coloured formulation here, only that when the face and vertex weights are chosen symmetrically this reduces to the fully packed $O(n)$ model above, and the partition functions corresponding to the two possible boundary colours coincide.

{It is known \cite{borot2012loop} that the Tutte map $T(\mathfrak{m}, \pfrak)$ of a self-dual FK$(q)$-weighted map $(\mathfrak{m}, \pfrak)$ with $q=q(p)$ (as introduced in \cref{sss:def_FK}) has the law of a fully packed loop-$O(n)$ triangulation with weights
\begin{equation}\label{eq:x_c_formula}
    n = \frac{2p}{1-p} \text{ and } x=x_{c}(n)=\frac1{\sqrt{8(n+2)}},
\end{equation}
in the following sense.}
Given $(\tfrak, \boldsymbol{\ell})\in \mathbb{T}_{\ell}$ with boundary of a given colour (which uniquely determines a colouring of all the triangles in $\tfrak$ by switching colours whenever a loop is crossed) we can join each of the $\ell$ boundary edges to an additional external vertex (of the opposite colour to the boundary vertices), and add a loop passing through these additional triangles. This creates a rooted triangulation with exactly $\#F(\tfrak)+\ell-1$ triangles (if $\tfrak$ has $\# F(\tfrak)$ {faces}), and $\# \boldsymbol{\ell}+1$ loops. Via the map-to-word direction of the Mullin--Bernardi--Sheffield bijection (\cref{para:mapword}), this corresponds to a word $w$ of length $2k=\# F(\tfrak)+\ell-1$ with $\overline{w}=\emptyset$ and weight 
\begin{equation}\label{eq:link_FK_Z}
\Big(\frac{1-p}{16}\Big)^{{\frac{\# F(\tfrak)+\ell-1}{2}}}\Big(\frac{2p}{1-p}\Big)^{\# \boldsymbol{\ell}+1}=nx_c^{\ell+1}x_c^{\# F(\tfrak)-1}n^{\# \boldsymbol{\ell}}=nx_c^{\ell+1}Z(\tfrak,\boldsymbol{\ell},x_c,n), 
\end{equation} 
for $n$ and $x_c$ as in \eqref{eq:x_c_formula}. As before, we identify blue (respectively red) triangles with primal (respectively dual) triangles.

\paragraph{The gasket decomposition.} The gasket decomposition is a powerful approach towards the analysis of various loop models in \cite{borot2012recursive}, \cite{borot2012loop}. We follow the exposition of \cite{borot2012loop} but under the setting of our special fully packed model defined in \eqref{eq: def weight triangulation} and \eqref{eq: def partition triangulation}. 

Let $(\tfrak,\boldsymbol{\ell})\in \mathbb{T}_\ell$. 
{We first define the gasket associated with $(\tfrak,\boldsymbol{\ell})$. Consider all the edges in $\tfrak$ that can be reached from the boundary, and traversed, without crossing any loops. {Since the loop configuration is fully packed,} this set of edges defines a map with ${\# \boldsymbol{\ell}}+1$ faces: the external/root face (which is the same as before and has boundary length $\ell$), and one face of degree $k$ for each loop of {outer} length $k$. We call this the \textbf{gasket} of $(\tfrak,\boldsymbol{\ell})$. If inside each of these (internal) faces we draw the triangles crossed by the corresponding loop, we obtain a ring of triangles with outer boundary length $k$ and inner boundary length $k'$ for some $k'$. The total number of triangles in this ring is $k+k'$. Moreover, for each such ring, {the original loop-decorated triangulation $(\tfrak, \boldsymbol{\ell})$ determines} a {(rooted)} triangulation with boundary length $k'$ and fully packed loop configuration {that was removed when constructing the gasket}. So, we can decompose $(\tfrak,\boldsymbol{\ell})$ as the gasket, plus a ring of triangles attached to the boundary of each internal face, and a fully packed loop decorated triangulation glued along the inner boundary of each such ring. }

{The contribution to the weight $Z({\tfrak,\boldsymbol{\ell}},x,n)$ in \eqref{eq: def weight triangulation} of an internal face of degree $k$ in the gasket, whose associated ring of triangles has inner boundary length $k'$ and  $(\tfrak',\boldsymbol{\ell}')\in \mathbb{T}_{k'}$ inside, is $n x^{k + k'}Z(\tfrak', \boldsymbol{\ell}', x, n)$.} 
We therefore define the weight
\begin{equation}\label{eq: gasket decomposition}
g_k := n\sum^{\infty}_{k' = 0}A_{k\rightarrow k'}x^{k + k'}F_{k'}
\end{equation}
\noindent where $A_{k\rightarrow k'}$ is the number of configurations of rings of triangles with outer boundary $k$ and inner boundary $k'$. 
If we consider a general random planar map with boundary length (i.e.\,length of the root/external face) equal to $\ell$, where each face of degree $m$ is assigned a weight $g_m$, then its partition function $\mathcal{F}_{\ell}((g_m)_{m\geq 1})$ is exactly the same as $F_{\ell}$. Hence we may rewrite 
\eqref{eq: gasket decomposition} as
\begin{equation}\label{eq: fixed point equation}
g_k = n\sum^{\infty}_{k' = 0}A_{k\rightarrow k'}x^{k + k'}\mathcal{F}_{k'}((g_m)_{m\geq 1}). 
\end{equation}

\noindent This equation is called the fixed point equation. Define the resolvent function $W(z)$ by 
\begin{equation}\label{eq: def resolvent}
W(z) = \sum^{\infty}_{l=0} \frac{\mathcal{F}_{\ell}((g_m)_{m\geq 1})}{z^{\ell + 1}} = \sum^{\infty}_{l=0} \frac{F_{\ell}}{z^{\ell + 1}}. 
\end{equation}

\noindent It has been proved $W$ is analytic in $\mathbb{C}\setminus[a, b]$ where $[a, b]$ is a real interval with $b\geq |a|$ and $0\in [a, b]$. In this setting, \cite[Equation (3.22), (3.23)]{borot2012loop} simplify to
\begin{equation}\label{eq: general resolvent equation}
W(z + i0) + W(z - i0) + nW\left(\frac{1}{x} - z\right) = z, \quad \mbox{for } z\in (a, b). 
\end{equation}

\noindent We remark that the admissibility implies $b \leq  \frac{1}{{2x}}$. 
We will solve this equation in the critical case when $n = 2$ and $x = x_c = (4\sqrt2)^{-1}$.

\subsection{Proof outline}\label{ssec: outline}

We now outline our approach to proving the scaling limit stated in \cref{thm: main intro}. As discussed in \cref{sss: counts}, the law of the burger count $\cS$ is independent of $p$ and coincides with that of a simple symmetric random walk. The main challenge lies in controlling the discrepancy process $\cD$.

For context and comparison, we start by briefly recalling Sheffield’s argument \cite{SheffieldScott2016QGAI} in the regime $0<q<4$. The proof is \emph{elementary} in that it relies solely on martingale techniques. At a high level, the argument hinges upon two central ideas:
\begin{itemize}
\item \textbf{Tightness.} The first key step is to show that the discrepancy $\cD$ remains of the same order as the burger count $\cS$ (denoted $\cC$ in Sheffield’s notation). To this end, Sheffield introduces a clever lattice-path encoding of burger stacks, known as \emph{embedded stacks}, which has the nice feature of being monotone \cite[Section~3.4]{SheffieldScott2016QGAI}. This monotonicity implies that $\cD$ is at most of the same order as the burger count $\cS$, leading to tightness in $n$ for the pair $(\cS_n/\sqrt{n}, \cD_n/\sqrt{n})$.
\item \textbf{Identification of the limit.} The second step is to show that, among the indices ${1,\ldots,n}$, the number of $\rF$ symbols matched to the left of $0$ is negligible at scale $\sqrt{n}$ (these $\rF$ symbols are called \emph{unmatched} in this paper). Proving this fact is technically involved but relies again only on martingale techniques and a variance estimate. This observation allows one to decompose the discrepancy $\cD_n$ at time $n$ into an independent sum of $k$ terms distributed as $\cD_{\lfloor n/k \rfloor}$, up to an error term that is negligible at scale $\sqrt{n}$. Hence, the scaling limit must be infinitely divisible, and one can further identify it as the process described in \eqref{eq: sheffield's convergence}.
\end{itemize}

\noindent In the regime $q=4$, both arguments break down. Indeed, $\cD_n$ is no longer of the same order as $\cS_n$, while we do not have precise enough control on the variance to ensure that the number of unmatched $\rF$ symbols is negligible at the right scale -- let alone to identify the limit.

Instead, we use the planar map structure to obtain some preliminary information about the hamburger-cheeseburger walks. Our start point is to map the FK$(4)$ model to the loop-$O(2)$ model on triangulations,  from which we derive both the exact expression and the asymptotic behaviour of the associated partition function $F_{\ell}$. This part, presented in \cref{sec:resolvent_F_ell}, draws on ideas from analytic combinatorics (in particular the \emph{gasket decomposition} introduced in \cite{borot2012loop}) and has a different flavour from the rest of the article. It leads to the first rigorous derivation of the loop-$O(2)$ partition function expression and asymptotic (\cref{thm:asympt_F_ell_intro}), thereby confirming the physics prediction of Gaudin and Kostov \cite{gaudin1989n}. 

Our strategy is then to approximate the discrepancy using Markovian walks, rather than dealing with the discrepancy directly. To do this, we construct two explorations -- one into the \textbf{past} and one into the \textbf{future} -- as defined in \cref{sec: walk definitions}. These explorations admit a natural interpretation in terms of certain peeling explorations in the infinite-volume FK$(4)$ planar map. We relate in \cref{sec: geometric exponents} some observables of the exploration into the past (namely, hitting times) to the partition function $F_\ell$,\footnote{At a very high level, this part of the proof therefore bears some connections to the literature on the loop-$O(n)$ model on quadrangulations, where the partition function has been used to establish estimates on the peeling exploration \cite{budd2018peeling,curien2019peeling}.}
which, in particular, implies the first two statements of \cref{thm:main_exponents}.
To the best of our knowledge, this connection appears to be novel. It goes through the notion of \emph{skeleton words} that encode clusters through the Mullin--Bernardi--Sheffield bijection (\cref{lem:skeleton=clusters}).
Using Tauberian theorems and regular variation techniques, it further enables us, in \cref{sec:backward}, to analyse the step distribution of the walk and to derive the scaling limit of the exploration into the past. 
We also emphasise that, because the regime $q=4$ is critical, both explorations will be proved to lie in the domain of attraction of a Cauchy-type process, which makes the analysis more subtle; in particular, we will make use of de Haan theory.

The analysis of the exploration into the future is more delicate, as it is only an \emph{approximate} random walk (\cref{sec:explo_future}). We use a change of measure argument to derive the law of the ``step distribution'' associated with this exploration. We shall see that this exploration can be effectively analysed using the exploration into the past. This approach enables us to establish two key properties: (a) the scaling limit of the exploration into the future; and (b) that the number of unmatched $\rF$ symbols is negligible at the right scale $v_n=\sqrt{n}/\log(n)$ (and in fact, we will see that it is only off by an extra logarithmic factor).

The tightness of the discrepancy is proved in \cref{sec:final_cvg}, where we relate it to the above two explorations. The idea is to constrain the discrepancy between the explorations into the past and future to get the desired tightness. The difficulty will lie in controlling the discrepancy at a fixed time $n$ in terms of the discrepancy at (random) times corresponding to the two walks. The proof draws inspiration from (discrete) excursion theory. We refer the reader to \cref{sec:tightness_Dn} for a more detailed account of how these two explorations are combined to establish tightness. Once tightness has been obtained, our estimates on the number of unmatched $\rF$ symbols can be used to derive the full scaling limit in \cref{thm: main intro}.

Finally, we emphasise that, since we have access to an explicit expression for the partition function $F_{\ell}$, all our estimates are \emph{exact}, capturing the integrable nature of the model. 
{These techniques are further extended by Berestycki and Da Silva in the concurrent work \cite{berestycki2025critical} to provide the exact critical exponents and partition functions of the FK$(q)$ and loop-$O(n)$ models in the regime $0<q<4$, $0<n<2$, which is the main focus of Sheffield's work \cite{SheffieldScott2016QGAI}.}

%
%
\section{Resolvent equation and partition function asymptotics}
\label{sec:resolvent_F_ell}

In this section we work in the critical case when 
\begin{equation}\label{eq: critical n  z}
n=2 \text{ and } x = x_c=\frac{1}{4\sqrt2}.
\end{equation}
The whole section is devoted to proving \cref{thm:asympt_F_ell_intro}. That is, solving \eqref{eq: general resolvent equation} to obtain the exact expression
\begin{equation}\label{eq:exact_F_ell}
F_{\ell} = 2(2\sqrt2)^{\ell}\int^1_0 u\bigg(1-\frac{\pi}{2}u\bigg)^{\ell}\log\bigg(\frac{1 + \sqrt{1-u^2}}{u}\bigg)\mathrm{d}u,
\end{equation}
and deriving the asymptotics 
\begin{equation}\label{eq:asympt_F_ell}
F_{\ell} \sim \frac{8}{\pi^2}(2\sqrt2)^{\ell}\frac{\log\ell}{\ell^2}
\quad \text{as } \ell\to \infty.
\end{equation}
This will play a crucial role in deriving the loop and boundary length exponents in \cref{sec: geometric exponents}, and in the proof of our main result \cref{thm: main intro}.

\subsection{Exact expression of the spectral density}
In this section, we denote by 
\[\displaystyle\dashint^b_a f(y)\frac{\mathrm{d}y}{y-c}\]
the principal-value integral. For a real-valued function $f$ that is continuous on $(a, b)$ and $c\in(a, b)$,  
\[{\dashint^b_a} f(y)\frac{\mathrm{d}y}{y-c} := \lim_{\varepsilon\to 0+}  \bigg(\int^{c-\varepsilon}_{a} f(y)\frac{\mathrm{d}y}{y-c} + \int^b_{c+\varepsilon} f(y)\frac{\mathrm{d}y}{y-c}\bigg)  = \lim_{\varepsilon\to 0}\mathrm{Re} \int^b_a f(y)\frac{\mathrm{d}y}{y-c+i\varepsilon},\]
whenever one of the limits exists.

{We first recall a few properties of the resolvent function $W$ introduced in \eqref{eq: def resolvent}, which are part of the so-called \emph{one-cut lemma} proved in \cite[Section 6]{borot2012more} using combinatorial arguments.}
The basic main object we use to derive the solution of \eqref{eq: general resolvent equation} is the \textbf{spectral density} $\rho$ introduced in \cite{borot2012loop}, defined as follows:
\begin{equation}\label{def: spectral density}
\rho(y): = - \frac{W(y+i0)-W(y-i0)}{2\pi i}, \quad y\in[a, b]. 
\end{equation}
It is known that $\rho$ is a positive and continuous function defined on the cut $[a,b]$ that vanishes at the endpoints $a$ and $b$, and it can be deduced from \cite[Equation 6.12]{borot2012more} that $\rho$ is differentiable on $(a, b)$.\footnote{In particular, these properties will justify the applications of the Sokhotsi--Plemelj formula below. Alternatively, we may simply assume that the Sokhotsi--Plemelj formula can be applied to $\rho$, and then after deriving our expression for $\rho$, justify this fact a posteriori by uniqueness.} 
Recall that \eqref{eq: def resolvent} implies $W(z)\sim\frac{1}{z}$ as $z\to\infty$. The Cauchy integral formula then implies that
\begin{equation}\label{eq: resolvent out of density}
W(z) = \int^b_a\frac{\rho(y)}{z-y}\mathrm{d}y, \quad  z\in\mathbb{C}\setminus[a, b]. 
\end{equation}

\noindent By the Sokhotski-Plemelj formula \cite[Theorem 3.2.3]{MandalB.N2011Asie}, we have 
\[W(y\pm i0) = \dashint^b_a\frac{\rho(w)}{y-w}\mathrm{d}w \mp \pi i \rho(y), \quad y\in(a, b). \]

\noindent We can then re-express the resolvent equation \eqref{eq: general resolvent equation}, together with the condition that $W(z)\sim \frac{1}{z}$ as $z\to\infty$, as follows: 
\begin{equation}\label{eq: resolvent equation rho}
\dashint^b_a \rho(y)\left(\frac{1}{y-w} - \frac{1}{y+w-4\sqrt2}\right)\mathrm{d}w = \frac{y}{2}, \quad y\in(a, b), \quad \int^b_a \rho(w)\mathrm{d}w = 1.
\end{equation}

\noindent Importantly, we view \eqref{eq: resolvent equation rho} as an equation on $(a, b, \rho)$. We stress that \eqref{eq: resolvent equation rho} contains all the information about $W(z)$. To see this, define $W$ as in \eqref{eq: resolvent out of density}. Then $\mathrm{Re}\, W$ is a harmonic function on $\mathbb{C}^*\setminus[a, b]$ with boundary condition given by the first equation in \eqref{eq: resolvent equation rho}
\[
\mathrm{Re} \, W(y)=\frac{y}{2} + \int^b_a \frac{\rho(w)\mathrm{d}w}{y+w-4\sqrt2}, \quad y\in [a, b]. 
\]
Therefore, $\mathrm{Re}\, W$ is determined by the function $\rho$ and the cut $[a, b]$. The second equation implies that $W(z) \sim 1/z$ as $z\to \infty$. This uniquely determines $W$. 

Now we focus on solving \eqref{eq: resolvent equation rho}. By the change of variables $u = \sqrt2 - \frac{z}{2}$ and $v = \sqrt2 - \frac{y}{2}$, we get
\begin{equation}\label{eq: resolvent equation tilde rho}
\dashint^{\tilde{a}}_{\tilde{b}} \tilde{\rho}(v)\left(\frac{1}{v-u} + \frac{1}{u+v}\right)\mathrm{d}v = \sqrt{2} - u, \quad u\in(\tilde{b}, \tilde{a}), \quad \int^{\tilde{a}}_{\tilde{b}} \tilde{\rho}(v)\mathrm{d}v = \frac{1}{2}.
\end{equation}

\noindent Here $\tilde{a} = \sqrt2 - \frac{a}{2}$, $\tilde{b} = \sqrt2 - \frac{b}{2}$ and $\tilde{\rho}(v) = \rho(2\sqrt{2} - 2v)$. Set $A = \frac{1}{2}(\tilde{a}^2+\tilde{b}^2)$ and $B = \frac{1}{2}(\tilde{a}^2-\tilde{b}^2)$. We perform another change of variables, namely $v = \sqrt{A + Bw}$ and $u = \sqrt{A + {B}t}$. Then \eqref{eq: resolvent equation tilde rho} is transformed into
\begin{equation}\label{eq: resolvent equation hat rho}
\dashint^{1}_{-1} \frac{\hat{\rho}(w)}{w-t}\mathrm{d}w = \sqrt{2} - \sqrt{A+Bt}, \quad t\in(-1, 1), \quad \int^{1}_{-1} \frac{\hat{\rho}(w)}{\sqrt{A+Bw}}\mathrm{d}w = \frac{1}{B},
\end{equation}

\noindent where $\hat{\rho}(w) = \tilde{\rho}(\sqrt{A+Bw})$.

We first fix the cut and show that one can solve the first equation of \eqref{eq: resolvent equation hat rho} given the cut.

\begin{Prop}[Explicit expression of $\hat\rho$ given $(A,B)$]
\label{prop: solution hat rho}
{For all fixed $A \geq B \geq 0$, the solution of the first equation of} \eqref{eq: resolvent equation hat rho} is 
\begin{equation}\label{eq: solution hat rho}
\hat{\rho}(w) = \frac{\sqrt{1-w^2}}{\pi^2}\dashint^1_{-1}\sqrt{\frac{A+Bs}{1-s^2}}\frac{\mathrm{d}s}{s-w}, \quad w\in(-1, 1). 
\end{equation}
\end{Prop}

\begin{proof}
The solution is derived using a ``standard'' argument. Let 
\begin{equation}\label{eq: def Phi}
\Phi(z) := \int^1_{-1}\frac{\hat{\rho}(s)}{z-s}\mathrm{d}s, \quad z\in{\mathbb{C}^*}\setminus[-1, 1]. 
\end{equation}

\noindent Then $\Phi$ is a holomorphic function in ${\mathbb{C}^*}\setminus [-1, 1]$ whose spectral density along the cut $[-1, 1]$ is $\hat\rho$. By the Sokhotski-Plemelj formula \cite[Theorem 3.2.3]{MandalB.N2011Asie}, 
\[\Phi(w\pm i0) = \dashint^1_{-1}\frac{\hat\rho(s)}{w-s}\mathrm{d}s \mp \pi i \hat\rho(w), \quad w\in(-1, 1). \]

\noindent Now introduce a holomorphic function $g(z)$ in ${\mathbb{C}^*}\setminus[-1, 1]$ such that $g(z)^{-2} = z^2-1$, $g(z)\sim {1/z}$ as $z\to\infty$ and $g(w\pm i0) = \pm \frac{i}{\sqrt{1-z^2}}$ for $w\in (-1, 1)$. 
Let 
\begin{equation}\label{eq: def hatPhi}
\hat{\Phi}(z) = {\Phi(z)g(z)}, 
\end{equation}

\noindent {which is} a holomorphic function in ${\mathbb{C}^*}\setminus [-1, 1]$. By our choice of $g(z)$, we have 
\begin{equation}\label{eq: hat phi cut}
\hat\Phi(w\pm i0) = \mp\frac{i}{\sqrt{1-w^2}}\dashint^1_{-1}\frac{\hat\rho(s)}{w-s}\mathrm{d}s - \frac{\pi\hat{\rho}(w)}{\sqrt{1-w^2}}, \quad w\in(-1, 1). 
\end{equation}

\noindent Hence {by the first equation of \eqref{eq: resolvent equation hat rho},}
\begin{equation} \label{eq:diff_hatPhi_cut}
    \hat\Phi(w+i0) - \hat\Phi(w-i0) = -\frac{2i}{\sqrt{1-w^2}}\dashint^{1}_{-1}\frac{\hat\rho(s)}{w-s}\mathrm{d}s \overset{\eqref{eq: resolvent equation hat rho}}{=} \frac{2i}{\sqrt{1-w^2}}(\sqrt{2}-\sqrt{A+Bw}), \quad w\in(-1, 1). 
\end{equation}

\noindent Using Cauchy's integration formula again, for $z\in{\mathbb{C}^*}\setminus [-1, 1]$ we have 
\begin{equation}\label{eq: hat phi}
\hat\Phi(z) = \frac{1}{2\pi i}\int^1_{-1}\frac{\hat\Phi(s+i0)-\hat\Phi(s-i0)}{s-z}\mathrm{d}s = \frac{1}{\pi}\int^1_{-1}\frac{\sqrt{2}-\sqrt{A+Bs}}{{\sqrt{1-s^2}}}\frac{\mathrm{d}s}{s-z}. 
\end{equation}

\noindent The value of $\hat\rho(w)$ can then be read from \eqref{eq: hat phi cut} and \eqref{eq: hat phi}. For $w\in(-1, 1)$, 
\begin{eqnarray*}
\hat\rho(w) &=& -\frac{\sqrt{1-w^2}}{\pi}\mathrm{Re} \, \hat\Phi(w+i0)\\
&=& -\frac{\sqrt{1-w^2}}{\pi^2}\dashint^1_{-1}\frac{\sqrt{2}-\sqrt{A+Bs}}{\sqrt{1-s^2}}\frac{\mathrm{d}s}{s-w}. 
\end{eqnarray*}

\noindent We complete the proof using 
{that
for all $w\in (-1, 1)$, 
\[
\dashint^{1}_{-1}\frac{1}{\sqrt{1-s^2}}\frac{\mathrm{d}s}{s-w} = 0,
\]
which follows from elementary calculations.
}
\end{proof}

We now use the exact expression of $\hat\rho$ derived in \cref{prop: solution hat rho} to solve the second equation of \eqref{eq: resolvent equation hat rho}. Before giving the exact expression of the cut, we need the following preliminary lemma.
\begin{lemma} \label{lem:int_A_B_pi}
    If $A$ and $B$ satisfy \eqref{eq: resolvent equation hat rho}, then
    \begin{equation} \label{eq:int_A_B_pi}
    \int^{\frac{\pi}{2}}_{-\frac{\pi}{2}}
\sqrt{A + B\sin\theta} \mathrm{d}\theta = \sqrt{2}\pi. 
    \end{equation}
\end{lemma}

\begin{proof}
Let $\Phi$ {and $\hat\Phi$} be the functions defined in \eqref{eq: def Phi} {and \eqref{eq: def hatPhi} respectively}. {From \eqref{eq:diff_hatPhi_cut}}, we have
\begin{align*}
\int^{1}_{-1}\frac{\sqrt2 - \sqrt{A+Bt}}{\sqrt{1-t^2}}\mathrm{d}t 
& = -\frac{i}{2}\int^{1}_{-1} \big( \hat\Phi(t + i0) - \hat\Phi(t - i0)\big) \mathrm{d}t
= \frac{i}{2} \oint_{\mathscr{C}} \hat\Phi(z) dz,
\end{align*}

\noindent where $\mathscr{C}$ is any closed contour in anti-clockwise orientation that encloses $[-1,1]$. In particular, we may choose $\mathscr{C} := \{z \in \mathbb{C}: |z| = R\}$ for any $R > 1$. Since $\hat\Phi$ is analytic in $\mathbb{C} \setminus [-1, 1]$ and $\hat\Phi(z) = \mathcal{O}(|z|^{-2})$ as $|z| \to \infty$, we conclude by Cauchy's residue theorem (at infinity) that
\begin{equation}\label{eq:hatPhi_Cauchy_0}
\int^{1}_{-1}\frac{\sqrt2 - \sqrt{A+Bt}}{\sqrt{1-t^2}}\mathrm{d}t 
=\frac{i}{2} \oint_{\mathscr{C}} \hat\Phi(z) dz = 0.
\end{equation}
{Finally, using that
\[
\int^{1}_{-1}\frac{\mathrm{d}t}{\sqrt{1-t^2}} 
= \pi,
\]
and putting $t=\sin \theta$ in \eqref{eq:hatPhi_Cauchy_0}, one recovers \eqref{eq:int_A_B_pi}.
}
\end{proof}
{
We now derive the whole information on the cut, as made explicit by the following result.
}
\begin{Prop}\label{prop: cut rho}
We have $A = B = \pi^2/4$. As a consequence, $\tilde{a} = \frac{\pi}{\sqrt2}$, $\tilde{b} = 0$, $a = (2-\pi)\sqrt2$, $b=2\sqrt2$. 
\end{Prop}

\begin{proof}
We start by plugging \eqref{eq: solution hat rho} into the second equation in \eqref{eq: resolvent equation hat rho}: 
\begin{equation}\label{noramlisation 1}
    \int^1_{-1} \frac{\mathrm{d}w}{f(w)} \, \dashint^1_{-1}f(s)\frac{\mathrm{d}s}{s-w} = \frac{\pi^2}{B},
\end{equation}

\noindent where $f(s) = \sqrt{\frac{A+Bs}{1-s^2}}$. We split the principal-value integral into 

\[\dashint^1_{-1}f(s)\frac{\mathrm{d}s}{s-w}
 = \int^{1}_{-1}\frac{f(s)-f(w)}{s-w}\mathrm{d}s + f(w)\dashint^1_{-1}\frac{\mathrm{d}s}{s-w}. \]

\noindent The first integral is not a principal-value integral since the fraction is continuously extended through $w$. The second term equals $f(w)\log\frac{1-w}{1+w}$ by definition of the principal-value integral. We can now re-express the left-hand side of \eqref{noramlisation 1} as 
\[\int^1_{-1} \frac{\mathrm{d}w}{f(w)}\dashint^1_{-1}f(s)\frac{\mathrm{d}s}{s-w} = \int_{-1}^1\int_{-1}^1\frac{f(s)-f(w)}{s-w}\frac{\mathrm{d}s\mathrm{d}w}{f(w)} + \int^1_{-1}\log\frac{1-w}{1+w}\mathrm{d}w. \]

\noindent The second integral vanishes. Moreover, switching the variables $s$ and $w$, we have
\[\int_{-1}^1\int_{-1}^1\frac{f(s)-f(w)}{s-w}\frac{\mathrm{d}z\mathrm{d}w}{f(w)} = \int_{-1}^1\int_{-1}^1\frac{f(w)-f(s)}{w-s}\frac{\mathrm{d}w\mathrm{d}s}{f(s)} = \frac{1}{2}\int_{-1}^1\int_{-1}^1\frac{f^2(s)-f^2(w)}{s-w}\frac{\mathrm{d}s\mathrm{d}w}{f(s)f(w)}. \]

\noindent Using the expression $f(s) = \sqrt{\frac{A+Bs}{1-s^2}}$, we can calculate the ratio $(f^2(s)-f^2(w))/(s-w)$ and rewrite \eqref{noramlisation 1} as 
\[\frac{B}{2\pi^2}\int_{-1}^1\int_{-1}^1\frac{A(s+{w})+B(s{w}+1)}{\sqrt{(A+Bs)(1-s^2)(A+Bw)(1-w^2)}}\mathrm{d}s\mathrm{d}w = 1. \]
The change of variables $s=\sin\theta$ and $w=\sin\varphi$ leads to
\[\frac{B}{2\pi^2}\int_{-\pi/2}^{\pi/2}\int_{-\pi/2}^{\pi/2}\frac{A(\sin\theta+\sin\varphi)+B(\sin\theta\sin\varphi+1)}{\sqrt{(A+B\sin\theta)(A+B\sin\varphi)}}\mathrm{d}\theta\mathrm{d}\varphi = 1. \]
On the other hand, we note that 
\begin{align*}
\frac{B}{2\pi^2}\int_{-\pi/2}^{\pi/2}\int_{-\pi/2}^{\pi/2}\frac{A(\sin\theta+\sin\varphi)+B\sin\theta\sin\varphi + A^2/B}{\sqrt{(A+B\sin\theta)(A+B\sin\varphi)}}\mathrm{d}\theta\mathrm{d}\varphi \\= \frac{1}{2\pi^2}\int_{-\pi/2}^{\pi/2}\int_{-\pi/2}^{\pi/2}\sqrt{A+B\sin\theta}\sqrt{A+B\sin\varphi}\mathrm{d}\theta\mathrm{d}\varphi = 1,
\end{align*}
{by \eqref{eq:int_A_B_pi} in \cref{lem:int_A_B_pi}. Hence subtracting the above two displays,}
\[(B^2-A^2)\int_{-\pi/2}^{\pi/2}\int_{-\pi/2}^{\pi/2}\frac{\mathrm{d}\theta\mathrm{d}\varphi}{\sqrt{(A+B\sin\theta)(A+B\sin\varphi)}} = 0. \]
As $A\geq |B|$ by the expression of $A$ and $B$, the integral is non-zero. We conclude that $A=B$ or $A=-B$. In any case, \eqref{eq:int_A_B_pi} further implies that $A = \pi^2/4$. Then, the value of $(\tilde{a}, \tilde{b})$ is either $(\pm\frac{\pi}{\sqrt{2}}, 0)$ or $(0, \pm\frac{\pi}{\sqrt{2}})$. Recalling $a = 2\sqrt2 - 2\tilde{a}$, $b = 2\sqrt2 - 2\tilde{b}$ and the condition $-b\leq a\leq 0\leq b$, we conclude that $\tilde{a} = \frac{\pi}{\sqrt{2}}, \tilde{b} = 0$, $a = (2-\pi)\sqrt{2}$, $b = 2\sqrt2$ and $A=B=\pi^2/4$.
\end{proof}

Using \cref{prop: cut rho}, we can simplify the formula for $\hat\rho(z)$ obtained in \cref{prop: solution hat rho}, which leads to
\begin{equation}\label{eq: formula hat rho}
\hat\rho(w) = \frac{\sqrt{1-w^2}}{2\pi}\dashint^1_{-1}\frac{1}{\sqrt{1-s}} \frac{\mathrm{d}s}{s-w} = \frac{\sqrt{1+w}}{2\pi}\log\bigg(\frac{\sqrt2+\sqrt{1-w}}{\sqrt2-\sqrt{1-w}}\bigg). 
\end{equation}

\noindent The explicit formula for $\tilde{\rho}$ is then 
\begin{equation}\label{eq: formula tilde rho}
\tilde{\rho}(v) = \frac{v}{\tilde{a}^2}\log \Big(\frac{v}{\tilde{a} - \sqrt{\tilde{a}^2-v^2}}\Big) = -\frac{v}{\tilde{a}^2}\log\Big(\frac{v}{\tilde{a} + \sqrt{\tilde{a}^2-v^2}}\Big),
\end{equation}
\noindent with $\tilde{a} = \pi/\sqrt2$. 
{
We may also recover $\rho$ by a simple affine transformation, but we will actually not need this.
}

\subsection{Expression and asymptotics of $F_{\ell}$}
\label{sec:expr_asympt_F_ell}

With the explicit formula of $\hat\rho$ and $\tilde{\rho}$ {in \eqref{eq: formula hat rho} and \eqref{eq: formula tilde rho}}, we can now {explicitly calculate $F_\ell$, and deduce the relevant asymptotic behaviours of $F_\ell$ and $W$}. The results are summarized {in the following proposition}. 

\begin{Prop} \label{prop:asympt_F_ell}
{The partition function can be expressed, for all $\ell \ge 1$, as
\begin{equation} \label{eq:expr_F_ell_prop}
F_{\ell} = 2(2\sqrt2)^{\ell}\int^1_0 u\bigg(1-\frac{\pi}{2}u\bigg)^{\ell}\log\bigg(\frac{1 + \sqrt{1-u^2}}{u}\bigg)\mathrm{d}u. 
\end{equation}
}
{In particular, as} $\ell\to\infty$, we have
\begin{equation}\label{eq: asymptotics F_ell}
F_{\ell} \sim \frac{8}{\pi^2}(2\sqrt2)^{\ell}\frac{\log\ell}{\ell^2}. 
\end{equation}

\noindent As $z\to 0^+$, we have 
\begin{equation}\label{eq: first order W}
W(2\sqrt2 + z)\longrightarrow c_0,
\end{equation}
\noindent with 
\begin{equation}\label{eq: value c_0}
c_0 = \frac{1}{\sqrt{2}}.
\end{equation}
\noindent Furthermore, 

\begin{equation}\label{eq: second order W}
W(2\sqrt2 + z)-c_0\sim -\frac{2}{\pi^2}z\log^2(z), \qquad z\to 0^+.
\end{equation}
\end{Prop}

\medskip
\noindent {The first two claims of the above proposition are nothing but the statement of \cref{thm:asympt_F_ell_intro}.
Noting from \cref{prop: cut rho} that} the right endpoint of the cut is $b=2\sqrt{2}$, the asymptotics \eqref{eq: asymptotics F_ell} match the form given by e.g.\ \cite[Equation (5.1)]{borot2012loop}, with the additional logarithmic correction that is expected from the critical case {($n=2$)}.

The rest of the section is devoted to the proof of \cref{prop:asympt_F_ell}, {which we divide into several steps.}

\begin{proof}[Proof of {\eqref{eq: first order W} and \eqref{eq: value c_0}}]
We first deal with the {behaviour of the} resolvent $W$ near $2\sqrt2+$. By \eqref{eq: resolvent out of density} and our definition of $\tilde{\rho}$ {in \eqref{eq: resolvent equation tilde rho}}, for $z\in\mathbb{C}\setminus [a, b]$, 
\begin{equation}\label{eq: resolvent out of tilde rho}
W(z) = \int^b_a\frac{\rho(y)}{z-y}\mathrm{d}y = \int^{\tilde{a}}_0 \frac{2\tilde{\rho}(v)}{z-2\sqrt2+2v}\mathrm{d}v. 
\end{equation}

\noindent Let $z = 2\sqrt{2} + u$ with $u>0$. As $u\to 0+$, we have {by monotone convergence, and using the exact expression of $\tilde{\rho}$ in \eqref{eq: formula tilde rho},} 
\[
W(2\sqrt2 + {u})\longrightarrow c_0
=
\int^{\tilde{a}}_0\frac{\tilde{\rho}(v)}{v}\mathrm{d}v
\overset{{\eqref{eq: formula tilde rho}}}{<}\infty. 
\]

\noindent The {expression \eqref{eq: value c_0}} for $c_0$ {is then determined by} the explicit formula for $\tilde{\rho}$ in \eqref{eq: formula tilde rho}. 
\end{proof}

\smallskip
\begin{proof}[Proof of \eqref{eq: second order W}]
For the second order estimate, we use again \eqref{eq: formula tilde rho} to calculate that for $u>0$, 
\[W(2\sqrt2+u) - c_0 = - \int^{\tilde{a}}_0 \frac{u}{(u+2v)v}\tilde{\rho}(v)\mathrm{d}v = \frac{u}{\tilde{a}^2}\int^{\tilde{a}}_0\frac{\mathrm{d}v}{u+2v}\log\bigg(\frac{v}{\tilde{a}+\sqrt{\tilde{a}^2 - v^2}}\bigg).\]

\noindent Changing variables with $v = uy$, this is 
\[W(2\sqrt2 + u) - c_0 = \frac{u}{\tilde{a}^2}\int^{\tilde{a}/u}_0\frac{\mathrm{d}y}{1+2y}\log\bigg(\frac{uy}{\tilde{a}+\sqrt{\tilde{a}^2 - u^2y^2}}\bigg).\]

\noindent Expanding the logarithm in $u$, it can be seen that as $u\to 0+$, 
\[
W(2\sqrt2 +u) - c_0 
\sim \frac{u}{\tilde{a}^2}\log(u)\int^{\tilde{a}/u}_0\frac{\mathrm{d}y}{1+2y}
\sim -\frac{u}{\tilde{a}^2}\log^2(u),
\]
and we conclude {by recalling from \cref{prop: cut rho} the value of $\tilde{a}=\frac{\pi}{\sqrt{2}}$}. 
\end{proof}

\smallskip
\begin{proof}[Proof of \eqref{eq:expr_F_ell_prop} and \eqref{eq: asymptotics F_ell}]
We finally turn to the partition function $F_\ell$. By \eqref{eq: resolvent out of tilde rho}, {for $z\in\mathbb{C}\setminus [a, b]$,}
\[W(z) = \int^{\tilde{a}}_0 \frac{2\tilde{\rho}(v)}{z-2\sqrt2 + 2v}\mathrm{d}v = \frac{2}{z}\int^{\tilde{a}}_0 \frac{\tilde{\rho}(v)}{1-\frac{2\sqrt2-2v}{z}}\mathrm{d}v. \]

\noindent Expanding $\frac{1}{1-X}$ in the integral {(taking $|z|$ large enough, say)}, $W(z)$ rewrites as  
\[W(z) = \frac{2}{z}\int^{\tilde{a}}_0\tilde{\rho}(v)\mathrm{d}v\sum^{\infty}_{\ell = 0}\Big(\frac{2\sqrt2-2v}{z}\Big)^{\ell} = \sum^{\infty}_{\ell = 0} \frac{1}{z^{\ell+1}} 2\int^{\tilde{a}}_0 \tilde{\rho}(v)(2\sqrt2-2v)^{\ell}\mathrm{d}v.  \]

\noindent Comparing with \eqref{eq: def resolvent}, we conclude that 
\begin{equation}\label{eq: formula F_ell}
F_{\ell} = 2\int^{\tilde{a}}_0 \tilde{\rho}(v)(2\sqrt2-2v)^{\ell}\mathrm{d}v. 
\end{equation}

\noindent Plugging \eqref{eq: formula tilde rho} into the above expression, 
\[F_{\ell} = -\frac{2}{\tilde{a}^2}\int^{\tilde{a}}_0 v(2\sqrt2-2v)^{\ell}\log\Big(\frac{v}{\tilde{a} + \sqrt{\tilde{a}^2-v^2}}\Big)\mathrm{d}v. \]

\noindent The change of variables $v = \tilde{a}y$ yields
\[F_{\ell} = -2(2\sqrt2)^{\ell}\int^1_0 y\bigg(1-\frac{\tilde{a}}{\sqrt2}y\bigg)^{\ell}\log\bigg(\frac{y}{1 + \sqrt{1-y^2}}\bigg)\mathrm{d}y, \]

\noindent {which is \eqref{eq:expr_F_ell_prop}.} Finally we change variables by $y = u/\ell$, the above expression boils down to
\[F_{\ell} = -\frac{2}{\ell^2}(2\sqrt2)^{\ell}\int^{\ell}_0 u\bigg(1-\frac{\tilde{a}}{\sqrt2}\frac{u}{\ell}\bigg)^{\ell}\log\bigg(\frac{u/\ell}{1 + \sqrt{1-(u/\ell)^2}}\bigg)\mathrm{d}u. \]

\noindent Expanding the logarithm provides that the leading order in the integral is 
\[\log(1/\ell)\int^{\ell}_0 v\bigg(1-\frac{\tilde{a}}{\sqrt2}\frac{v}{\ell}\bigg)^{\ell}\mathrm{d}v \sim \bigg(\int^{\infty}_{0} v\exp\Big(-\frac{\tilde{a}}{\sqrt2}v\Big)\mathrm{d}v\bigg)\log(1/\ell)
=-\frac{2}{\tilde{a}^2} \log \ell, \qquad \ell\to\infty. \]

\noindent {The asymptotic \eqref{eq: asymptotics F_ell}} is then a consequence {of \cref{prop: cut rho}}. 
\end{proof}

%
%

\section{From the partition function to geometric exponents}
\label{sec: geometric exponents}

In this section we combine the asymptotics \eqref{eq: asymptotics F_ell} of $F_\ell$ with the correspondence between infinite FK maps, inventory accumulation processes and the fully packed critical loop-$O(2)$ model, to obtain precise geometric scaling exponents. 
In \cref{ss:cluster} we describe the law of a \emph{typical cluster} in terms of the $O(2)$ weights and hitting times for the reduced walks (recall \cref{sec: walk definitions}). We go on in \cref{ss:hit} to derive asymptotics for these hitting times using \eqref{eq: asymptotics F_ell}. Finally, in \cref{sec: exponents loop and cluster}, we combine these results to give the tail exponents for the perimeter of a  {typical cluster} and the length of  a \emph{typical loop}, proving \cref{thm:main_exponents}.

\subsection{Law of a typical cluster}\label{ss:cluster}

 Recall that on the event that $X(0)=\rF$ we denote by $E=X(\varphi(0))\cdots X(0)$ the $\rF$-excursion encoding the envelope $\mathfrak{e}(0)$, and we write $\mathsf{sk}(E)$ for the skeleton of $E$ as defined in \eqref{eq:skeleton_fact}. In \cref{lem:skeleton=clusters}, we showed that the triangles in the infinite FK map that have an edge in $\mathfrak{c}(0)$ are in one-to-one correspondence with symbols in $\mathsf{sk}(E)$. {Moreover, under this correspondence, the triangles are explored consecutively in Sheffield's bijection when reading $\mathsf{sk}(E)$ from left to right. }

    We use this to describe the law of a $\mathfrak{c}(0)$ in terms of the critical $O(2)$ weights \eqref{eq: def weight triangulation} and the reduced walk hitting times defined in \cref{sss:dictionary}.
    By symmetry, we may assume that the boundary of $\mathfrak{c}(0)$ is primal (blue), or equivalently $X(\varphi(0)) = \rh$.

\begin{Prop}[Typical cluster probability] \label{prop: outer_bdry_prob}
Fix a rooted, loop-decorated triangulation $(\tfrak,\boldsymbol{\ell})\in \Tb_\ell$ with boundary length, $\ell \ge 1$. Let $\tN_{\ell+1}$ be the sum of $(\ell+1)$ i.i.d.\ geometric random variables with parameter $1/2$, {independent of $\tauc$}. Then 
\[\mathbb{P}(\mathfrak{c}(0) = \tfrak \mid X(0) = \rF, X(\varphi(0)) = \rh)
\quad \propto\quad 
Z(\tfrak,\boldsymbol{\ell}, x_c, 2) (2x_c)^{\ell+1}\mathbb{P}(\tauc > \tN_{\ell+1}).\] 
\end{Prop}

\begin{proof} 
By Sheffield's bijection, we may encode $\tfrak$ as a word $w$ with $\overline{w}=\emptyset$. {More precisely, Sheffield's bijection is between the word $w$ and the map $\tfrak'$ (in one-to-one correspondence with $\tfrak$) which is obtained from $\tfrak$ by joining all the vertices on the external face of $\tfrak$ to a single extra vertex.} See \cref{sec:O(2)_model}. On the other hand, conditioned on $X(0) = \rF$ and $X(\varphi(0)) = \rh$, the envelope $\mathfrak{e}(0)$ at $0$ is encoded by the $\rF$-excursion $E:=X(\varphi(0))\cdots X(0)$. By \cref{lem:skeleton=clusters}, the probability we are after is

\begin{equation} \label{eq:sk_weight_sum}
\mathbb{P}(\mathsf{sk}(E) = w \mid X(0) = \rF, X(\varphi(0)) = \rh)
= \sum_{e\in\mathcal{S}(w)} \mathbb{P}(E=e \mid X(0) = \rF, X(\varphi(0)) = \rh),
\end{equation}
where $\mathcal{S}(w)$ is the set of words $e$ which are $\rF$-excursions of type $\rh$ such that $\mathsf{sk}(e)=w$. 

We now split the sum as follows. Write $e\in \mathcal{S}(w)$ in the decomposition of \eqref{eq:dec_skeleton} as
\begin{equation} \label{eq:e_r(i)_s(i)}
e = \rh r(\ell)s(\ell)\cdots r(1)s(1)r(0)\rF, 
\end{equation}
where the $s(i)$ are either $\rh$, $\rH$ or a maximal $\rF$-excursion of type $\rc$, and the $r(i)$'s are (possibly empty) words in $\mathscr{A}_{\rc}$. Likewise, we write
\[
E = \rh R(\ell)S(\ell)\cdots R(1)S(1)R(0)\rF.
\]
As in \eqref{eq:dec_skeleton} we note that in the decomposition \eqref{eq:e_r(i)_s(i)},  $\mathsf{sk}(e) = s(\ell)\cdots s(1)$. Therefore the $s(i)$, $1\le i\le \ell$, are fixed by the condition that $\mathsf{sk}(e) = w$. 

We now take a look at the set of possible $r(0),\ldots,r(\ell) \in \mathscr{A}_{\rc}$ such that the associated word $e$ in \eqref{eq:e_r(i)_s(i)} satisfies $\mathsf{sk}(e) = w$. By definition, each $r(i)$, $0\le i\le \ell$, is a word in the alphabet $\mathscr{A}_{\rc}$ made of $\rc$, $\rC$ and maximal $\rF$-excursions of type $\rh$. The condition that $r(0),\ldots,r(\ell) \in \mathscr{A}_{\rc}$ satisfy $\mathsf{sk}(e) = w$ translates into the fact that $\tauc > \sum_{i=0}^{\ell} |r(i)|$, where $|r(i)|$ denotes the length of $r(i)$ seen as an element of $\mathscr{A}_{\rc}$. Moreover, $\tauc$ does not depend on the sequence $s(i)$, $1\le i\le \ell$. As a consequence, the whole sum in \eqref{eq:sk_weight_sum} is proportional to
\begin{multline} \label{eq:sk_weight_r(i)}
\mathbb{P}(\mathsf{sk}(E) = w \mid X(0) = \rF, X(\varphi(0)) = \rh) \\
\propto \quad \mathsf{w}(s(1))\cdots \mathsf{w}(s(\ell)) \sum_{r(0),\ldots,r(\ell)\in \mathscr{A}_{\rc}} \mathsf{w}(r(0))\cdots \mathsf{w}(r(\ell)) \mathds{1}_{\{\tauc>\sum_{i=0}^{\ell} |r(i)|\}},
\end{multline}
where $\mathsf{w}$ {of a sequence of letters (the ``hamburger-cheeseburger weight'') is given by the product of the probabilities in \eqref{eq: symbol weights}}. We emphasise that for \eqref{eq:sk_weight_r(i)} to make sense, we need to define the weight of the empty word (since any $r(i)$ could be empty): for \eqref{eq:sk_weight_r(i)} to hold, we have to take this weight to be $1$ (so that the corresponding term does not contribute to the sum). By the correspondence \eqref{eq:link_FK_Z}, $\mathsf{w}(s(1))\cdots \mathsf{w}(s(\ell)) \; \propto \; x_c^{\ell+1} Z(\tfrak,\boldsymbol{\ell}, x_c,2)$, whence
\begin{multline*}
\mathbb{P}(\mathsf{sk}(E) = w \mid X(0) = \rF, X(\varphi(0)) = \rh) \\
\propto \quad x_c^{\ell+1} Z(\tfrak,\boldsymbol{\ell}, x_c, 2) \sum_{r(0),\ldots,r(\ell)\in \mathscr{A}_{\rc}} \mathsf{w}(r(0))\cdots \mathsf{w}(r(\ell)) \mathds{1}_{\{\tauc>\sum_{i=0}^{\ell} |r(i)|\}}.
\end{multline*}

In addition, we can express the above sum using that
\[
\mathbb{P}\Big(\tauc>\sum_{i=0}^{\ell} |R(i)|\Big)
=
\frac{
\sum_{r(0),\ldots,r(\ell)\in \mathscr{A}_{\rc}} \mathsf{w}(r(0))\cdots \mathsf{w}(r(\ell)) \mathds{1}_{\{\tauc>\sum_{i=0}^{\ell} |r(i)|\}}}
{\sum_{r(0),\ldots,r(\ell)\in \mathscr{A}_{\rc}} \mathsf{w}(r(0))\cdots \mathsf{w}(r(\ell))}.
\]
The weight $\mathsf{w}(r(i))$ of $r(i)$ is obviously the same for each $i$, so that we can focus on $r(0)$.
Since $r(0)$ is a word in the alphabet $\mathscr{A}_{\rc}$, let us write $r(0)=y(k)\cdots y(1)$ with $y(1), \ldots, y(k) \in \mathscr{A}_{\rc}$. Then $\sum_{r(0)\in \mathscr{A}_{\rc}} \mathsf{w}(r(0)) = \sum_{k\ge 0} \sum_{y(1),\ldots,y(k)\in\mathscr{A}_{\rc}} \mathsf{w}(y(k)\cdots y(1))$. Furthermore, by symmetry between hamburgers and cheeseburgers, the total weight for each $y(i)$ is $1/2$. For $k=0$, the weight of the empty word is $1$ as we already discussed after display \eqref{eq:sk_weight_r(i)}. Therefore, we get that $\sum_{r(0)\in \mathscr{A}_{\rc}} \mathsf{w}(r(0)) = \sum_{k\ge 0} 2^{-k} = 2$, and hence
\[
\sum_{r(0),\ldots,r(\ell)\in \mathscr{A}_{\rc}} \mathsf{w}(r(0))\cdots \mathsf{w}(r(\ell))
=
2^{\ell+1}.
\]
As for the first term, namely $\mathbb{P}\big(\tauc>\sum_{i=0}^{\ell} |R(i)|\big)$, it suffices to notice that the $|R(i)|$, $0\le i\le \ell$, are precisely the independent geometric random variables introduced in the construction \eqref{eq:(h,c)_geo_construc}. 
This concludes the proof of \cref{prop: outer_bdry_prob}.
\end{proof}

\subsection{Hitting time estimates}\label{ss:hit}
We start with the tail behaviour of the hitting time of the one-dimensional (or "non-lazy") reduced random walk, as defined in \cref{sec: walk definitions}. 

\begin{Prop}[Hitting time $\tauh$]\label{prop: prob_tau}
We have 
\[\mathbb{P}(\tauh = \ell+1) = {\sqrt{2}} (2x_c)^{\ell+1}F_{\ell}. \]
\end{Prop}

\begin{proof}
Summing over $M\in \Tb_{\ell}$ in \cref{prop: outer_bdry_prob}, we get
\[\mathbb{P}(|\partial\mathfrak{c}(0)| = \ell \mid X(0) = \rF, X(\varphi(0)) = \rh)
\quad \propto\quad (2x_c)^{\ell+1}F_{\ell}\cdot\mathbb{P}(\tauc > \tN_{\ell+1}).\]
\noindent By \cref{prop: loop_bdry_length}, on the event that $X(0) = \rF$ and $X(\varphi(0)) = \rh$, $\thleft<\tcleft$ and $|\partial\mathfrak{c}(0)|$ {is equal to} $\tauh-1$. Hence 
\begin{equation} \label{eq:tau^h=ell_prop}
\mathbb{P}(\thleft< \tcleft, \tauh = \ell + 1) 
\quad \propto\quad (2x_c)^{\ell+1}F_{\ell}\cdot\mathbb{P}(\tauc > \tN_{\ell+1}).    
\end{equation}

{We now make use of the construction of $(\hleft,\cleft)$ in \eqref{eq:(h,c)_geo_construc}. We claim that, on the event $\{\tauh = \ell+1\}$, the event $\{\thleft<\tcleft\}$ is nothing but $\{\tauc>\tN_{\ell + 1}\}$. Indeed, recall that the geometric variables $\tG$, $i\ge 0$, are the lengths of the intervals of time where $\hleft$ stays put while $\cleft$ may move (i.e.~between $h$-steps). Therefore, if $\tauh=\ell+1$, then $\cleft$ hits $-1$ after $\hleft$ if, and only if, the non-lazy walk $c$ hits $-1$ after $\tN_{\ell+1}$.}

By independence in the construction \eqref{eq:(h,c)_geo_construc}, we may split the probability as
\begin{equation}\label{eq:tau^h=ell_construct}
\mathbb{P}(\thleft< \tcleft, \tauh = \ell+1) = \mathbb{P}(\tauh = \ell+1)\mathbb{P}(\tauc>\tN_{\ell + 1}). 
\end{equation}
\noindent Therefore we deduce from \eqref{eq:tau^h=ell_prop} and \eqref{eq:tau^h=ell_construct} that
\[
\mathbb{P}(\tauh = \ell+1) \quad \propto \quad (2x_c)^{\ell+1}F_\ell. 
\]

\noindent By \eqref{eq: first order W}, we have that $\sum (2x_c)^{\ell+1}F_\ell = \lim_{z\to 2\sqrt{2}+} W(z)= {1/\sqrt{2}}$ {(since $F_\ell \geq 0$ for all $\ell$)}. We conclude that
\[\mathbb{P}(\tauh = \ell+1) = {\sqrt{2}}(2x_c)^{\ell+1}F_{\ell}. \]
\end{proof}

{Recall from \eqref{eq: critical n  z} that $x_c=\frac{1}{4\sqrt{2}}$.} Combining \cref{prop: prob_tau} and \eqref{eq: asymptotics F_ell}, we get the following tail asymptotics for the hitting times. 
\begin{Cor}[Tail of hitting time $\tauh$ {and $\tauc$}]
\label{cor: tail_tau}
    \[
    \Pb(\tauh = \ell)
    \sim
    {\frac{4}{\pi^2}}
     \frac{\log \ell}{\ell^2}
    \quad \text{as } \ell\to \infty, 
    \] 
    In particular, this implies that
    \[
    \Pb(\tauh > \ell)
    \sim
    {\frac{4}{\pi^2}}
    \frac{\log \ell}{\ell}
    \quad \text{as } \ell\to \infty. 
    \]
  For the two dimensional walk $(\hleft, \cleft)$, we have
    \[
    \Pb(\thleft = \ell) \sim 
    {\frac{16}{\pi^2}}
   \frac{\log\ell}{\ell^2}\quad \text{as } \ell\to\infty. 
    \]
    The same asymptotics hold for $\tauc$ and $\tcleft$ by symmetry.
\end{Cor}

\begin{proof}
Note that the first two displayed equations follow easily from \cref{prop: prob_tau} and \eqref{eq: asymptotics F_ell}. It remains to prove the third asymptotic. 

Since $\thleft$ has the same distribution as $\tN_{\tauh} + \tauh$ we have
\begin{equation}\label{eq: tilde tau h}
\mathbb{P}(\thleft=\ell) = \mathbb{P}(\tN_{\tauh} + \tauh=\ell) = \sum^{\ell}_{k=1}\mathbb{P}(\tN_k = \ell - k)\mathbb{P}(\tauh=k). 
\end{equation}

We will show that the main contribution to the sum in \eqref{eq: tilde tau h} comes from $k$ close to $\ell/2$, where $\mathbb{P}(\tauh=k)$ has the correct asymptotic behaviour. 
More precisely, we claim that for any $0 < \varepsilon< 1/2$, there exists $\delta = \delta(\varepsilon) > 0$, such that for $|k - \ell/2| > \varepsilon\ell$
\begin{equation}\label{eq: Chernoff bound Nk.}
\mathbb{P}(\tN_k=\ell-k) \leq \mathrm{e}^{-\delta \ell}. 
\end{equation}
\noindent Indeed, when $k > (1/2 + \varepsilon)\ell$, we have $\mathbb{P}(\tN_k = \ell-k)\leq \mathbb{P}(\tN_{(1/2+\varepsilon)\ell} \leq (1/2 - \varepsilon)\ell)$. When $k < (1/2 - \varepsilon)\ell$, we have $\mathbb{P}(\tN_k = \ell-k)\leq \mathbb{P}(\tN_{(1/2-\varepsilon)\ell} \geq (1/2 + \varepsilon)\ell)$. The rest is a direct result of the Chernoff bound. Taking $0<t<\log 2$, by a standard Chernoff bound, 
\[
\mathbb{P}(\tN_k \geq (1+\varepsilon)k)\leq \mathrm{e}^{-t(1+\varepsilon)k}\mathbb{E}[\exp(t\tN_k)] = \Big((2-\mathrm{e}^t)\mathrm{e}^{(1+\varepsilon)t}\Big)^{-k}
\]
and
\[
\mathbb{P}(\tN_k \leq  (1-\varepsilon)k)\leq \mathrm{e}^{t(1-\varepsilon)k}\mathbb{E}[\exp(-t\tN_k)] = \Big(\mathrm{e}^{-t(1-\varepsilon)}(2-\mathrm{e}^{-t})\Big)^{-k}. 
\]
Taking $t$ small enough, we see that there exists $\delta'>0$ such that $\mathbb{P}(|\tN_k-k|>\varepsilon k)\leq \mathrm{e}^{- \delta'k}$. This finishes the proof of \eqref{eq: Chernoff bound Nk.}. 

By the asymptotics for $\mathbb{P}(\tauh = k)$ from the first item of \cref{cor: tail_tau}, when $\ell$ is large enough and $|k - \ell/2|
\leq \varepsilon\ell$ we have, {setting $c_1 = 4/\pi^2$,}
\[
\frac{c_1-\varepsilon}{(1/2+\varepsilon)^2}\frac{\log\ell}{\ell^2} \leq \mathbb{P}(\tauh = k)\leq \frac{c_1+\varepsilon}{(1/2-\varepsilon)^2}\frac{\log\ell}{\ell^2}. 
\]
Therefore
\begin{align*}
\frac{c_1-\varepsilon}{(1/2+\varepsilon)^2}\frac{\log\ell}{\ell^2}\sum_{|k - \ell/2|
\leq \varepsilon\ell} \mathbb{P}(\tN_k = \ell - k)&\leq 
\sum_{1\leq k\leq \ell}\mathbb{P}(\tN_k = \ell - k)\mathbb{P}(\tauh=k)\\
& \leq \ell \mathrm{e}^{-\delta(\varepsilon)\ell} + \frac{c_1+\varepsilon}{(1/2-\varepsilon)^2}\frac{\log\ell}{\ell^2}\sum_{|k - \ell/2| \leq  \varepsilon\ell} \mathbb{P}(\tN_k = \ell - k) 
\end{align*}
and also 
\[1-\sum_{|k - \ell/2| \leq  \varepsilon\ell}\mathbb{P}(\tN_k = \ell - k)\mathbb{P}(\tauh=k) \le \ell \mathrm{e}^{-\delta(\varepsilon)\ell}.
\]
By \eqref{eq: tilde tau h}, since we can take $\varepsilon$ arbitrarily small, we conclude the desired asymptotics. 
\end{proof}

\subsection{Exponents for the typical loop and cluster}\label{sec: exponents loop and cluster}

\paragraph{Typical cluster exponent.}
Since $\tauc$ and $\tauh$ have the same distribution, we may plug \cref{prop: prob_tau} back into \cref{prop: outer_bdry_prob} and deduce the tail asymptotics of $|\partial\mathfrak{c}(0)|$. 

\begin{Prop}[Typical cluster exponent]\label{cor: outer_bdry_expo}
We have
\[\mathbb{P}(|\partial\mathfrak{c}(0)| = \ell\mid X(0) = \rF) 
\sim  
{\frac{32}{\pi^4}}\frac{\log^2\ell}{\ell^3}
\quad \text{as } \ell\to\infty. 
\]
\end{Prop}

\begin{proof}
{We start as in the proof of \cref{prop: prob_tau}.}
Using first the symmetry of $\rc$ and $\rh$, we interpret the probability as {$\mathbb{P}(|\partial\mathfrak{c}(0)| = \ell\mid X(0) = \rF) = 
2\mathbb{P}(|\partial\mathfrak{c}(0)|= \ell, X(\varphi(0)) = \rh \mid X(0) = \rF)$. Then, by \cref{prop: loop_bdry_length} and \eqref{eq:tau^h=ell_construct},}
\begin{equation}\label{eq:outer_bdry_factors}
\mathbb{P}(|\partial\mathfrak{c}(0)| = \ell\mid X(0) = \rF) 
= 
2\mathbb{P}(\thleft< \tcleft, \tauh = \ell+1) 
= 
2\mathbb{P}(\tauh = \ell+1)\mathbb{P}(\tauc>\tN_{\ell + 1}). 
\end{equation}
\noindent \cref{cor: tail_tau} gives that $\mathbb{P}(\tauc > \ell+1) \sim {\frac{4}{\pi^2}} \frac{\log\ell}{\ell}$. {The law of large numbers for $\tN$ yields that $\tN_{\ell} \sim \ell$ as $\ell\to\infty$, almost surely.} By independence of $\tN$ and $\tauc$, we deduce that 
\[
\mathbb{P}(\tauc>\tN_{\ell + 1}) 
\sim 
{\frac{4}{\pi^2}} \mathbb{E}\bigg[\frac{\log \tN_{\ell+1}}{\tN_{\ell+1}}\bigg]
\sim
{\frac{4}{\pi^2}} \frac{\log\ell}{\ell}
\qquad \text{as } \ell\to\infty.
\]
This gives the asymptotic for the second factor on the right-hand side of \eqref{eq:outer_bdry_factors}. For the first one, we conclude by \cref{cor: tail_tau} again. 
\end{proof}

\paragraph{Typical loop exponent.}
Our next result deals with the length $|\mathfrak{L}(0)|$ of the typical loop $\mathfrak{L}(0)$, defined as the number of triangles it crosses. This is the analogue, for the critical FK model, of the first part of \cite[Theorem 1.1]{berestycki2017critical}. Note that the latter also contains a statement about the area {inside} $\mathfrak{L}(0)$. We should also be able to obtain such a statement from an asymptotic that we will derive in the next section, namely \eqref{eq:asympeta}. However this would require some more work, and since we will not need it in the proof of our main result, we will not pursue this direction. 

\begin{Prop}[Loop length exponent]\label{prop: loop_expo}
We have
\[
\mathbb{P}(|\mathfrak{L}(0)| = \ell \mid X(0)=\rF) \sim 
{\frac{256}{\pi^4}}
\frac{\log^2 \ell}{\ell^3} \qquad \text{as } \ell \to \infty. 
\]
\end{Prop}

\begin{proof}
Recall from \cref{prop: loop_bdry_length} that, on the event $X(0) = \rF$, {$|\mathfrak{L}(0)|= \tleft$}. By symmetry between $\rh$ and $\rc$, we have
\[
\mathbb{P}(|\mathfrak{L}(0)| = \ell \mid X(0)=\rF) 
= 
{2\mathbb{P}(|\mathfrak{L}(0)| = \ell, X(\varphi(0)) = \rh \mid X(0)=\rF)} 
= 
2\mathbb{P}(\thleft = \ell, {\thleft<\tcleft}). 
\]
Applying the coupling construction of \eqref{eq:(h,c)_geo_construc}, we see that on the event $\{\tauh = k\}$, $\{\thleft = \ell, \thleft < \tcleft\}$ is almost surely equal to $\{\tN_k = \ell-k, \tauc > \tN_k\}$. This follows from the same argument as in the proof of \cref{prop: prob_tau}. Recall that the geometric random variables $\tG_i$, $i\ge 1$, are the lengths of the intervals where $\th$ stays put while $\tc$ may move (i.e.~between $h$-steps). Therefore
\begin{align} 
\mathbb{P}(\thleft = \ell, {\thleft<\tcleft}) 
&= 
\sum^{\ell}_{k=1}\mathbb{P}(\tauh = k, \thleft = \ell, {\thleft<\tcleft}) \notag \\
&= 
\sum^{\ell}_{k=1}\mathbb{P}(\tauh = k, \tN_k = \ell-k,\tauc > \tN_k) \notag \\
&= 
\sum^{\ell}_{k=1}\mathbb{P}(\tauc > \ell-k)\mathbb{P}(\tN_k = \ell-k)\mathbb{P}(\tauh = k), \label{eq:Lfrak_0_sum}
\end{align}
\noindent by independence of $\tc$, $\th$ and $\tN$. Fix $0<\varepsilon<1/2$. {By \eqref{eq: Chernoff bound Nk.}, there exists $\delta > 0$, such that for $|k - \ell/2| > \varepsilon\ell$, $
\mathbb{P}(\tN_k=\ell-k) \leq \mathrm{e}^{-\delta \ell}$.} Then 
\begin{equation}\label{eq: proof loop expo 1}
\sum_{|k - \ell/2|> \varepsilon\ell}\mathbb{P}(\tauc > \ell-k)\mathbb{P}(\tN_k = \ell-k)\mathbb{P}(\tauh = k) \leq \sum_{|k - \ell/2|> \varepsilon\ell}\mathbb{P}(\tN_k = \ell-k) \leq \ell e^{-{\delta}\ell}.  
\end{equation}

\noindent On the other hand, the asymptotics of $\tauc$ in \cref{cor: tail_tau} imply that for $\ell$ large enough and $|k - \ell/2|\leq \varepsilon\ell$, {setting $c_1=4/\pi^2$,}
\begin{multline*}
\frac{c_1-\varepsilon}{1/2+\varepsilon}\bigg(\frac{\log\ell}{\ell} {+ \frac{\log (1/2-\varepsilon)}{\ell}}\bigg)
\leq 
(c_1-\varepsilon)\frac{\log (\ell - k)}{\ell - k} \\
\leq 
\mathbb{P}(\tauc > \ell-k) 
\leq 
(c_1+\varepsilon)\frac{\log (\ell - k)}{\ell - k}
\leq 
\frac{c_1+\varepsilon}{1/2-\varepsilon}\frac{\log \ell}{\ell}. 
\end{multline*}

\noindent Plugging this back into the summation, we have
\begin{multline}
\frac{c_1-\varepsilon}{1/2+\varepsilon}\bigg(\frac{\log\ell}{\ell} {+ \frac{\log (1/2-\varepsilon)}{\ell}}\bigg) \sum_{|k - \ell/2|\leq \varepsilon\ell}\mathbb{P}(\tN_k = \ell-k)\mathbb{P}(\tauh = k) \\
\leq 
\sum_{|k - \ell/2|\leq \varepsilon\ell}\mathbb{P}(\tauc > \ell-k)\mathbb{P}(\tN_k = \ell - k)\mathbb{P}(\tauh = k) \\
\leq
\frac{c_1+\varepsilon}{1/2-\varepsilon}\frac{\log\ell}{\ell}\sum_{|k - \ell/2|\leq \varepsilon\ell}\mathbb{P}(\tN_k = \ell-k)\mathbb{P}(\tauh = k). \label{eq: proof loop expo 2}
\end{multline}

\noindent The sum in the upper and lower bounds can be packed up again using $\tauc = k$ and $\tcleft = \ell$, up to a negligible term as in \eqref{eq: proof loop expo 1}:
\[\sum_{|k - \ell/2|\leq \varepsilon\ell}\mathbb{P}(\tN_k = \ell-k)\mathbb{P}(\tauh = k)
=
\sum^{\ell}_{k=1}\mathbb{P}(\tN_k = \ell-k)\mathbb{P}(\tauh = k) + O(\ell e^{-C\ell}) 
= 
\mathbb{P}(\thleft = \ell) + O(\ell e^{-C\ell}),
\]
as $\ell\to\infty$.
\noindent
{By \cref{cor: tail_tau}, $\mathbb{P}(\thleft = \ell)\sim 4c_1 \frac{\log\ell}{\ell^2}$ as $\ell\to\infty$.} Thus for $\ell$ large enough we have 
\begin{equation}\label{eq: proof loop expo 3}
(4c_1-\varepsilon)\frac{\log\ell}{\ell^2}\leq \sum_{|k - \ell/2|\leq \varepsilon\ell}\mathbb{P}(\tN_k = \ell-k)\mathbb{P}(\tauh = k) \leq (4c_1+\varepsilon)\frac{\log\ell}{\ell^2}. 
\end{equation}
\noindent Combining \eqref{eq:Lfrak_0_sum}, \eqref{eq: proof loop expo 1}, \eqref{eq: proof loop expo 2} and \eqref{eq: proof loop expo 3}, we get that for $\ell$ large enough, 
\[
\frac{c_1-\varepsilon}{1/2+\varepsilon}\bigg(\frac{\log\ell}{\ell} {+ \frac{\log (1/2-\varepsilon)}{\ell}}\bigg)(4c_1-\varepsilon)\frac{\log\ell}{\ell^2}
\leq 
\mathbb{P}(\thleft = \ell, {\thleft<\tcleft}) 
\leq 
\frac{(c_1+\varepsilon)(4c_1+\varepsilon)}{1/2-\varepsilon}\frac{\log^2\ell}{\ell^3}. 
\]
\noindent Since $\varepsilon$ can be chosen arbitrarily small, we conclude that
\[\mathbb{P}(|\mathfrak{L}(0)| = \ell \mid X(0)=\rF) = 2\mathbb{P}(\thleft = \ell<\tcleft) \sim 16c_1^2\frac{\log^2\ell}{\ell^3}
\qquad \text{as } \ell\to\infty. \]
{We conclude by recalling that $c_1 = 4/\pi^2$.}
\end{proof}

\cref{prop: loop_bdry_length} and \cref{cor: outer_bdry_expo} prove the first two assertions in \cref{thm:main_exponents}. The third will be proven in the next section.

\section{Exploration into the past} 
\label{sec:backward}

\subsection{The step distribution of the reduced walk}
\label{sec: step_rw}

Recall from \cref{par:red_walk} our definition of the (lazy) reduced walks $(\hleft,\cleft)$, and their (non-lazy) one dimensional projections $(\th,\tc)$.
Let $\xi$ be a random variable with law the step distribution of $\th$ (equivalently $\tc$). The aim of the present section is to deduce the tail asymptotics of $\xi$ from the expression \eqref{eq:expr_F_ell_prop} for $F_\ell$. 

In \cref{cor: tail_tau}, we derived the tail asymptotics of the hitting time $\tauh$. In fact, from the exact expression of its marginals (\cref{prop: prob_tau}) and the exact expression for $F_\ell$ (\cref{prop:asympt_F_ell}), one can actually deduce more precise asymptotics for the probability generating function
\begin{equation} \label{eq:def_r(s)}
r(s):=\mathbb{E}\big[s^{\tauh}\big] \qquad \text{as } s\nearrow 1.
\end{equation}

\begin{lemma}[Asymptotics of $r(s)$]
\label{L:r(s)}
As $s\nearrow 1$, we have the asymptotics
\[ 1-r(s)=a_1(1-s)\log^2(1-s)-a_2(1-s)\log(1-s)+o((1-s)\log(1-s)),\]
where
\[
a_1 := 
\frac{2}{\pi^2}
\quad \text{and} \quad
a_2 := 
\frac{4}{\pi^2}\log(\pi).
\] 
\end{lemma}

\begin{proof} 
By \cref{prop: prob_tau}, we observe that 
\[
r(s)
=
c_0^{-1} \sum_{\ell\ge 0} s^{\ell+1} (2x_c)^{\ell+1} F_{\ell}.
\]
The remainder of the proof is plain calculus, using the exact expression \eqref{eq:expr_F_ell_prop} for $F_\ell$. Indeed, recalling from \eqref{eq: critical n  z} that $2x_c = (2\sqrt{2})^{-1}$ and from \eqref{eq: value c_0} that $c_0 = 1/\sqrt2$, the latter expression yields 
\begin{align*}
r(s)
&=
  s \sum_{\ell \ge 0}  s^{\ell} \int_{0}^{1} u \log\left( \frac{1+\sqrt{1-u^2}}{u}\right) \big( 1-\pi u/2\big)^\ell \mathrm{d}u \\
&=
  s \int_{0}^{1} u \log\left( \frac{1+\sqrt{1-u^2}}{u}\right) \big( 1-s(1-\pi u/2)\big)^{-1} \mathrm{d}u.
\end{align*}
In particular, for $s=1$ we get $1 = \displaystyle\int_{0}^{1} u \log\left( \frac{1+\sqrt{1-u^2}}{u}\right) (\pi u/2)^{-1} \mathrm{d}u$. We deduce that
\begin{align}
1-r(s) 
&=
 \int_{0}^{1} u \log\left( \frac{1+\sqrt{1-u^2}}{u}\right) \Big( (\pi u/2)^{-1} - {s} \big( 1-s(1-\pi u/2)\big)^{-1} \Big) \mathrm{d}u \notag \\
&=
 \frac{2}{\pi} (1-s)  \int_{0}^{1} \log\left( \frac{1+\sqrt{1-u^2}}{u}\right) \frac{{\mathrm{d}u}}{1-s+\pi su/2}. \label{eq:1-r(s)_expansion}
\end{align}

Applying \cref{lem: integral expansion} in the Appendix with $\varepsilon = \frac{2(1-s)}{\pi s}$, we get 
\begin{align}
\int_{0}^{1} \log\left( \frac{1+\sqrt{1-u^2}}{u}\right) \frac{{\mathrm{d}u}}{1-s+\pi su/2} &= \frac{2}{\pi s}\Bigg(\frac{1}{2}\log^2\bigg(\frac{\pi s}{2(1-s)}\bigg) + \log (2) \log\bigg(\frac{\pi s}{2(1-s)}\bigg) + O(1)\Bigg) \notag \\
&= \frac{1}{\pi} \log^2(1-s) - \frac{2}{\pi}\log(\pi)\cdot\log(1-s) + o(\log(1-s)). 
\end{align}
The expansion in the statement of \cref{L:r(s)} is obtained by going back to \eqref{eq:1-r(s)_expansion}. 
\end{proof}

From {the expansion for the generating function of the hitting time in} \cref{L:r(s)}, it turns out that we can analyse the Laplace transform of $\xi$, 
\begin{equation} \label{eq:hat_F_xi}
{F}_\xi(\lambda):=\mathbb{E}\big[\mathrm{e}^{-\lambda\xi}\big], \quad \lambda>0. 
\end{equation}

\begin{Prop}[Asymptotics of $F_\xi(z)$]\label{prop:xi} 
 As $\lambda\searrow 0$, 
 \[
    {F}_\xi(\lambda)
    = 
    1 + \frac{\lambda}{a_1\log^2 \lambda} + \frac{4\lambda\log\log \frac{1}{\lambda}}{a_1\log^3\lambda} + \frac{2a_1\log a_1 + a_2}{a_1^2} \frac{\lambda}{\log^3 \lambda} + o\left(\frac{\lambda}{\log^3 \lambda}\right),
    \]
    with $a_1$, $a_2$ as in \cref{L:r(s)}.
\end{Prop}

\begin{proof} 
We follow the strategy of \cite[Proof of Theorem 1]{doney1982exact}. First, we relate the generating function $G(z)=\mathbb{E}[z^\xi]$ of $\xi$ and $r$ of \eqref{eq:def_r(s)}. Decomposing according to the first jump, we have for all $s\in(0,1)$,
\[
r(s) 
=
s\Pb(\xi=-1) + s\sum_{k\ge 1} \Pb(\xi=k) \mathbb{E}\big[s^{\tauh_{k+1}}\big],
\]
where $\tauh_k$ denotes the hitting time of $-k$ by $(\th_n, n\ge 0)$. Noting that $\tauh_{k+1}$ can be written as a sum of $(k+1)$ independent copies of $\tauh$, we end up with the identity
\[
r(s) 
=
s\Pb(\xi=-1) + s\sum_{k\ge 1} \Pb(\xi=k) r(s)^{k+1}
=
s r(s) G(r(s)).
\]
This entails that $G(r(s)) = 1/s$. Furthermore, $r$ is increasing on $(0,1)$, and such that $r(s)\to 1$ as $s\to 1$. As a consequence, it admits an inverse function $y\mapsto s(y)$ such that $s(y) \to 1$ as $y\to 1$. The previous identity then implies that 
\begin{equation} \label{eq:rel_Fxi_s}
    G(y) = \frac{1}{s(y)}.
\end{equation}

We now derive an asymptotic for $\tilde{s}(y):=1-s(y)$ as $y\to 1$. By \cref{L:r(s)}, 
\begin{equation} \label{eq:stilde_equiv1}
1-y = 1-r(s(y)) 
\sim a_1 \tilde{s}(y) \log^2(\tilde{s}(y)) 
\quad \text{as } y\to 1.
\end{equation}
Since $\tilde{s}(y)\to 0$ as $y\to 1$, we can take logarithms on both sides to obtain that
\[
\log \tilde{s}(y) 
\sim
\log(1-y) 
\quad \text{as } y\to 1.
\]
Plugging this back into \eqref{eq:stilde_equiv1}, we have
\begin{equation} \label{eq:stilde_equiv_final}
\tilde{s}(y) \sim \frac{1-y}{a_1 \log^2(1-y)} 
\quad \text{as } y\to 1.
\end{equation}

Next we define, for $x\in (1,+\infty)$,
\[
M(x):= \frac{1}{x(1-s(1-1/x))} \quad \text{and} \quad M^*(x):=\frac{1}{x(1-r(1-1/x))}.
\]
Since $r(s(y)) = y$, it is readily checked that $M(x)M^*(xM(x)) = 1$, and therefore 
\begin{equation} \label{eq:M_M*_ratio}
    M(x)=\frac{1}{M^*(xM(x))}.
\end{equation}
Using \eqref{eq:stilde_equiv_final}, one sees that $M(x) \sim a_1 \log^2(x)$ as $x\to\infty$. Thus, we can write $M(x) = a_1 \log^2(x)(1+m(x))$, where $m$ satisfies $m(x)\to 0$ as $x\to\infty$. On the other hand, by \cref{L:r(s)}, as $x\to\infty$, 
\[
\frac{1}{M^*(x)} = a_1 \log^2(x) + a_2\log(x)+{o(\log x)}.
\]
We deduce from the latter display and \eqref{eq:M_M*_ratio} that, as $x\to\infty$,
\begin{align*}
M(x) 
&=
a_1 \log^2(xM(x)) {+} a_2\log(xM(x))+a_3+o(1) \\
&= 
{a_1 \big(\log^2(x)+2\log(x)\log(M(x))\big) {+} a_2\log(x) +o(\log x).}
\end{align*}
By definition of $m$, this entails, as $x\to\infty$, 
\[
1+m(x)
=
{
1+2\frac{\log(M(x))}{\log(x)} {+} \frac{a_2}{a_1} \frac{1}{\log(x)}+o\bigg(\frac{1}{\log(x)}\bigg).
}
\]
Finally, using that $M(x) \sim a_1 \log^2(x)$ as $x\to\infty$, we can reduce the above expansion to
\[
m(x)
=
4\frac{\log\log(x)}{\log(x)}+ \frac{2\log(a_1) {+}\frac{a_2}{a_1}}{\log(x)} + o\bigg(\frac{1}{\log(x)}\bigg)
\qquad \text{as } x\to\infty.
\]

\noindent We conclude that
\[
M(x) = a_1 \log^2(x) \bigg( 1 +4\frac{\log\log(x)}{\log(x)}+ \frac{2\log(a_1) {+}\frac{a_2}{a_1}}{\log(x)} + o\bigg(\frac{1}{\log(x)}\bigg) \bigg)
\qquad \text{as } x\to\infty. 
\]

Recall the notation $\tilde{s}(y)=1-s(y)$. By definition of $M$, this translates into 
\[
\frac{1}{\tilde{s}(1-1/x)} 
= 
a_1 x \log^2(x) \bigg( 1 +4\frac{\log\log(x)}{\log(x)}+ \frac{2\log(a_1){+}\frac{a_2}{a_1}}{\log(x)} + o\bigg(\frac{1}{\log(x)}\bigg) \bigg)
\qquad \text{as } x\to\infty,
\]
and thence
\[
\tilde{s}(1-1/x)
=
\frac{1}{a_1 x \log^2(x)} \bigg( 1 -4\frac{\log\log(x)}{\log(x)}- \frac{2\log(a_1){+}\frac{a_2}{a_1}}{\log(x)} + o\bigg(\frac{1}{\log(x)}\bigg) \bigg)
\qquad \text{as } x\to\infty.
\]
Recalling the relation \eqref{eq:rel_Fxi_s} and that $\tilde{s}=1-s$, we get the following expansion of $G$:
\[
G(1-1/x)
=
1+ \frac{1}{a_1 x \log^2(x)} - 4 \frac{\log\log(x)}{a_1 x\log^3(x)} - \frac{2a_1\log(a_1){+}a_2}{a_1^2 x\log^3(x)} + o\bigg(\frac{1}{x\log^3(x)}\bigg)
\qquad \text{as } x\to\infty.
\]
Setting $z=1-\mathrm{e}^{-\lambda}$ (together with a few more expansions that we feel free to skip) gives the asymptotic for the Laplace transform $F_\xi$.
\end{proof}

\cref{prop:xi} yields the following.
\begin{Cor}[Tail asymptotics for $\xi$]
We have 
\begin{equation}
\label{eq:xi_centred}
\mathbb{E}[\xi]=0,
\end{equation}
and 
\begin{equation} \label{eq: tail xi}
\mathbb{P}(\xi\ge k) \sim 
\frac{\pi^2}{k\log^3(k)}
\quad \text{as } k\to\infty.
\end{equation}
\end{Cor}

\begin{remark}
    One could deduce \eqref{eq:xi_centred} from the fact that $\chi=2$ in Sheffield's notation, see \cite[Section 3.1]{SheffieldScott2016QGAI}. Note, however, that this would require the difficult implication of \cite[Lemma 3.1]{SheffieldScott2016QGAI}.
\end{remark}

\begin{proof}
{
We remark that \cref{prop:xi} entails $F_{\xi}(\lambda) = 1 + o(\lambda)$ as $\lambda\to 0$. Since $\xi \geq -1$, we may use \cite[Section 2]{bingham1974asymptotic} to deduce that $\xi$ has finite first moment and that $\Eb[\xi] = 0$, which proves~\eqref{eq:xi_centred}.}

  We now turn to \eqref{eq: tail xi}. Using \cite[Corollary 8.1.7]{bingham1989regular} and the Laplace transform asymptotics in \cref{prop:xi}, we get for free that 
    \begin{equation}\label{eq:Exik}
    \mathbb{E}[\xi \mathrm{1}_{\xi\ge k}] \sim \frac{1}{a_1\log^2 k},
    \end{equation}
    with $a_1 = 2/\pi^2$.
    The asymptotics of the distribution function in \eqref{eq: tail xi} are more delicate and require some de Haan theory, because our case is critical (which amounts to $\alpha=1$ in the notation of \cite[Corollary 8.1.7]{bingham1989regular}). In this case one needs to check an extra condition, which is recalled in the discussion following \cite[Corollary 8.1.7]{bingham1989regular}. Namely, one needs to check that, if we set $\Phi(y) := y(1-{F}_\xi(1/y))$, then for all $r \ge 1$,
    \begin{equation} \label{eq:deHaan}
    \frac{\log^3 y}{\pi^2}(\Phi(ry)-\Phi(y)) \rightarrow \log r, 
    \quad \text{as } y\to\infty.
    \end{equation}
   It thus remains to check that \eqref{eq:deHaan} holds, which follows from elementary calculations using the expansion in \cref{prop:xi}, that we feel free to skip.
\end{proof}

\begin{proof}[Proof of \cref{thm:main_exponents}]
Recall that the law of $\xi$ is the step distribution of the (non-lazy) reduced walk, whose steps correspond to the reduced lengths of words given by single letters or $\rF$-excursions. Conditioning on the event that the word is an $\rF$-excursion gives a reduced length with exactly the law of $|\partial \mathfrak{e}|$). It follows, for instance, that $\mathbb{P}(\xi\ge k)=\mathbb{P}(|\partial \mathfrak{e}|\ge k)$ for all $k\ge 2$. Thus \cref{cor: tail_tau} proves the third assertion of \cref{thm:main_exponents}. Together with \cref{prop: loop_bdry_length} and \cref{cor: outer_bdry_expo} this completes the proof of \cref{thm:main_exponents}.
\end{proof}


\subsection{Scaling limit of the reduced walks}\label{sec: scaling limit of rw}

We now derive scaling limits for the (lazy) reduced walks $(\hleft_n,\cleft_n,n\ge 0)$ and the (non-lazy) projections $(\th_n,n\ge 0)$ and $(\tc_n,n\ge 0)$. 

\paragraph{Scaling limit of the one-dimensional reduced walks $\th$ and $\tc$.} 
Recall that $\xi$ is the step distribution of $\th$ and $\tc$. The asymptotics \cref{prop:xi} for the Laplace transform of $\xi$ allow us to obtain the following scaling limit result for $\th$ (by symmetry, the same statement holds for $\tc$). 

\begin{Thm}[Scaling limit of $\th$ and $\tc$] \label{thm:hn_scaling}
We have 
\begin{equation} \label{eq:hn_sc_1}
\frac{\log^2(n)}{n} \th_n \overset{\mathrm{\Pb}}{\longrightarrow} 
-\frac{\pi^2}{2} \qquad \text{as } n\to\infty.
\end{equation}
Moreover, 
\begin{equation} \label{eq:hn_sc_2}
\frac{\log^3(n)}{n} \th_n + {\frac{\pi^2}{2}}\log n + {\pi^2}\log\log n  - {\pi^2 \log\Big(\frac{2}{\pi}\Big)} 
\overset{\mathrm{d}}{\longrightarrow}
\zeta \qquad \text{as } n\to\infty,
\end{equation}
where $\zeta$ is a $1$-stable random variable. More precisely, $\zeta$ has Laplace transform
\begin{equation} \label{eq:Lapl_zeta}
\mathbb{E}\big[\mathrm{e}^{-\lambda \zeta} \big]
=
\exp(\pi^2\lambda \log \lambda),
\qquad \lambda>0.
\end{equation}
\end{Thm}

\begin{proof}
This follows from the convergence of Laplace transforms, using \cref{prop:xi}. Indeed, since the increments of $\th$ are i.i.d.\ with law $\xi$, we have for all $\lambda>0$, 
\[
\mathbb{E}\Big[ \exp\Big(-\lambda \frac{\log^3(n)}{n} \th_n\Big)\Big]
=
\mathbb{E}\Big[ \exp\Big(-\lambda \frac{\log^3(n)}{n} \xi\Big)\Big]^n.
\]
By \cref{prop:xi}, 
 we therefore get, as $n\to\infty$,
\begin{multline} \label{eq:hn_expansion}
\mathbb{E}\Big[ \exp\Big(-\lambda \frac{\log^3(n)}{n} \th_n\Big)\Big] 
= 
\bigg(1 + \frac{\lambda \log^3 n}{a_1 n \log^2 \big(\frac{\lambda\log^3 n}{n}\big)}
+ \frac{4\lambda \log^3 n}{a_1 n} \frac{\log\log \big(\frac{n}{\lambda \log^3 n}\big)}{\log^3\big(\frac{\lambda \log^3 n}{n} \big)} \\
+ \frac{2a_1\log a_1 {+} a_2}{a_1^2} \frac{\lambda \log^3 n}{n \log^3 \big(\frac{\lambda\log^3 n}{n}\big)} 
+ o \bigg(\frac{1}{n}\bigg) \bigg)^n.
\end{multline}
The last two terms give 
\[
\frac{4\lambda \log^3 n}{a_1 n} \frac{\log\log \big(\frac{n}{\lambda \log^3 n}\big)}{\log^3\big(\frac{\lambda \log^3 n}{n} \big)}
=
- \frac{4\lambda \log \log n}{a_1 n} + o\bigg(\frac{1}{n}\bigg),
\]
and
\[
\frac{2a_1\log a_1{+}a_2}{a_1^2} \frac{\lambda \log^3 n}{n \log^3 \big(\frac{\lambda\log^3 n}{n}\big)}
=
- \frac{2a_1\log a_1{+}a_2}{a_1^2 n} \lambda + o\bigg(\frac{1}{n}\bigg).
\]
A back-of-the-envelope calculation shows that the first term has the expansion 
\[
\frac{\lambda \log^3 n}{a_1 n \log^2 \big(\frac{\lambda\log^3 n}{n}\big)}
=
\frac{\lambda \log n}{a_1 n} + \frac{6\lambda \log \log n}{a_1 n} + \frac{2\lambda \log \lambda}{a_1 n}  + o\bigg( \frac{1}{n} \bigg).
\]
Putting these three expansions back into \eqref{eq:hn_expansion}, we obtain
\begin{multline} \label{eq:xi_expansion}
\mathbb{E}\Big[ \exp\Big(-\lambda \frac{\log^3(n)}{n} \th_n\Big)\Big] \\
= 
\bigg(1 + \frac{\lambda \log n}{a_1 n} + \frac{2\lambda \log \log n}{a_1 n}  + \frac{1}{n}\bigg(\frac{2\lambda \log (\lambda/a_1)}{a_1} {-} \frac{a_2}{a_1^2 } \lambda\bigg) 
+ o \bigg(\frac{1}{n}\bigg) \bigg)^n \\
= 
\exp\bigg(\frac{\lambda \log n}{a_1} + \frac{2\lambda \log \log n}{a_1}  + \Big(\frac{2\lambda \log (\lambda/a_1)}{a_1} {-} \frac{a_2}{a_1^2 } \lambda\Big)
+ o (1) \bigg).
\end{multline}
This establishes \eqref{eq:hn_sc_1}, and the convergence 
\[
\frac{\log^3(n)}{n} \th_n + \frac{\log n}{a_1} + \frac{2}{a_1}\log\log n - \frac{2a_1\log a_1 {+} a_2}{a_1^2}
\overset{\mathrm{d}}{\longrightarrow} \zeta,
\]
towards the random variable $\zeta$ with Laplace transform
\[
\mathbb{E}\big[\mathrm{e}^{-\lambda \zeta} \big]
=
\exp \Big(\frac{2\lambda \log \lambda}{a_1} \Big),
\quad \lambda>0.
\]
This Laplace transform is that of a $1$-stable random variable, see e.g.\ \cite[Chapter 3, Theorem 14.11]{sato1999levy}. 
Tracing the constants (see \cref{L:r(s)}), this concludes the proof of \cref{thm:hn_scaling}.
\end{proof}

\paragraph{Scaling limit of the lazy reduced walk $(\hleft,\cleft)$.}
Since the walks $\hleft$ and $\cleft$ move at different times, we can transfer  the one-dimensional scaling limit result of \cref{thm:hn_scaling} at almost no cost to the lazy reduced walk $(\hleft,\cleft)$. Note that in the result for $(\hleft,\cleft)$ below we multiply by $2$ in the rescaling, to account for the difference in speed and recover the same stable process in the limit.

\begin{Thm}[Scaling limit of $(\hleft,\cleft)$] \label{thm:lazy_scaling}
    Let 
    \[
    {\Hleft_n := \frac{2\log^3(n)}{n} \hleft_n + \frac{\pi^2}{2} \log n + \pi^2\log\log n - \pi^2 \log\Big( \frac{2}{\pi}\Big)}
    \]
    and
    \[
    {\Cleft_n := \frac{2\log^3(n)}{n} \cleft_n +  \frac{\pi^2}{2} \log n + \pi^2\log\log n - \pi^2 \log\Big( \frac{2}{\pi}\Big).}
    \]
    Then
    \begin{equation} \label{eq:lazy_sc}
    (\Hleft_n,\Cleft_n) \overset{\mathrm{d}}{\to} (\zeta,\zeta'), \quad \text{as } n\to\infty,
    \end{equation}
where $\zeta$ and $\zeta'$ are independent $1$-stable random variables with law as in \eqref{eq:Lapl_zeta}. 
\end{Thm}

\begin{proof}
Let $\lambda,\mu>0$. Then for all $n\ge 1$, since $\hleft$ and $\cleft$ move at different times, each of them with probability $1/2$ at each step,
    \[
    \mathbb{E}\Big[ \exp\Big(-\lambda \frac{2\log^3(n)}{n} \hleft_n - \mu \frac{2\log^3(n)}{n} \cleft_n\Big)\Big]
    =
    \mathbb{E}\Big[ \frac12\exp\Big(-\lambda \frac{2\log^3(n)}{n} \xi\Big) + \frac12\exp\Big(-\mu \frac{2\log^3(n)}{n} \xi\Big)\Big]^n. 
    \]
We use the expansion derived in the second line of \cref{eq:xi_expansion}, yielding
\begin{multline*}
    \mathbb{E}\Big[ \exp\Big(-\lambda \frac{2\log^3(n)}{n} \hleft_n - \mu \frac{2\log^3(n)}{n} \cleft_n\Big)\Big] = 
    \bigg(1 + \frac{(\lambda+\mu) \log n}{a_1 n} + \frac{2(\lambda+\mu) \log \log n}{a_1 n}  \\
    + \frac{1}{n} \bigg(\frac{2\lambda \log (2\lambda)}{a_1} +\frac{2\mu \log (2\mu)}{a_1} 
    {-} \frac{a_2 {+} 2a_1\log a_1}{a_1^2 } (\lambda+\mu)\bigg) + o \bigg(\frac{1}{n}\bigg) \bigg)^n. 
\end{multline*}

\noindent We thus obtain 
\begin{multline*}
    \mathbb{E}\Big[ \exp\Big(-\lambda \frac{2\log^3(n)}{n} \hleft_n - \mu \frac{2\log^3(n)}{n} \cleft_n\Big)\Big] = \\
    \exp\bigg( \frac{(\lambda+\mu) \log n}{a_1} + \frac{2(\lambda+\mu) \log \log n}{a_1}  + \bigg(\frac{2\lambda \log ({2}\lambda)}{a_1} +\frac{2\mu \log ({2}\mu)}{a_1} {-} \frac{a_2 {+} 2a_1\log a_1}{a_1^2 } (\lambda+\mu)\bigg) + o (1) \bigg).
\end{multline*}
This proves that $(\Hleft_n,\Cleft_n)$ converges in distribution as $n\to \infty$ to a pair $(\zeta,\zeta')$ of independent {random variables}, which both have the law described in \eqref{eq:Lapl_zeta}. Replacing the constants $a_1$ and $a_2$ with their values in \cref{L:r(s)}, this concludes the proof of \cref{thm:lazy_scaling}.
\end{proof}

\paragraph{Geometric interpretation (heuristic).} 
We now argue that \cref{thm:lazy_scaling} can be seen as a scaling limit result for the  primal and dual boundary length processes, when exploring towards a marked point in the infinite FK(4) random planar map. The lazy reduced walks $\hleft$ and $\cleft$ are indeed counting respectively the primal and dual boundary lengths of the component containing the triangle at $0$ {at each time in Sheffield's exploration of the map}, seen from  infinity (see \cref{fig:half-plane}). Whenever the lazy walk encounters a symbol $\rF$, it jumps by an amount equal to the perimeter of the corresponding envelope, which gets disconnected from the triangle at $0$. {In the critical mating-of-trees coupling \cite{AHPS}, the limiting pair $(\zeta,\zeta')$ in \cref{thm:lazy_scaling} comes from subordinating planar Brownian motion on the inverse local time of excursions above the running infimum of its imaginary part. In this coupling, the two processes describe the boundary lengths in the uniform $\mathrm{CLE}_4$-exploration of a critical $\gamma=2$ quantum {cone} towards a Liouville-typical point.}\footnote{Technically the results of \cite{AHPS} are stated for quantum discs, but they are expected to extend to other types of quantum surfaces.}
Therefore we may interpret \cref{thm:lazy_scaling} as a scaling limit result for the {primal and dual} boundary length process in the critical FK map, to their (conjectural) continuum analogues. A convenient way to see this is to embed the portion of the map from $-\infty$ to $0$ in the half-plane, as in \cref{fig:half-plane}.

\begin{figure} 
  \bigskip
  \centering
  \begin{subfigure}[b]{0.9\textwidth}
  \centering
    \includegraphics[page=3,width=\textwidth]{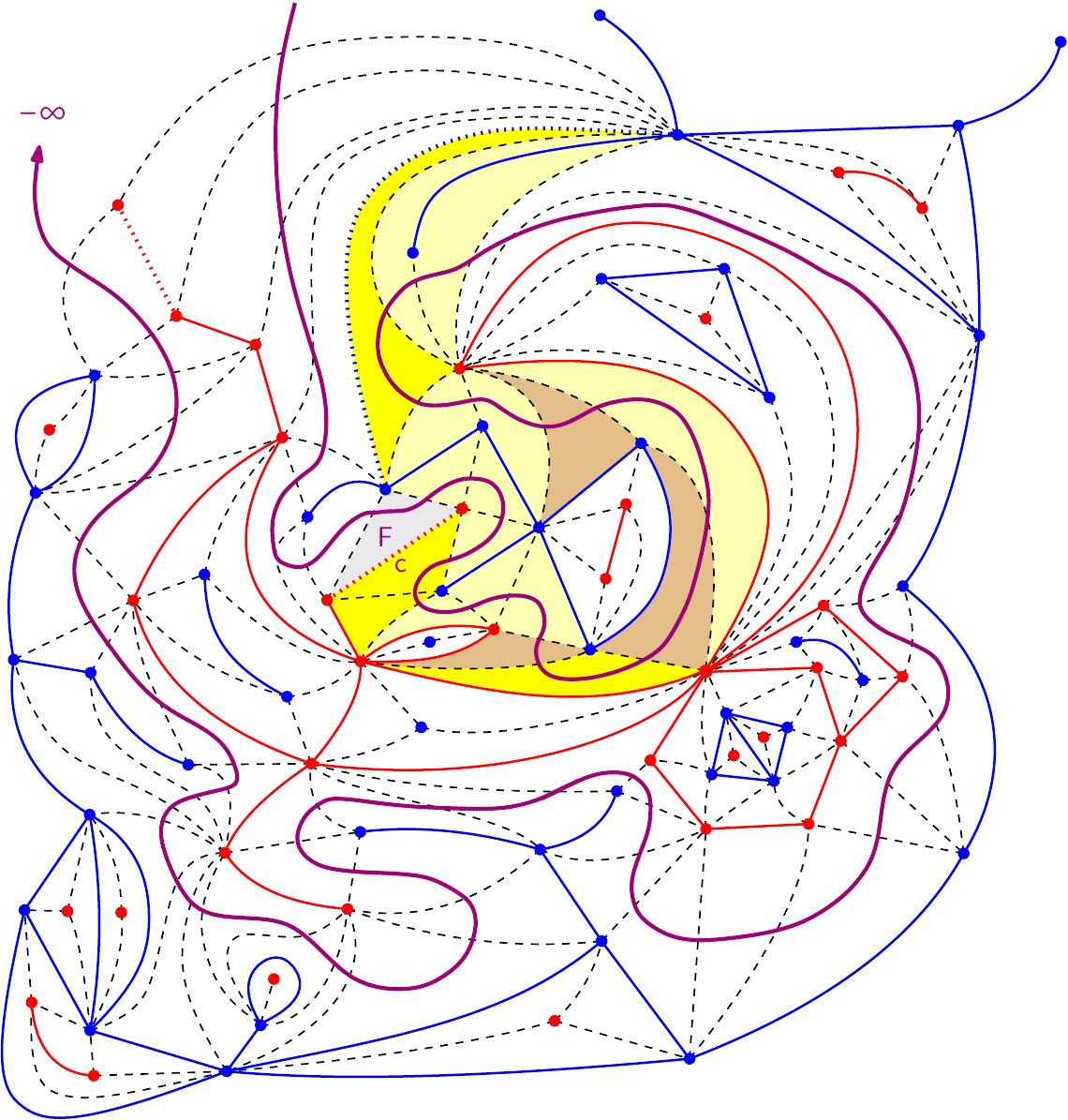}
    \caption{Sheffield's exploration towards the triangle at $0$. The picture is zoomed out from \cref{fig:outer-bdry-sheffield}. The triangle at $0$ is in grey, and we represented in purple Sheffield's exploration towards this triangle, which does not explore the components disconnected from it (we only look at the past of the exploration before time $0$).}
  \end{subfigure}
\end{figure}%
\begin{figure}[htb]
\ContinuedFloat
\centering
  \begin{subfigure}[b]{0.9\textwidth}
  \vspace{-0.6cm}
  \centering
    \includegraphics[width=\textwidth]{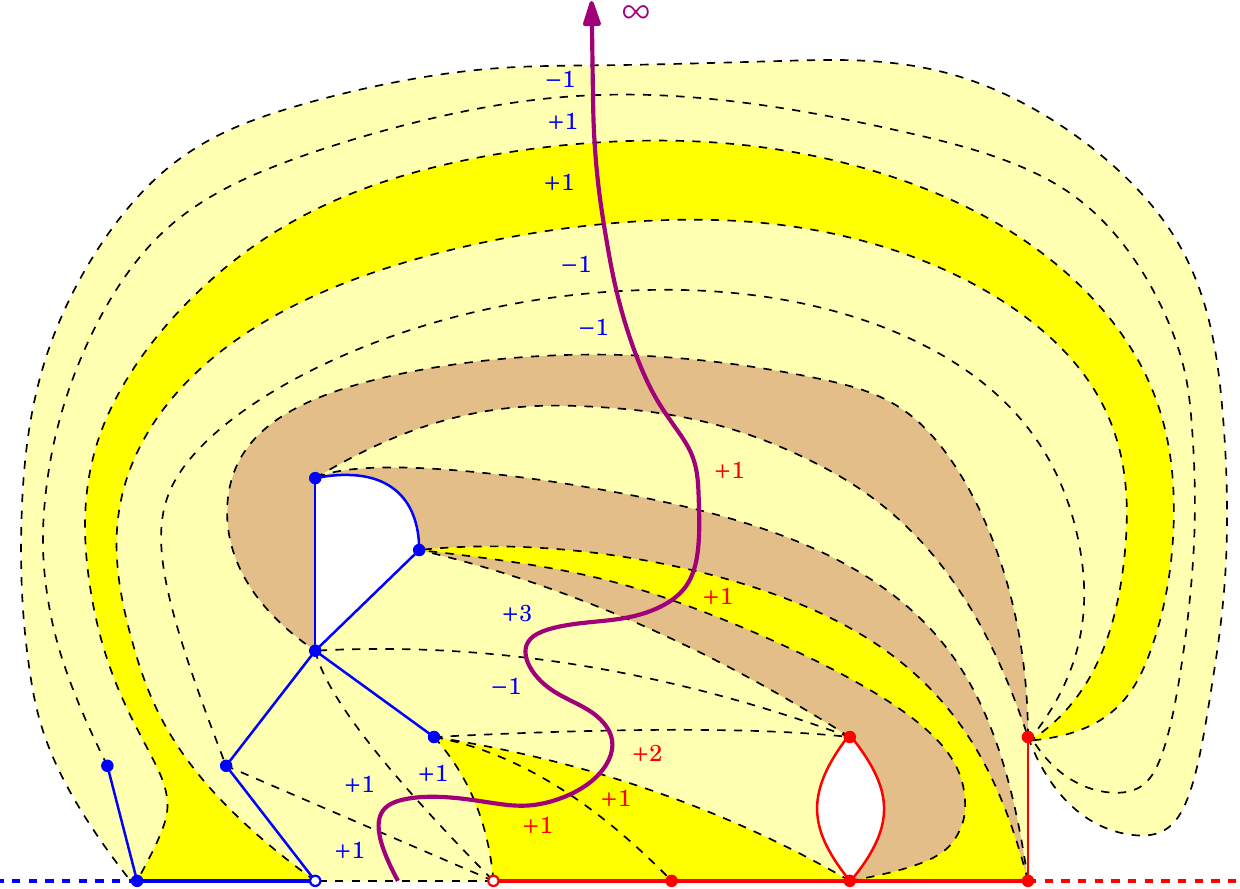}
    \caption{Half-plane representation of the portion of the map between $-\infty$ and $0$. To construct the half-plane map, we start from the blue and red vertices $\textcolor{blue}{\circ}$ and $\textcolor{red}{\circ}$ on the line, connected with dashed lines, and grow triangles (yellow) following Sheffield's exploration in the past, started from the triangle at $0$. Whenever we meet a triangle whose match is not in $(-\infty,0)$, we place it on the boundary (bright yellow). Triangles corresponding to the contour of envelopes that are disconnected from $0$ (brown) are skipped since each envelope is attached all at once in the exploration. The increments of the reduced hamburger and cheeseburger walk are indicated along the exploration, in blue and red respectively. These increments correspond, respectively, to the net change in left and right boundary lengths at each step.}
  \end{subfigure}
  \caption{A geometric interpretation of the hamburger/cheeseburger reduced walk.}
  \label{fig:half-plane}
\end{figure}


\subsection{Joint law with the duration}
\label{sec: joint laws}

In the previous subsection we obtained a scaling limit for $(\hleft_n,\cleft_n)$, which by definition is a random time change of the reverse time hamburger-cheeseburger count $(\cH_{-k},\cC_{-k})$, whose sum and discrepancy we wish to understand for our main theorem \cref{thm: main intro}. To this end, we will look more closely at the behaviour of this time change. We define $\sigma:\mathbb{N}\to \mathbb{N}$ by 
\begin{equation} \label{eq:time_change_cal}
(\hleft_n,\cleft_n)=(\cH_{-\sigma(n)},\cC_{-\sigma(n)}), \quad n\ge 0.
\end{equation}
Recall from \cref{sec: walk definitions} that each increment of $(\hleft,\cleft)$ corresponds to a word in the maximal excursion decomposition $X(-\infty,0) = \cdots Y(2)Y(1)X(0)$ which is either a single symbol $\rc,\rC,\rh,\rH$ (in which case the increment of $(\hleft,\cleft)$ is $(0,-1),(0,1),(-1,0),(+1,0)$ respectively) or an $\rF$-excursion $E$, in which case $\hleft$ (respectively $\cleft)$ increases by $|\overline{E}|$ if $E$ is of type $\rc\dots \rF$ (respectively $\rc\dots \rF$). It follows from the definition that the corresponding increment of $\sigma$ will be $1$ in the case of a single symbol, and will be $|E|$ in the case of an $\rF$-excursion $E$. Thus, $((\hleft_n,\cleft_n,\sigma(n)), n\ge 0)$ is a random walk, which gives the reduced length and duration of each word reading backwards in the maximal excursion decomposition. In the case where the word is an $\rF$-excursion, the reduced length corresponds to the boundary length of the corresponding envelope $\mathfrak{e}$ in the planar map, and the duration corresponds to the area of $\mathfrak{e}$ (i.e.~the number of triangles it contains). 
We will use the notation $(\xi,\eta)$ for a pair with the joint law of $(\hleft_1+\cleft_1, \sigma(1))$ and also use the representation 
\begin{equation}\label{eq:etaxisum}
(\hleft_n+\cleft_n,\sigma(n))=\Big(\sum_{i=1}^n \xi_i,\sum_{i=1}^n \eta_i\Big),
\end{equation}
where $(\xi_i,\eta_i, i\ge 1)$ are i.i.d.~with joint law $(\xi,\eta)$. Note that this $\xi$ has the same law as $\xi$ considered in \cref{sec: step_rw}.

The following identity concerning the joint law of $(\xi, \eta)$ will allow us to deduce asymptotics for their joint Laplace transform. 
Let
\[ \mathrm{e}^{-f(t)}=\mathbb{E}[\mathrm{e}^{-tT}] = \frac{\mathrm{e}^{-t}}{1 + \sqrt{1 - \mathrm{e}^{-2t}}},\] 
where $T$ is the hitting time of $-1$ for a simple symmetric random walk. (Note that $f(\delta)/\sqrt{2\delta}\to {1}$ as $\delta\searrow 0$.) 

\begin{lemma}\label{lem:joint_xi_eta}
For $t\geq 0$, we have
\[\mathbb{E}[\exp(-t\eta-f(t)\xi)] = 1. \]
\end{lemma}

\begin{proof}
Let $\cS_{-n}=-\cS(-n,-1)$ be as defined in \eqref{eq:def_calS_calD}, i.e.~the number of orders minus the number of burgers in the reduced word $\overline{X(-n,-1)}=\overline{X(-n)\ldots X(-1)}$. Then $(\cS_{-n})_{n\geq 0}$ has the law of a simple random walk. Denote the hitting time of $-1$ for $(\cS_{-n})_{n\geq 0}$ by $T$. 

By the definition of $(\hleft_n+\cleft_n, \sigma(n))_{n\geq 0}$ a.s.\ we have 
\begin{equation}\label{eq: time correspondence}
\hleft_n+\cleft_n = \cH_{-\sigma(n)}+\cC_{-\sigma(n)}=\mathcal{S}_{-\sigma(n)}, \quad  n\geq 0.
\end{equation}

\noindent Furthermore, $\cS$ cannot hit $-1$ at a time $k$ with $\sigma(m)<k<\sigma(m+1)$ for some $m$. Indeed, if $\sigma(m+1)\ne \sigma(m)+1$, then $Y(m)$ is an $\rF$ excursion, which means that $\cH_{-k}+\cC_{-k}\ge \cH_{-\sigma(m)}+\cC_{-\sigma(m)}$ for all $k=\sigma(m),\dots, \sigma(m+1)-1$.
This observation together with \eqref{eq: time correspondence} implies that $\sigma(T_1)=T$ a.s.\ where $T_1$ is the first time that $\hleft_n+\cleft_n$ hits $-1$.
Consequently, we have $\mathbb{E}[\exp(-t\sigma(T_1)])= \exp(-f(t))$. 

The lemma is then obtained by conditioning on the first step of $(\hleft_n+\cleft_n,\sigma(n))$. Given the first step $(\hleft_1+\cleft_1,\sigma(1))=:(\xi_1,\eta_1)$, the conditional distribution of $\sigma(T_1)$ is that of $\eta_1 + \sigma(T_{1+\xi_1})$, where $T_k$ is the first hitting time of $-k$ by $(\hleft_n+\cleft_n)$. We therefore obtain 
\[\mathbb{E}[\exp(-t\sigma(T_1))\mid (\xi_1,\eta_1)] = \exp(-t\eta_1)\mathbb{E}[\exp(-t\sigma({T_{1+\xi_1}}))\mid \xi_1] = \exp(-t\eta_1-(1+\xi_1)f(t)). \]

\noindent The proof is finished by taking expectation on both sides and using that $(\xi_1,\eta_1)$ is equal in distribution to $(\xi,\eta)$ in the statement of the lemma.
\end{proof}

We know from \cref{thm:lazy_scaling} in the previous section that, using the representation \eqref{eq:etaxisum}, 
\[
\frac{\sum_{i=1}^n\xi_i}{n/\log^2 n}
\to {-\frac{\pi^2}{2}} 
\]
in probability as $n\to \infty$.
This implies that
\[
\exp\bigg(-\lambda \frac{\sum_{i=1}^n\xi_i}{n/\log^2 n}\bigg)\to 
\exp\Big({\frac{\pi^2}{2}}\lambda\Big) 
\]
in probability and $L^p$ for all $p\ge 0$ as $n\to \infty$. To justify the $L^p$ convergence it is enough to show that $\mathbb{E}[\exp(-q\lambda \frac{\sum_{i=1}^n\xi_i}{n/\log^2 n})]$ is uniformly bounded in $n$ for some $q>p$, since this implies that
the family 
\[
\left(\exp\left({-p\lambda \frac{\sum_{i=1}^n\xi_i}{n/\log^2 n}}\right), n\ge 1\right)
\]
is uniformly integrable.
The first claim follows from the convergence of  $\mathbb{E}[\exp(-q\lambda \frac{\sum_{i=1}^n\xi_i}{n/\log^2 n})]$ established in \cref{thm:hn_scaling}.
We now upgrade this scaling limit result into a statement joint with the duration.

\begin{Prop} \label{lem:sc_duration_length}
For $t\ge 0$ and $s>\sqrt{2t}$ we have 
\[
\mathbb{E}\bigg[\exp\bigg({-t\frac{\sum_{i=1}^n \eta_i}{n^2/\log^4 n}-s\frac{\sum_{i=1}^n \xi_i}{n/\log^2 n}}\bigg)\bigg]\to 
\exp\Big({\frac{\pi^2}{2}}(s-\sqrt{2t})\Big)
\quad 
\text{as } n\to \infty.
\]
In particular, 
\[
\Big(\frac{\sum_{i=1}^n \eta_i}{n^2/\log^4 n},\frac{\sum_{i=1}^n \xi_i}{n/\log^2 n}\Big)\overset{\textnormal{d}}{\longrightarrow} \Big(\sigma, -{\frac{\pi^2}{2}}\Big)
\]
as $n\to \infty$, where $\sigma$ is the hitting time of $-{\pi^2/2}$ for a Brownian motion.
\end{Prop}

\begin{proof}
For brevity we set $k_n = \frac{\log^2 n}{n}$. For $t\ge 0$ and $n\geq 1$, by \cref{lem:joint_xi_eta} and independence, 
\[
\mathbb{E}\Big[\exp\Big(-tk_n^2\sum_{i=1}^n \eta_i-f(tk_n^2)\sum_{i=1}^n \xi_i\Big)\Big] = 1. \]
Therefore, taking $s>\sqrt{2t}$, we can write
\begin{multline*}
\mathbb{E}\Big[\exp\Big(-tk_n^2\sum_{i=1}^n \eta_i-sk_n\sum_{i=1}^n \xi_i\Big)\Big] - \exp\Big(\frac{\pi^2}{2}(s-\sqrt{2t})\Big)\\
= \mathbb{E}\bigg[\exp\Big(-tk_n^2\sum_{i=1}^n \eta_i-f(tk_n^2)\sum_{i=1}^n \xi_i\Big)\cdot \bigg(\exp\Big( (f(tk_n^{{2}})-sk_n)\sum_{i=1}^n \xi_i\Big) - \exp\Big(\frac{\pi^2}{2}(s-\sqrt{2t})\Big)\bigg)\bigg]. 
\end{multline*}

\noindent The last display converges to $0$ since, recalling that $f(\delta) \sim \sqrt{2\delta}$ as $\delta\to 0$,
\[\exp\Big({-tk_n^2\sum_{i=1}^n \eta_i}\Big)\le 1, \quad \exp\Big({-f\big(tk_n^2\big)\sum_{i=1}^n \xi_i}\Big)\to \exp\Big(\frac{\pi^2\sqrt{t}}{\sqrt{2}}\Big) \text{ in }L^2 \text{ as }  n\to \infty,\]
and 
\[\exp\Big((f(tk_n^2)-sk_n)\sum_{i=1}^n \xi_i\Big)-\exp\Big(\frac{\pi^2}{2}(s-\sqrt{2t})\Big)\to 0 \text{ in } L^2 \text{ as } n\to \infty.\] 
\end{proof}

From this we obtain the tail asymptotics for $\eta$.
\begin{Cor}\label{cor : eta large}
 We have the following tail asymptotic:
 \begin{equation}\label{eq:asympeta}
\mathbb{P}(\eta > x) \sim 
\frac{{(2\pi)^{3/2}}}{\sqrt{x} \log^2x}
\quad 
\text{as } x\to\infty. 
\end{equation}
\end{Cor}
\begin{proof}
Define $\psi$ by $\mathbb{E}[\exp(-t\eta)] = \exp(-\psi(t))$. Then \cref{lem:sc_duration_length} implies that for $t>0$, 
\[n\psi\bigg(\frac{\log^4 n}{n^2} t\bigg)\to 
{\pi^2 \sqrt{t/2}}. 
\]
\noindent 
Thus for $\lambda_n=\log^4(n)/n^2$,
\begin{equation} \label{eq:psi lambda n}
    \psi(\lambda_n) \sim 
    {2\sqrt{2} \pi^2}\frac{\sqrt{\lambda_n}}{\log^2\lambda_n},
    \quad \text{as } n\to \infty.
\end{equation}
Moreover, for $\lambda\in [\lambda_{n+1},\lambda_n)$ we have $\psi(\lambda)\in [\psi(\lambda_{n+1}),\psi(\lambda_n)]$ because $\eta$ is positive, so that the asymptotic \eqref{eq:psi lambda n} holds with $\lambda$ in place of $\lambda_n$ as $\lambda\searrow 0$.
\noindent By the usual Tauberian Theorems (see for instance \cite[Theorem A]{bingham1974asymptotic}), we obtain the result.
\end{proof}

By the representation \eqref{eq:etaxisum}, \cref{lem:sc_duration_length} implies that we  we have the joint convergence 
\[
\Big(\frac{\hleft_n+\cleft_n}{n/\log^2 n}, \frac{\sigma(n)}{n^2/\log^4 n}\Big) \overset{\text{d}}{\longrightarrow} 
\Big(-{\frac{\pi^2}{2}},\sigma\Big),
\]
where 
$\sigma$ is the hitting time of $-{\pi^2/2}$ for a standard Brownian motion. On the other hand, by \cref{thm: sheffield's convergence} we have that 
\[ 
\Big(\frac{\cH_{-Nt}}{\sqrt{N}},\frac{\cC_{-Nt}}{\sqrt{N}}\Big)
\overset{\textnormal{d}}{\longrightarrow} (\tfrac12 B_t,\tfrac12B_t)_{t\ge 0}
\quad 
\text{as } N\to \infty.
\]  
So now, with $\lambda_n=\log^4(n)/n^2$ as before, consider the joint law of 
\begin{multline*}
\Big(\Big(\frac{\cH_{-(n^2/\log^4n)t}}{n/\log^2n},\frac{\cC_{-(n^2/\log^4n)t}}{n/\log^2n}\Big)_{t\ge 0}, \frac{\hleft_n+\cleft_n}{n/\log^2 n}, \frac{\sigma(n)}{n^2/\log^4 n}\Big)\\ =\Big(\Big(\frac{\cH_{-t/\lambda_n}}{1/\sqrt{\lambda_n}},\frac{\cC_{-t/\lambda_n}}{1/\sqrt{\lambda_n}}\Big)_{t\ge 0}\Big), \frac{\cH_{-\sigma(n)}+\cC_{-\sigma(n)}}{n/\log^2 n}, \frac{\sigma(n)}{n^2/\log^4 n}\Big). 
\end{multline*}
We know that the joint law is tight in $n$, since all the components converge in distribution. Moreover, suppose that we have convergence along some subsequence. By Skorohod's {representation} theorem we may assume this convergence is almost sure, which means that the law of the limit must be equal to that of $((\tfrac12 B_t,\tfrac12 B_t)_t, B_{\sigma},\sigma)$ where the marginal law of $\sigma$ is that of the hitting time of $-{\pi^2/2}$ by a Brownian motion, and must also satisfy $B_\sigma=-{\pi^2/2}$ almost surely. We claim that this implies $((\tfrac12 B_t,\tfrac12 B_t)_t, B_{\sigma},\sigma)$ is equal in distribution to $((\tfrac12 B_t,\tfrac12 B_t), B_{\sigma},\sigma(B))$ where $\sigma(B)$ is the hitting time of $-{\pi^2/2}$ \emph{for the Brownian motion $B$}. Indeed, since $B_\sigma=-{\pi^2/2}$ we have $\sigma\ge \sigma(B)$ almost surely, while on the other hand, they have the same distribution. In conclusion, we have the following.

\begin{Prop} \label{lem:concl_sigma}
We have the scaling limit
\[\Big(\Big(\frac{\cH_{-(n^2/\log^4n)t}}{n/\log^2n},\frac{\cC_{-(n^2/\log^4n)t}}{n/\log^2n}\Big)_{t\ge 0}, \frac{\hleft_n+\cleft_n}{n/\log^2 n}, \frac{\sigma(n)}{n^2/\log^4 n}\Big)
\overset{\textnormal{d}}{\longrightarrow}((\tfrac12 B_t,\tfrac12 B_t)_{t\ge 0}, B_{\sigma(B)},\sigma(B))\]
as $n\to \infty$, where $\sigma(B)$ is the hitting time of $-{\pi^2/2}$ by $B$ (and therefore, clearly $B_{\sigma(B)}=-{\pi^2/2}$). 
\end{Prop}

We conclude by providing the scaling limit of the \emph{number} of reduced steps before time $n$, namely
\begin{equation}
\label{eq:def_N_red}
\Nleft_n := 
\sup \Big\{
k\geq 1, \; \sum_{i=1}^{k} \eta_i < n
\Big\}.
\end{equation}
The random variable $\Nleft_n$ is the generalised inverse of $\sigma(n)$.
This will be useful when establishing the scaling limit of \cref{thm: main intro}, which concerns the whole process $(\cS,\cD)$ in which the reduced walks are embedded. 

\begin{Prop}[Tightness of rescaled number of reduced steps]
\label{prop:tightness_N_red}
    Let $u_n:=\sqrt{n} \log^2(n)$. 
    We have the following convergence in distribution:
    \[
    \frac{\Nleft_n}{u_n}
    \overset{\textnormal{d}}{\longrightarrow}
    \sigma^{-1/2},
    \]
    where $\sigma$ is the Brownian hitting time of $-{\pi^2/2}$, as in \cref{lem:sc_duration_length}.
\end{Prop}
\begin{proof}
    Let $x \in \R_+$. By definition of $\Nleft_n$, 
    \[
    \Pb\Big(\frac{\Nleft_n}{u_n} \geq x\Big)
    =
    \Pb\Big( \sum_{i=1}^{xu_n} \eta_i < n \Big)
    =
    \Pb\Big( \frac{1}{x^2 n}\sum_{i=1}^{xu_n} \eta_i < \frac{1}{x^2} \Big).
    \]
    By \cref{lem:sc_duration_length}, we know that 
    \[
    \frac{1}{x^2 n}\sum_{i=1}^{xu_n} \eta_i
    \overset{\text{d}}{\longrightarrow} \sigma
    \quad \text{as } n\to\infty.
    \]
    Therefore, we have 
    \[
    \Pb\Big(\frac{\Nleft_n}{u_n} \geq x\Big)
    \rightarrow 
    \Pb(\sigma^{-1} \geq x^2)
    \quad \text{as } n\to\infty.
    \]
    We conclude that $\frac{\Nleft_n}{u_n}$ converges towards $\sigma^{-1/2}$ in distribution.    
\end{proof}


\section{Exploration into the future}
\label{sec:explo_future}

The purpose of this section is to present the analogue of \cref{sec:backward}, now going forward in time. The main issue is that the exploration into the future does not have such a nice random walk structure as the reduced walk (or exploration into the past): indeed, there are $\rF$ symbols whose match lies in the negative half-line $(-\infty,0]$. We will see that one can describe the future exploration as a concatenation of i.i.d.\ words. The common law, denoted $P_\rF$, represents the law of a word in between two consecutive $\rF$ symbols matched to the left of $0$. We will describe this law explicitly in \cref{sec:forward_step}, which will enable us to estimate the number of hamburgers and cheeseburgers that remain in $\overline{P}_\rF$, and then to derive a scaling limit akin to \cref{thm:hn_scaling}. This will also allow us to rule out the contribution into the discrepancy of $\rF$ symbols that are matched in $(-\infty,0]$.


\subsection{Description of the step distribution}
\label{sec:forward_step}

Recall the definitions of $\tau_{\rF}, P_{\rF}, (\tau_{\rF}^{i}, P_{\rF}^{i})_{i\ge 0}$ at the end of \cref{par:future walk}. 
In the subcritical case $q<4$, the {expected reduced} length $|\overline{P_\rF}|$ of $P_{\rF}$ is infinite {(this is essentially a consequence of \cite[Lemma 3.7]{SheffieldScott2016QGAI}, see the second paragraph of \cite[Section 3.6]{SheffieldScott2016QGAI})}. The number of $\rF$ symbols left in the reduced word {$\overline{P}_{\rF}$} are then negligible compared to the Brownian {$\sqrt{n}$} scaling. As we expect a different scaling in the critical case $q=4$, understanding the {reduced} length of $P_{\rF}$ is central.

We will soon derive the scaling limit of burger counts along the exploration forwards in time which is defined by concatenating the $P_{\rF}^{i}$s  (the conclusion will be reached at the end of \cref{sec:scal_lim_forw}). 
The present section is only concerned with the law of the word $P_\rF$ itself.

{
We will describe the law of $P_{\rF}$ as part of a \emph{biased} $\rF$-excursion. If $e$ is a word ending with $\rF$, we can write it in its maximal excursion decomposition $e=y(k)\cdots y(1)y(0)$ with $y(0)=\rF$; that is, the unique concatenation of this form with all the $y(k)$ equal to $\rc,\rC,\rh,\rH$ or a maximal $\rF$-excursion (in the sense of \cref{par:red_walk}).  We call $r(e)=k$ the \textbf{excursion length} of the word $e$. In words, $r(e)$ is the length of $e$ seen as a word in the alphabet $\mathscr{A}_{\rh} \cup \mathscr{A}_{\rc}$.
Note that, under $\mathbb{P}$ and on the event $\{X(0)=\rF\}$, the $\rF$-excursion word $E=X(\varphi(0))\cdots X(0)$ has excursion length $r(E)=\tleft$ by definition of the reduced walk in \cref{par:red_walk}. Hence by \cref{prop: loop_bdry_length} and \cref{prop: loop_expo}, we have that 
\begin{equation} \label{eq:asympt_r(E)}
\mathbb{P}(r(E) = \ell \mid X(0)=\rF) 
\sim 
{\frac{256}{\pi^4}} \frac{\log^2 \ell}{\ell^3} \quad \text{as } \ell \to \infty,
\end{equation}
so that in particular $\mathbb{E}[r(E)  \mid X(0)=\rF] <\infty$. We can actually calculate this expectation (which is also the expectation of $\tleft$).
\begin{lemma}[Expected excursion length]
\label{lem:exp_exc_length}
    We have 
    \[
    \mathbb{E}[r(E)  \mid X(0)=\rF]
    =
    4.
    \]
\end{lemma}
\begin{proof}
    Let $\mathcal{P}_{\rF}$ be the state space of $P_{\rF}$, i.e.\ the set of words $w=w(1)\cdots w(k)$ such that $w(k)=\rF$ and $\overline{w}$ only contains one $\rF$ symbol. Further, for $w\in \mathcal{P}_{\rF}$, let $\mathcal{E}_{\rF}(w)$ the set of words $w'$ such that $w'w$ is an $\rF$-excursion. For $w\in \mathcal{P}_{\rF}$ we can write  
    \[
    \Pb(P_{\rF}=w)
    =
    \sum_{e\in \mathcal{E}_{\rF}(w)} \Pb(X(\varphi(\tau_{\rF}) \cdots X(0) \cdots X(\tau_{\rF}) = e).
    \]
    By translation invariance, this is also 
    \[
    \Pb(P_{\rF}=w)
    =
    \sum_{e\in \mathcal{E}_{\rF}(w)} \Pb(E = e, X(0)=\rF)
    =
    \sum_{e\in \mathcal{E}_{\rF}(w)} \mathsf{w}(e).
    \]
    Summing over $w\in \mathcal{P}_{\rF}$, we get 
    \[1 = \sum_{w\in \mathcal{P}_{\rF}} \sum_{e\in \mathcal{E}_{\rF}(w)} \mathsf{w}(e).\]
    
    But we can express the latter sum in a different way, counting the number of occurrences of a given $\rF$-excursion $e$. Let $e$ an $\rF$-excursion and decompose it over its maximal $\rF$-excursions as $e=y(r(e))\cdots y(1)y(0)$ with $y(0)=\rF$. We claim that the number of $w\in \mathcal{P}_{\rF}$ such that $e=w'w$ for some $w'$ is exactly $r(e)$. Indeed, on the one hand, it is clear that any shift of the form $w=y(k)\cdots y(0)$ for $0\le k\le r(e)-1$ is in $\mathcal{P}_{\rF}$, and one can append $w'=y(r(e))\cdots y(k+1)$ to the left of $w$ to get $e$. On the other hand, any other subword $w$ of $e$ would start inside an $\rF$-excursion, say $y(i)$, $i\ge 1$, so that $\overline{w}$ will contain the final $\rF$ symbol of $y(i)$ in addition to $y(0)=\rF$, and hence $w\notin \mathcal{P}_{\rF}$. Thus our claim follows, and we can write the above sum as
    \[
    1 = \sum_{w\in \mathcal{P}_{\rF}} \sum_{e\in \mathcal{E}_{\rF}(w)} \mathsf{w}(e) 
    = \sum_{e\in \mathcal{E}_{\rF}} r(e)\mathsf{w}(e),
    \]
    where $\mathcal{E}_{\rF}$ is the set of all $\rF$-excursions. 
    We conclude that
    \[
    \mathbb{E}[r(E)  \mid X(0)=\rF]
    =
    \frac{1}{\Pb(X(0)=\rF)}\sum_{e\in \mathcal{E}_{\rF}} r(e) \mathsf{w}(e)
    =
    4.
\vspace{-1em}    \] 
\end{proof}

Using \cref{lem:exp_exc_length}, we may define the biased law $\mathbb{Q}$ on $\rF$-excursion words by the following change of measure: for all measurable positive functions $f$,
\begin{equation} \label{eq:law_Q}
\mathbb{E}_{\mathbb{Q}}[f(E)]
=
\frac{1}{4}\mathbb{E}[r(E) f(E) \mid X(0)=\rF]. 
\end{equation}
}

\begin{lemma}[Description of $P_{\rF}$]
\label{lem: law of first F-symbol}
{Let $E$ be an $\rF$-excursion sampled under $\mathbb{Q}$, and write its maximal excursion decomposition as $E=Y(r(E))\cdots Y(1)Y(0)$ with $Y(0)=\rF$.}
{Conditionally on $E$,} sample $U$ uniformly from $ \{1, \cdots, {r(E)}\}$ and set $P = y({U}-1)\cdots y(1)x(0)$. Then {the law of $P$ under $\mathbb{Q}$} is the same as {that of $P_{\rF}$ under $\mathbb{P}$}. 
\end{lemma}

\begin{proof} 
As in the beginning of the proof of \cref{lem:exp_exc_length}, we start by writing for all $w\in \mathcal{P}_{\rF}$,
\[
    \Pb(P_{\rF}=w)
    =
    \sum_{e\in \mathcal{E}_{\rF}(w)} \mathsf{w}(e).
\]
On the other hand, by definition of $\mathbb{Q}$,
\[
\mathbb{Q}(P=w)
=
\sum_{e\in \mathcal{E}_{\rF}(w)} \mathbb{Q}(E=e, P=w) 
=
\sum_{e\in \mathcal{E}_{\rF}(w)} \frac{1}{r(e)}\mathbb{Q}(E=e)
=
\frac{1}{4} \sum_{e\in \mathcal{E}_{\rF}(w)} \Pb(E=e \mid X(0)=\rF).
\]
Since, for all $e\in \mathcal{E}_{\rF}(w)$, 
\[\Pb(E=e \mid X(0)=\rF) = \frac{\mathsf{w}(e)}{\Pb(X(0)=\rF)} = 4 \mathsf{w}(e),\]
we can equate the two marginals and conclude the proof.
\end{proof}

\subsection{Scaling limit of the forward burger counts}
\label{sec:scal_lim_forw}

Recall from \cref{par:future walk} the notation $P_{\rF} = X(1)\cdots X(\tau_{\rF})$ for the portion of the word between time $1$ and the first $\rF$ symbol that has a match in $(-\infty,0]$, and the quantities $\mathcal{C}^*(P_{\rF})$, $\mathcal{H}^*(P_{\rF})$ and $\mathcal{D}^*(P_{\rF})$. 
Our main objective in this section is to derive a scaling limit statement for the (forward) burger walks $\hright$ and $\cright$ defined in \eqref{eq:def_hcright}.

As in the case of the reduced walk (see \cref{prop:xi}), we first derive the asymptotics of the Laplace transform of $\mathcal{H}^*(P_{\rF})$. In fact, it turns out to have a closed-form expression. By symmetry, the exact same result holds for $\mathcal{C}^*(P_{\rF})$.

\begin{Prop}[Laplace transform of $\mathcal{H}^*(P_{\rF})$]
\label{prop:Laplace_H_PF}
The Laplace transform of $\mathcal{H}^*(P_{\rF})$ is given for all $\lambda \geq 0$ by
\[
\mathbb{E}[\mathrm{e}^{-\lambda\mathcal{H}^*(P_{\rF})}]
=
{\frac{2}{\pi^3}}\mathrm{e}^{\lambda}\int^1_0\int^1_0 \log\bigg(\frac{1+\sqrt{1-u^2}}{u}\bigg) \log\bigg(\frac{1+\sqrt{1-v^2}}{v}\bigg)\frac{u(1-\pi v/2)}{u+\varepsilon} \frac{\mathrm{d}u\mathrm{d}v}{u+v}.
\]
\noindent In particular, 
as $\lambda \to 0$, we have
\begin{equation}\label{eq: Laplace H PF}
\mathbb{E}[\mathrm{e}^{-\lambda\mathcal{H}^*(P_{\rF})}] = 1-c_1 \lambda\log^2\Big(\frac{1}{\lambda}\Big) {-} c_2 \lambda\log\Big(\frac{1}{\lambda}\Big)\cdot \log\log\Big(\frac{1}{\lambda}\Big) {+} c_3 \lambda\log\Big(\frac{1}{\lambda}\Big) + o\Big(\lambda\log\Big(\frac{1}{\lambda}\Big)\Big), 
\end{equation}
where $c_1$, $c_2$ and $c_3$ are positive and given by:
\[
{
c_1
=
\frac{1}{4\pi^2},
\quad 
c_2
=
\frac{2}{\pi^2} 
\quad \text{and} \quad 
c_3
=
\frac{\log(\pi/2)}{\pi^2}.}
\]
\end{Prop}

\begin{proof}
{We will express the Laplace transform of $\mathcal{H}^*(P_{\rF})$ in terms of that of the reduced walk.}
Indeed, \cref{lem: law of first F-symbol} ensures that we can construct $P_{\rF}$ as $P = Y(U-1)\cdots Y(1)X(0)$ {under $\mathbb{Q}$}, where $E = Y(r(E))\cdots Y(1)X(0)$ is an $\rF$-excursion and $U$ is uniformly chosen in $\{1,\cdots, r(E)\}$. For all non-negative measurable function $f$ of the reduced word $\overline{P}_{\rF}$, we have
\begin{eqnarray*}
\mathbb{E}[f(\overline{P}_{\rF})] &=& \mathbb{Q}[f(\overline{Y(U-1)\cdots Y(1)X(0)})]\\
&=& \mathbb{Q}\bigg[\frac{1}{r(E)}\sum^{r(E)-1}_{j=0}f(\overline{Y({j})\cdots Y(1)X(0)})\bigg]\\
&=& {\frac14} \mathbb{E}\bigg[\sum^{r(E)-1}_{j=0}f(\overline{Y({j})\cdots Y(1)X(0)}) \; \Big| \; {X(0)=\rF}\bigg].  
\end{eqnarray*}

\noindent Consider the reduced random walk $(\cleft_n, \hleft_n)$ induced by $Y$, {as defined in \cref{par:red_walk}. Under the conditional law $\Pb( \, \cdot \, | X(0)=\rF)$}, we have $r(E) = \tleft$, and for $j \geq 0$,  
\[
\mathcal{H}^*(\overline{Y(j)\cdots Y(1)X(0)}) = \hleft_{j}
\quad 
\text{on } \{j<\tleft\}.\footnote{Recall from our definition of $\mathcal{H}^*(P_{\rF})$ that $\mathcal{H}^*(\overline{Y(j)\cdots Y(1)X(0)})$ does not count the $\rF$ symbol at time $0$.} 
\]

\noindent Thus by setting $f(\overline{P}_{\rF}) = \mathrm{e}^{-\lambda\mathcal{H}^*(P_{\rF})}$, 
\[\mathbb{E}[\mathrm{e}^{-\lambda\mathcal{H}^*(P_{\rF})}] = 
{\frac14} \mathbb{E}\bigg[\sum^{\tleft-1}_{n=0} \mathrm{e}^{-\lambda\hleft_n}\bigg] 
= {\frac14} \mathbb{E}\bigg[\sum^{\infty}_{n=0}\mathrm{e}^{-\lambda\hleft_n}\mathds{1}_{\{\tleft>n\}}\bigg].\]

\noindent Applying the coupling of $(\hleft, \cleft)$ with $(\th,\tc)$ in \eqref{eq:(h,c)_geo_construc}, we can express the sum in terms of the one-dimensional (non-lazy) walks $(\th,\tc)$:
\begin{eqnarray}
\mathbb{E}\bigg[ \sum^{\infty}_{n=0} \mathrm{e}^{-\lambda\hleft_n}\mathds{1}_{\{\tleft>n\}}\bigg] 
&=& \mathbb{E}\bigg[\sum^{\infty}_{k=0} \mathrm{e}^{-\lambda \th_k}\mathds{1}_{\{\tauh>k \}} \sum^{\tN_{k+1}}_{j=\tN_{k}+1}\mathds{1}_{\{\tauc>j\}}\bigg]\nonumber\\
&=&
\sum^{\infty}_{k=0} \mathbb{E}\Big[\mathrm{e}^{-\lambda \th_k}\mathds{1}_{\{\tauh>k \}}\Big]\mathbb{E}\bigg[\sum^{\tN_{k+1}}_{j = \tN_k+1}\mathds{1}_{\{\tauc>j\}}\bigg]. \label{eq: expression Laplace H PF in rw}
\end{eqnarray}

{Let us take a look at the hamburger term in the sum, for fixed $k$.} For brevity, we henceforth write $F(\lambda) := F_{\xi}(\lambda)$, $\lambda \geq 0$, {for the Laplace transform of $\xi$, see \eqref{eq:hat_F_xi}}. We also set $\varepsilon = \varepsilon(\lambda) := \frac{2}{\pi}(F(\lambda) - 1)$, {which is non-negative}. From \cref{prop:xi}, we have $\varepsilon\sim \frac{2}{\pi a_1} \lambda/\log^2(\lambda) = {\pi \lambda/\log^2(\lambda)}$. 
For any $\lambda \geq 0$, we introduce the exponential martingale
\[
M_k(\lambda)=\mathrm{e}^{-\lambda \th_k} F(\lambda)^{-k}, 
\quad k\geq {0}.
\]

\noindent For {$k \geq 0$}, by the optional stopping theorem, 
\[
1 
=
\mathbb{E}[\mathrm{e}^{- \lambda \th_{\tauh\wedge k}} F(\lambda)^{-\tauh\wedge k}] 
= F(\lambda)^{-k}\mathbb{E}[\mathrm{e}^{-\lambda \th_k}\mathds{1}_{\{\tauh>k\}}] + \mathrm{e}^{\lambda}\mathbb{E}[F(\lambda)^{-\tauh}\mathds{1}_{\{\tauh\leq k\}}]. 
\]
\noindent Since $F(\lambda)\geq 1$ and $\mathrm{e}^{-\lambda \th_k}\leq 1$ on $\tauh>k$,  the first term goes to 0 as $k\to \infty$. We deduce by monotone convergence that
\[
1 = 
\mathrm{e}^{\lambda}\mathbb{E}[F(\lambda)^{-\tauh}].
\]

\noindent Hence
\[
\mathbb{E}\Big[\mathrm{e}^{-\lambda \th_k}\mathds{1}_{\{\tauh>k\}}\Big] = F(q)^k\Big(1-\mathrm{e}^{\lambda}\mathbb{E}\Big[F(\lambda)^{-\tauh}\mathds{1}_{\{\tauh\leq k\}}\Big]\Big) = \mathrm{e}^{\lambda}\mathbb{E}\Big[F(\lambda)^{-(\tauh-k)}\mathds{1}_{\{\tauh>k\}}\Big]. 
\]

\noindent By \cref{prop: prob_tau} together with the exact expression of $F_{\ell}$ from \eqref{eq:expr_F_ell_prop}, we continue the calculation by\footnote{{The use of Fubini here and throughout the proof is easily justified (and will never be expressly mentioned) since all the double sums converge absolutely.}} 
\begin{align}
\mathbb{E}\Big[\mathrm{e}^{-\lambda \th_k}\mathds{1}_{\{\tauh>k\}}\Big] &= \mathrm{e}^{\lambda} \sum^{\infty}_{\ell = k} F(\lambda)^{-(\ell + 1 - k)}\mathbb{P}(\tauh = \ell+1)\nonumber\\
&= \mathrm{e}^{\lambda} F(\lambda)^{k-1}\int^1_0 u\log\bigg(\frac{1+\sqrt{1-u^2}}{u}\bigg) \sum^{\infty}_{\ell = k}\Big(\frac{1-\pi u/2}{F(\lambda)}\Big)^{\ell}\mathrm{d}u\nonumber\\
&= \frac{2}{\pi} \mathrm{e}^{\lambda} \int^1_0 \log\bigg(\frac{1+\sqrt{1-u^2}}{u}\bigg) \frac{u}{u+\varepsilon}(1-\pi u/2)^k\mathrm{d}u. \label{eq: Laplace H PF hk} 
\end{align}

Now we turn to the calculation of the term with respect to $c$ in \eqref{eq: expression Laplace H PF in rw}. Using \cref{prop: prob_tau} and  \eqref{eq:expr_F_ell_prop} again, 
\begin{align*}
\mathbb{E}\bigg[\sum^{\tN_{k+1}}_{j = \tN_k+1}\mathds{1}_{\{\tauc>j\}}\bigg] &= \mathbb{E}\bigg[\sum^{\tN_{k+1}}_{j=\tN_k+1}\sum^{\infty}_{\ell = j}\mathbb{P}(\tauc = \ell + 1)\bigg]\\
&=
\mathbb{E}\bigg[\sum^{\tN_{k+1}}_{j=\tN_k+1} \int^1_0 v\log\bigg(\frac{1+\sqrt{1-v^2}}{v}\bigg) \sum^{\infty}_{\ell = j}(1-\pi v/2)^{\ell}\mathrm{d}v\bigg]\\
&= \frac{2}{\pi} \int^1_0 \log\bigg(\frac{1+\sqrt{1-v^2}}{v}\bigg) \mathbb{E}\bigg[\sum^{\tN_{k+1}}_{j = \tN_{k}+1}(1-\pi v/2)^{j}\bigg]\mathrm{d}v.
\end{align*}
\noindent A direct calculation shows that for all $|a| < 2$, $\mathbb{E}[a^{\tN_k+1} + \cdots + a^{\tN_{k+1}}] = a (2-a)^{-(k+1)}$. The last display thus simplifies to 
\begin{equation}\label{eq: Laplace H PF tau c}
\mathbb{E}\bigg[\sum^{\tN_{k+1}}_{j = \tN_k+1}\mathds{1}_{\{\tauc>j\}}\bigg] 
= \frac{2}{\pi} \int^1_0 \log\bigg(\frac{1+\sqrt{1-v^2}}{v}\bigg)(1-\pi v/2)(1 + \pi v/2)^{-(k+1)}\mathrm{d}v.
\end{equation}
We now plug \eqref{eq: Laplace H PF hk} and \eqref{eq: Laplace H PF tau c} back into \eqref{eq: expression Laplace H PF in rw} and use the identity $\sum^{\infty}_{k=0}(1-\pi u/2)^{k}(1+\pi v/2)^{-(k+1)} = \frac{2}{\pi}\frac{1}{u+v}$ for all $u,v\in (0,1)$. It gives that
\[
\mathbb{E}[\mathrm{e}^{-\lambda\mathcal{H}^*(P_{\rF})}] . 
= \frac{2}{\pi^3} \mathrm{e}^{\lambda}\int^1_0\int^1_0 \log\bigg(\frac{1+\sqrt{1-u^2}}{u}\bigg) \log\bigg(\frac{1+\sqrt{1-v^2}}{v}\bigg)\frac{u(1-\pi v/2)}{u+\varepsilon} \frac{\mathrm{d}u\mathrm{d}v}{u+v}. 
\]
This already proves the first claim of \cref{prop:Laplace_H_PF}. The second is purely analysis and is proven in the appendix, see \cref{appendix:Laplace_H_PF}.
\end{proof}

{
We now use the Laplace transform expansion in \cref{prop:Laplace_H_PF} to derive the scaling limit of the number of burger orders along our exploration into the future. 
Recall the two processes $\hright$ and $\cright$ defined in \eqref{eq:def_hcright}.
The result is the analogue of the scaling limit of the reduced walks (\cref{thm:lazy_scaling}), but in the exploration into the future. 
Recall from the discussion around~\eqref{eq:tau_F^n} that this exploration consists of an i.i.d.\ concatenation $(P_{\rF}^k)_{k\geq 1}$ of words with law $P_{\rF}$. 
}

\begin{Thm}[Scaling limits in the exploration into the future]
\label{thm:CPFtight}
Let $(P_{\rF}^k)_{k\geq 1}$ be i.i.d. words with law $P_{\rF}$. 
We introduce 
\[
\Hright_n
:=
\frac{1}{n\log n}\hright_n - {\frac{1}{4\pi^2}}\log n - {\frac{5}{2\pi^2}} \log\log n {+} {\frac{\log(\pi/2)}{\pi^2}}, 
\]
and symmetrically,
\[
\Cright_n
:=
\frac{1}{n\log n}\cright_n - {\frac{1}{4\pi^2}}\log n - {\frac{5}{2\pi^2}} \log\log n  {+} {\frac{\log(\pi/2)}{\pi^2}}.
\]
Then we have the convergence
\begin{equation}\label{eq: scaling to stable HPF}
\Hright_n\stackrel{\textnormal{d}}{\longrightarrow} \hat\zeta,
\end{equation}
{where $\hat\zeta$ is a $1$-stable random variable, whose Laplace transform is simply $\Eb[\mathrm{e}^{-\lambda \hat\zeta}] = \exp(-\lambda\log \lambda)$, $\lambda>0$.
By symmetry, the same convergence holds for $\Cright_n$.
As a consequence, the process
\begin{equation}\label{eq: scaling to stable DPF}
\left(\frac{1}{n\log n}\sum^n_{k=1} \mathcal{D}^*(P_{\rF}^k), \; n\geq 1 \right) 
=
\left(\frac{\hright_n-\cright_n}{n\log n}, \; n\geq 1 \right) 
\end{equation}
is tight.
}
\end{Thm}

\begin{remark}
    Note that the statement of \cref{thm:CPFtight} is a bit more convoluted than that of \cref{thm:lazy_scaling}. The issue is that the joint convergence of $(\Hright_n,\Cright_n)$ does not come for free from marginal convergence since, unlike for reduced walks, the pair is not independent. We believe that the proof of \cref{prop:Laplace_H_PF} could be adapted to deal with the joint Laplace transform of $\mathcal{H}^*(P_{\rF})$ and $\mathcal{C}^*(P_{\rF})$, which would then upgrade \eqref{eq: scaling to stable HPF} to a joint convergence statement and \eqref{eq: scaling to stable DPF} to proper convergence in distribution. However, tightness will be enough for our purposes.
\end{remark}

\begin{proof}
Subtracting $\Cright_n$ from $\Hright_n$ the tightness of \eqref{eq: scaling to stable DPF} is a consequence of \eqref{eq: scaling to stable HPF} and the analogous statement for $\Cright_n$ (actually, tightness of $\Hright_n$ would be enough).
Therefore, it remains to prove the marginal convergence in \eqref{eq: scaling to stable HPF}. It follows by expressing the Laplace transform of $\Hright_n$ in terms of that of $\mathcal{H}^*(P_{\rF})$ (by independence), and using the expansion in \cref{prop:Laplace_H_PF}. The convergence therefore boils down to calculations that we feel free to skip since they are very similar to the proof of Theorems~\ref{thm:hn_scaling}--\ref{thm:lazy_scaling}.
\end{proof}

Recall that the notation $\mathcal{H}^*(P_\rF)$ and $\mathcal{C}^*(P_\rF)$ refers to the ham and cheeseburger counts \emph{without} taking into account the final $\rF$ symbol in $P_\rF$. We can now add in the missing $\rF$ symbols in \cref{thm:CPFtight}, i.e.\ derive the scaling limit of the two burger counts up to time $\tau_{\rF}^n$. 

\begin{Cor}[Adding back the $\rF$ symbols]
\label{cor:sc_lim_H_tauF}
    Recall the times $\tau_{\rF}^n$, $n\geq 0$, defined in \eqref{eq:tau_F^n}. Then
    \[
    - \frac{1}{n\log n} \cH_{\tau_{\rF}^n} - {\frac{1}{4\pi^2}}\log n - {\frac{5}{2\pi^2}} \log\log n {+} {\frac{\log(\pi/2)}{\pi^2}} \stackrel{\textnormal{d}}{\longrightarrow} \hat\zeta,
    \]
    where $\hat\zeta$ is as in \cref{thm:CPFtight}. The same is true with $\cH_{\tau_{\rF}^n}$ replaced by $\cC_{\tau_{\rF}^n}$.
\end{Cor}

\begin{proof}
    Recall from the discussion around \eqref{eq:tau_F^n} that our \emph{exploration into the future}  gives a coupling between the hamburger-cheeseburger accumulation model and a sequence of i.i.d.\ words $P_{\rF}^i$, $i\geq 1$.
    By construction, under this coupling, the sum\footnote{We stress that it is a sum (and not a difference) because we have defined $\mathcal{H}^*(P_{\rF})$ as the number of hamburger orders in $P_{\rF}$, while orders are counted negatively in $\cH_{\tau_{\rF}^n}$ by definition.} $N^{\rF \leftrightarrow \rH}_n := \cH_{\tau_{\rF}^n}+ \hright_n$ is exactly the number of $\rF$ symbols among the symbols $X(\tau_{\rF}^i)$, $1\leq i\leq n$, that are matched to a hamburger. Therefore, writing 
    \[
    - \cH_{\tau_{\rF}^n}
    =
    \hright_n - N^{\rF \leftrightarrow \rH}_n 
    \]
    and using the crude bound $N^{\rF \leftrightarrow \rH}_n \leq n$, \cref{thm:CPFtight} provides our claim. The statement for cheeseburgers holds by the symmetry between burgers.
\end{proof}

\subsection{On the contribution of unmatched $\rF$ symbols}
\label{sec:unmatched_F}
The aim of this section is to provide a few estimates relative to unmatched $\rF$ symbols in $(1,+\infty)$. We say that $X(i) =\rF$, $i\geq 1$, is \textbf{unmatched} if the match $\varphi(i)$ of $i$ lies in $(-\infty,0]$.
In order to establish a scaling limit for the discrepancy, we will need to answer the following questions:
\begin{itemize}
    \item[(a)] How many unmatched $\rF$ symbols are there in $[1,n]$ as $n\to\infty$?
    \item[(b)] What is the contribution of the discrepancy of the unmatched symbols to the overall discrepancy at time $n$?
\end{itemize}
We answer both questions in this section. In fact, we will answer (a) and (b) at the same time by showing that the quantity of unmatched $\rF$ symbols by time $n$ is of smaller order than $v_n := \frac{\sqrt{n}}{\log(n)}$, which will be the order of magnitude of the discrepancy, see \cref{sec:tightness_Dn}, \eqref{eq:vn}.

\paragraph{Number of unmatched symbols $\Nright_n$.}
Recall from \eqref{eq:tau_F^n} the definition of the times $\tau_{\rF}^k$, $k\geq 1$, associated with unmatched $\rF$ symbols.
We introduce for $n\geq 1$, the number of unmatched symbols at time $n$, 
\begin{equation} \label{eq:def_NF}
\Nright_n
:=
\sup \big\{
k\geq 0, \; \tau_{\rF}^k< n
\big\},
\end{equation}
i.e.\ $\Nright_n$ is the number of $\rF$ symbols in $\overline{X(1)\cdots X(n)}$.
The following tightness result states that $\Nright_n$ is of order at most 
\begin{equation} \label{eq:def_wn}
w_n
:=
\frac{\sqrt{n}}{\log^2(n)},
\end{equation}
thereby answering Question (a) above. {It can also be seen as an analogue of \cref{prop:tightness_N_red}, for the exploration into the future.\footnote{Note, however, that the proofs are drastically different, because we have no information on the \emph{durations} along the exploration into the future. The key input in this context will be some extra monotonicity property (that the exploration into the past does not possess).}}
\begin{Prop}[{Tightness of rescaled number of unmatched $\rF$ symbols}]
\label{prop:tightness_N_F}
    The sequence 
    \[\Big(\frac{\Nright_n}{w_n}, \, n\geq 1\Big)\] 
    is tight.
\end{Prop}

\begin{remark} \label{rk:Deltak_negligible}
Recall that $\cD(X(\tau_\rF^i))$ is the contribution to the discrepancy from the symbol at time $\tau_{\rF}^i$ (also equal to $\cD_{\tau_\rF^i}-\cD_{\tau_\rF^i-1}$). We set
\begin{equation}
\label{eq:def_Deltak}
\Delta_{\rF}(k):= \sum_{i=1}^k \mathcal{D}(X(\tau_\rF^i)), 
\quad k\geq 1.
\end{equation}
Using the crude bound $|\Delta_{\rF}(\Nright_n)| \leq \Nright_n$, a consequence of \cref{prop:tightness_N_F} is that the total discrepancy $\Delta_{\rF}(\Nright_n)$ carried by unmatched $\rF$ symbols by time $n$ is negligible at scale $v_n$. This answers Question (b) above.
\end{remark}

\begin{proof}[Proof of \cref{prop:tightness_N_F}]
    The idea is to say that when $\Nright_n/w_n$ is large, the burger count $\cS(1,\tau_\rF^{\Nright_n})$ at time $\tau_{\rF}^{\Nright_n}$ must be atypically large, which is unlikely because $\mathcal{S}$ is a simple random walk.  To make this rigorous, we will use what we know about the burger count at time $\tau_{\rF}^k$. (In the proof we use the notation $\cS(1,m)$ rather than $\cS_m$ for the burger count at time $m\ge 1$ to avoid compound indices.)

    The word $X(1,\tau_{\rF}^k)$ contains, asymptotically, $k\log^2 k$ orders (including the $\rF$ symbols) by \cref{cor:sc_lim_H_tauF}. This entails for example that the sequence
    \begin{equation} \label{eq:tightness_count_PF}
    \bigg(\frac{k \log^2(k)}{|\mathcal{S}(1,\tau_{\rF}^{k})|}, \; k\geq 1\bigg)
    \end{equation}
    is tight.

    We use this tightness as follows. Let {$B\in\mathbb{N}^*$ and $\delta>0$}. {Since there are only orders in $\overline{P}_{\rF}$, the count $k \mapsto |\mathcal{S}(1,\tau_{\rF}^k)|$ is increasing. Therefore, the following bound holds for all $\varepsilon>0$:
    \begin{align*}
    &\Pb(\Nright_n \geq Bw_n) \\
    &\leq 
    \Pb(|{\mathcal{S}(1,\tau_{\rF}^{\Nright_n})}|
    \geq |{\mathcal{S}(1,\tau_{\rF}^{Bw_n})}|
    ) \\
    &\leq
    \Pb(|\mathcal{S}(1,\tau_{\rF}^{Bw_n})|
    \leq \varepsilon Bw_n \log^2(B w_n))
    +
    \Pb(|\mathcal{S}(1,\tau_{\rF}^{\Nright_n})|
    > \varepsilon Bw_n \log^2(B w_n)). \\
    &\leq
    \sup_{k\geq 1} \; \Pb(|\mathcal{S}(1,\tau_{\rF}^{\tau_\rF^k})|
    \leq \varepsilon k \log^2(k))
    +
    \Pb(|\mathcal{S}(1,{\tau_{\rF}^{\Nright_n}})|
   > \varepsilon Bw_n \log^2(B w_n)).
    \end{align*}
    By tightness of \eqref{eq:tightness_count_PF}, we may fix $\varepsilon$ (small enough) so that the first term is less than $\delta$. Now that $\varepsilon$ is fixed, we deal with the second term. Since $\tau_{\rF}^{\Nright_n} \leq n$ by definition of $\Nright_n$,}
    we can bound it as 
    \[
    \Pb(|{\mathcal{S}(1,\tau_{\rF}^{\Nright_n})}|
    > \varepsilon Bw_n \log^2(B w_n))
    \leq 
    \Pb \Big(\sup_{1\le k\le n} |\mathcal{S}(1,k)| > \varepsilon Bw_n \log^2(B w_n)\Big).
    \]
    Note that $w_n \log^2(w_n) \sim \sqrt{n}$ as $n\to\infty$. 
    Thus, the right-hand side goes to $0$ as $B\to\infty$ uniformly over $n$ by standard simple random walk estimates (combining, say, L\'{e}vy's inequality \cite[Proposition A.1.2]{VaartA.W.vander2023Wcae} and Hoeffding's inequality). 
\end{proof}

\paragraph{An approximate Markov property for the discrepancy.}
\label{par:approx_markov}
Recall the notation $\Delta_{\rF}(\Nright_n)$ from \eqref{eq:def_Deltak}.
In order to estimate the discrepancy $\mathcal{D}_n$ at time $n$, it will be useful to single out the contribution $\Delta_{\rF}(\Nright_n)$ of unmatched $\rF$ symbols by time $n$, since $\mathcal{D}_n$'s dependence on the past comes through these unmatched symbols.  
The following elementary observation formalises this dependence.

\begin{Prop}[Approximate Markov property of the discrepancy]
\label{prop:approx_markov_F}
    Introduce the $\sigma$-field 
    \[
    \Gcal_0 
    :=
    \sigma(X(-k), k \in \mathbb{N}^*).
    \]
    Consider a coupling of $(X(k),k\in\mathbb{Z})$ with an independent sequence $(X'(-k),k\in\mathbb{N})$ of i.i.d.\ symbols with law \eqref{eq: symbol weights}. 
    Let $n\in\mathbb{N}^*$. 
    Under this coupling, we denote by $\mathcal{D}'(1,n)$ the discrepancy of $X(1,n)$, with  symbols $(X'(-k), k\in \mathbb{N})$ appended to the left of $(X(k),k\in\mathbb{N}^*)$.
    Then we have the equality
    \begin{equation} \label{eq:markov_approx}
    \mathcal{D}(1,n)
    =
    \mathcal{D}'(1,n) +  \Delta_{\rF}(\Nright_n)+ \Delta_{\rF}'(\Nright_n),
    \end{equation}
    where,
    for all $k\geq 1$, $\Delta_{\rF}(k)$ and $\Delta_{\rF}'(k)$ are the random variables defined in \eqref{eq:def_Deltak} with respective discrepancies $\mathcal{D}$ and $\mathcal{D}'$. 
    Moreover, $\mathcal{D}'(1,n)$ and $\Delta_{\rF}'(k)$ are independent of $\Gcal_0$.
\end{Prop}

\begin{remark}
    By translation invariance, an analogous statement holds at time $i$, defining
    \[
    \Gcal_i 
    :=
    \sigma(X(k), k <i),
    \]
    for all $i\in\mathbb{Z}$. The reason we see the decomposition \eqref{eq:markov_approx} as an \emph{approximate Markov property} is that it enables to express the future discrepancy trajectory (from time $i$ onwards) as a sum of an independent part and another random variable that was shown to be negligible at scale $v_n$ as $n\to\infty$ (\cref{rk:Deltak_negligible}).
\end{remark}

\begin{proof}
    The sequence $(X(k),k\in\mathbb{Z})$ is i.i.d.\ with law \eqref{eq: symbol weights}. Let $(X'(-k), k\geq 0)$ an \emph{independent} sequence of i.i.d.\ symbols with the same law \eqref{eq: symbol weights}, and denote by $\mathcal{D}'$ the discrepancy in the bi-infinite sequence $(\ldots X'(-1)X'(0)X(1)X(2)\ldots)$ relative to time $0$.
    
    We begin by noting that for any $k\in\mathbb{Z}$, whenever $X(k)\neq \rF$, the discrepancy (under $\mathcal{D}$ or $\mathcal{D}'$) of $X(k)$ is completely determined by its value. Moreover, if $X(k)=\rF$ finds its match in $[1,+\infty)$, then the discrepancies $\mathcal{D}(X(k))$ and $\mathcal{D}'(X(k))$ also coincide, since $X(k)$ has the same match in both bi-infinite sequences. This implies that the difference 
    \[
    \mathcal{D}(1,n)-\mathcal{D}'(1,n) 
    \]
    only comes from indices $k\leq n$ such that $X(k)=\rF$ and $\varphi(k)\leq 0$, i.e.\ from the times $\tau_\rF^i$ introduced in \eqref{eq:tau_F^n} such that $\tau_\rF^i \leq n$. Namely, we can write 
    \[
    \mathcal{D}(1,n)-\mathcal{D}'(1,n) 
    =
    \sum_{i=1}^{\Nright_n} \big(\mathcal{D}(X(\tau_{\rF}^i))-\mathcal{D}'(X(\tau_{\rF}^i))\big)
    =
    \Delta_{\rF}(\Nright_n)+ \Delta_{\rF}'(\Nright_n). 
    \]
    This shows \eqref{prop:approx_markov_F}. The independence comes from the fact that the two variables $\mathcal{D}'(X(1,n))$ and $\Delta_{\rF}'(k)$ are measurable with respect to $(\ldots X'(-1)X'(0)X(1)X(2)\ldots)$, which is independent of $\Gcal_0$. 
\end{proof}

%
%
\section{Scaling limit of the critical hamburger-cheeseburger walk}
\label{sec:final_cvg}

\subsection{Tightness of the discrepancy trajectory}
\label{sec:tightness_Dn}

Recall the scaling factor 
\begin{equation}\label{eq:vn}
    v_n = \frac{\sqrt{n}}{\log n}.
\end{equation}
The purpose of this section is to show that the discrepancy $\mathcal{D}_n$ at time $n$ is of order \emph{at most} $v_n$. Because we are after a functional scaling limit, we will actually need the estimate to be uniform over $n$.

The main result of this subsection is the following.
\begin{Thm}[Tightness of rescaled discrepancy]
\label{thm:UB_D_n}
    The sequence of rescaled discrepancies is tight, uniformly over the initial stack, i.e.\
    \[
    \sup_{n\in\mathbb{N}^*} \;  \mathbb{P}\big(|\mathcal{D}_n| > A v_n \mid \mathcal{G}_0\big)\to 0
    \quad \text{as } A\to \infty,
    \] 
    uniformly over realisations of the sequence $(X(-k), k\in\mathbb{N})$.
\end{Thm}

\noindent 
The remainder of this section is devoted to the proof of \cref{thm:UB_D_n}. 

\begin{proof}
Let us first describe the proof strategy.
The proof will hinge upon all the scaling limits that we have established so far, for the exploration into the past (\cref{thm:lazy_scaling}) as well as the exploration into the future (\cref{thm:CPFtight}).

The main issue is that we only know how to control the discrepancy at specific \emph{random} times (i.e.\ in the natural time parametrisation of our two explorations into the past and future).
These random times are so that both explorations display a nice random walk structure, which allowed us to deduce the scaling limits. We will now write (or bound) $\mathcal{D}_n$ at fixed time $n$ in terms of these nicer objects. This decomposition is in some sense based on excursion theory.

Recall from \eqref{eq:tau_F^n} that for $k\geq 1$, the random variable $\tau_{\rF}^k$ stands for the time when the $k$-th unmatched $\rF$ symbol appears. Moreover, in \cref{sec:unmatched_F} we introduced the number $\Nright_n$ of such unmatched $\rF$ symbols (i.e.~matched to the left of zero) that occur between time $1$ and time $n$. 
We now write, for $n\geq 1$, 
\begin{equation} 
\label{eq:D_n_decomposition}
\mathcal{D}_n 
=
\mathcal{D}\big(1,\tau_{\rF}^{\Nright_n}\big) + \mathcal{D}\big(\tau_{\rF}^{\Nright_n}+1, n\big).
\end{equation}
Observe that $\tau_{\rF}^{\Nright_n}$ is nothing but the time of the last unmatched $\rF$ symbol before time $n$.
For $|\mathcal{D}_n|$ to be large (at scale $v_n$), one of the two terms will have to be large (in absolute value). We can therefore treat them separately. We now make two simple observations. The reader can follow the construction on \cref{fig:discrep_decomp}.

\begin{figure}[t] 
  \bigskip
  \begin{center}
    \includegraphics[page=1,scale=1]{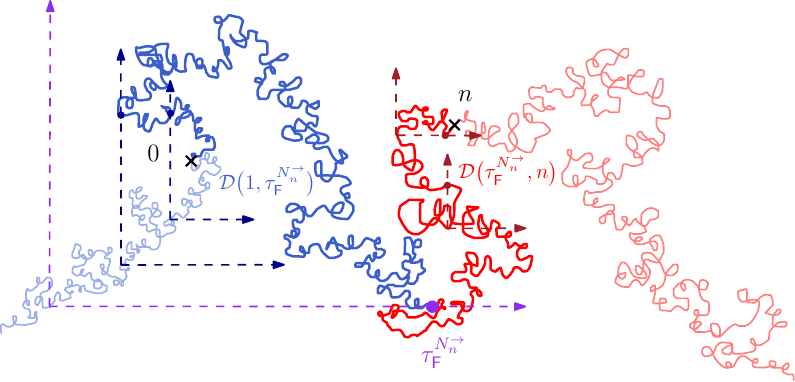}
  \end{center}
  \caption{Decomposition of the trajectory at time $n$. We split the discrepancy $\mathcal{D}_n$ at time $n$ into two parts, according to \eqref{eq:D_n_decomposition} and show that it can be understood from our explorations into the past and into the future. The \emph{blue} trajectory corresponds to the term $\mathcal{D}\big(1,\tau_{\rF}^{\Nright_n}\big)$. Its contribution comes from a concatenation of $\Nright_n$ words with law $P_\rF$, i.e.\ from our exploration into the future (we represented in dark blue some $\rF$ symbols that correspond to $\rF$-excursions straddling time $0$). The \emph{red} trajectory is the other term $\mathcal{D}\big(\tau_{\rF}^{\Nright_n}+1, n\big)$. Looking backwards from time $n$, it may be understood as the discrepancy along the reduced walk (i.e.\ our exploration into the past), stopped at the first time when we see an $\rF$-excursion straddling time $0$. Some $\rF$-excursions (corresponding to some steps in the reduced walk) are depicted in dark red. Note that the purple excursion is both an excursion straddling $0$ (exploration into the future) and an $\rF$-excursion backwards from time $n$ (exploration into the past).} 
  \label{fig:discrep_decomp}
\end{figure}

First, notice that 
\[\mathcal{D}\big(1,\tau_{\rF}^{\Nright_n}\big)=\cD(P_\rF^1\cdots P_\rF^{\Nright_n}),\]
in the notation of \cref{sec:forward_step}. 
If $\Nright_n$ were a given fixed value, we could control the discrepancy by \cref{cor:sc_lim_H_tauF}. Moreover, we know $\Nright_n$ is at most $w_n$ by \cref{prop:tightness_N_F}. In other words, estimating $\mathcal{D}\big(1,\tau_{\rF}^{\Nright_n}\big)$ roughly boils down to estimating the probability that the \emph{maximum} of a ``random walk''  exceeds a large value (at scale $v_n$).\footnote{To be more accurate, the discrepancy along our exploration into the future is not a true random walk (because of the unmatched $\rF$ symbols), but we can approximate it by a random walk.} 
This will be done in Step 1 below.

The second term can be treated in a similar way using the reduced walks, provided we are looking \emph{backwards} from time $n$. Namely, $n-\tau_{\rF}^{\Nright_n}$ is precisely the first index $i$ for the shifted inventory accumulation path $(X(n-i), 0\leq i \leq n)$ such that $X(n-i)$ is an $\rF$, and the corresponding $\rF$-excursion straddles $X(0)$. From the viewpoint of the shifted inventory accumulation path, this corresponds to the first $\rF$ excursion straddling time $n$.
To express this in a useful way, we need to recall the (lazy) reduced walks $\hleft$ and $\cleft$ defined in \cref{par:red_walk}. Let $\Nleft_n$ be the number of reduced steps before time $n$, as introduced in \eqref{eq:def_N_red} in \cref{sec: joint laws}, where the $\eta_i$, $i\geq 1$, are the durations sampled conditionally on the two reduced walks. 
Then by the above time reversal observations we have that
\begin{equation} \label{eq:D_tau_n_reduced}
- \mathcal{D}\big(\tau_{\rF}^{\Nright_n}+1, n\big)
\overset{\text{d}}{=}
\hleft_{\Nleft_n} - \cleft_{\Nleft_n}.
\end{equation}
In words, we see that the discrepancy $-\mathcal{D}\big(\tau_{\rF}^{\Nleft_n}+1, n\big)$ is the difference of the hamburger and cheeseburger reduced walks, taken at the random time $\Nleft_n$. Again, since we have established the scaling limit of the reduced walk (\cref{thm:lazy_scaling}) at fixed times, and given our estimate on $\Nleft_n$ (\cref{prop:tightness_N_red}), the problem boils down to estimating the probability that the maximum of another (true) random walk exceeds a large value. This will be achieved in Step 2.

We now make rigorous sense of these arguments, by showing the tightness of the two terms in \eqref{eq:D_n_decomposition}, rescaled by $v_n$, and conditioning on $\mathcal{G}_0$.
We fix $A>0$.

\paragraph{Step 1: Tightness for the first term of \eqref{eq:D_n_decomposition}.}
Notice that, for all $k$, we can write
\[
|\mathcal{D}({1},\tau_{\rF}^{k})|
\leq 
{|\hright_{k}-\cright_{k}| + k},
\]
{where the final term $k$ is an upper bound on the discrepancy $\Delta_\rF(k)$ coming from the $k$ unmatched $\rF$ symbols (see \eqref{eq:def_Deltak}) and $\hright_k,\cright_k$ are independent of $\mathcal{G}_0$.}
{Recall the scaling $w_n$ for $\Nright_n$ in \eqref{eq:def_wn}.}
We have, {for all $0<B\leq A$,}
\[
\Pb \big(|\mathcal{D}({1},\tau_{\rF}^{\Nright_n})| > 2Av_n \, \big| \, \mathcal{G}_0 \big)
\leq 
{
\Pb(|\hright_{\Nright_n}-\cright_{\Nright_n}| > A v_n, \Nright_n \leq B w_n)
+
\Pb(\Nright_n > B w_n).
}
\]
Note that the right hand side does not depend on the realisation of $(X(-k), k\in\mathbb{N})$.
The second term goes to $0$ uniformly over $n$ as $B\to \infty$ by \cref{prop:tightness_N_F}. So we can now focus on the first term, for fixed $B$ and all $A\geq B$. 
It can be bounded by 
\begin{align*}
{
\Pb(|\hright_{\Nright_n}-\cright_{\Nright_n}| > A v_n, \Nright_n \leq B w_n)
\leq 
\Pb\Big(\max_{1\leq k\leq Bw_n} |\hright_{k}-\cright_{k}|>Av_n \Big).}
\end{align*}
The process $(\hright_{k}-\cright_{k}, k\geq 1)$ is a symmetric random walk.
By L\'{e}vy's inequality \cite[Proposition A.1.2]{VaartA.W.vander2023Wcae}, we deduce that
\begin{multline*}
\Pb(|\hright_{\Nright_n}-\cright_{\Nright_n}| > A v_n, \Nright_n \leq B w_n) \\
\leq 
2\Pb(|\hright_{Bw_n}-\cright_{Bw_n}|>Av_n)
= 2\Pb\bigg({\frac{1}{Bw_n \log(Bw_n)}}|\hright_{Bw_n}-\cright_{Bw_n}|>{\frac{Av_n}{Bw_n\log(Bw_n)}}\bigg).
\end{multline*}
Noting that ${\frac{v_n}{Bw_n\log(Bw_n)}}\sim 1$ as $n\to\infty$, we obtain the desired tightness by \eqref{eq: scaling to stable DPF}, sending $A\to\infty$.

\paragraph{Step 2: Tightness for the second term of \eqref{eq:D_n_decomposition}.}

To estimate the second term of \eqref{eq:D_n_decomposition}, we now argue in a similar way, using our exploration into the past. 
Using our observation in \eqref{eq:D_tau_n_reduced}, we first write 
\[
\Pb(|\mathcal{D}(\tau_{\rF}^{\Nright_n}+1
, n)| > Av_n \mid \mathcal{G}_0)
= \Pb(|\mathcal{D}(\tau_{\rF}^{\Nright_n}+1, n)| > Av_n \mid \mathcal{G}_0)
= \Pb(|\hleft_{\Nleft_n}-\cleft_{\Nleft_n}| > Av_n),
\]
where the first equality holds since $\mathcal{D}(\tau_{\rF}^{\Nright_n}+1, n)$ is measurable with respect to $(X(k), k\in \mathbb{N}^*)$ and therefore independent of $\mathcal{G}_0$.
Recall from \cref{prop:tightness_N_red} the scaling $u_n$ for $\Nleft_n$.
Let $B>0$ and split the event as
\[
\Pb(|\hleft_{\Nleft_n}-\cleft_{\Nleft_n}| > Av_n) = \Pb(|{\hleft_{\Nleft_n}-\cleft_{\Nleft_n}}| > Av_n, \Nleft_n \leq Bu_n) + \Pb(\Nleft_n > Bu_n).
\]
The second term is again treated using \cref{prop:tightness_N_red} (taking $B$ large enough), so we now focus on the first one, for fixed $B$.
The process 
$({\hleft_{k}-\cleft_{k}}, k\geq 0)$
is a symmetric random walk, so that we can use again L\'{e}vy's inequality \cite[Proposition A.1.2]{VaartA.W.vander2023Wcae}. It gives that 
\begin{align*}
\Pb(&|{\hleft_{\Nleft_n}-\cleft_{\Nleft_n}}| > Av_n, \Nleft_n \leq Bu_n) \\
&\leq 
\Pb \Big( \max_{1\leq k\leq Bu_n} |{\hleft_{k}-\cleft_{k}}| > Av_n \Big) \\
& \leq
2\Pb \Big(|{\hleft_{Bu_n}-\cleft_{Bu_n}}| > Av_n\Big) 
=
2\Pb\left(\frac{\log^3(Bu_n)}{Bu_n}|{\hleft_{Bu_n}-\cleft_{Bu_n}}|>\frac{Av_n\log^3(Bu_n)}{Bu_n}\right).
\end{align*}
Going back to \cref{thm:hn_scaling}, we see that $\frac{\log^3(k)}{k}({\hleft_{k}-\cleft_{k}})$ converges in distribution to a stable random variable.
Thus, noting that $v_n\log^3(Bv_n)/u_n\sim 1$, we conclude the proof of tightness by again sending $A\to \infty$.
\end{proof}

\subsection{Proof of the scaling limit \cref{thm: main intro}}

\subsubsection{Tightness in the Skorohod space}
\begin{Prop}[Tightness] 
\label{lem: tight CD cts}
The sequence
 \[
 \bigg(\bigg(\frac{\cS_{\lfloor nt \rfloor}}{\sqrt{n}}, \frac{\cD_{\lfloor nt \rfloor}}{v_n}\bigg)_{t\geq 0}, \,
 n\ge 1 \bigg)
 \]
 is tight 
 in the space of {càdlàg functions} with the local {Skorohod $J_1$} topology. 
\end{Prop}

\begin{proof}
We start by proving the statement for the second component $\cD_{\lfloor nt\rfloor}/v_n$.
Write 
\[ \Xn=(\Xn_t)_{t\ge 0}=\left(\frac{\cD_{\lfloor nt \rfloor}}{v_n}\right)_{t\ge 0}.\]
By \cite[Theorem 23.11]{kallenberg}, to prove tightness in the Skorokhod $J_1$ topology, it suffices to prove tightness at a fixed time (which is given by \cref{thm:UB_D_n}) and that for any sequence $(\tau_n)$ of predictable $\Xn$ stopping times, and any sequence $\varepsilon_n\to 0$, we have that 
\[
|\Xn_{\tau_n+\varepsilon_n}-\Xn_{\tau_n}|\to 0
\]
in probability as $n\to \infty$. Using our definition of $\Xn$ and the definition of the discrepancy, it in turn suffices to show that for any sequence $\varepsilon_n\to 0$ and $\delta>0$
\[
\mathbb{P}\left(\left|\frac{\cD_{n\varepsilon_n}}{v_{n}}\right|>\delta \, \Big| \, \mathcal{G}_0 \right)\to 0
\]
as $n\to \infty$, uniformly over the possible realisations of $(X(-k), k\in \mathbb{N})$. For this, we rewrite  
\[
\mathbb{P}\left(\left|\frac{\cD_{\lfloor n\varepsilon_n\rfloor }}{v_{n}}\right|>\delta\, \Big| \, \mathcal{G}_0 \right)=\mathbb{P}\left(\left|\frac{\cD_{\lfloor n\varepsilon_n \rfloor }}{v_{\lfloor n \varepsilon_n \rfloor}}\right|>\delta\frac{v_n}{v_{\lfloor n \varepsilon_n \rfloor}}\, \Big| \, \mathcal{G}_0 \right)
\]
Since 
\[
\frac{v_n}{v_{\lfloor n \varepsilon_n \rfloor}}\to \infty
\]
because $\varepsilon_n$ is assumed to converge to $0$, we conclude that the right hand side converges to $0$ using \cref{thm:UB_D_n}.
The first component converges {by Donsker's theorem,}
so in particular, is tight. 
Finally, tightness of both components implies tightness of the pair.
\end{proof}

\subsubsection{Identification of the limit}
\begin{Prop}[Identification of the limit] 
\label{lem: conv CD cts}
 For any sequence of integers, there exists $\rho\ge 0$ and a further subsequence, along which
 the convergence
 \[
 \Big(\frac{\cS_{\lfloor nt \rfloor}}{\sqrt{n}}, \frac{\cD_{\lfloor nt \rfloor}}{v_n}\Big)_{t\ge 0}
 \overset{\textnormal{d}}{\longrightarrow} (B^1_t,B^2_{\rho t})_{t\ge 0}
 \]
 holds in the space of {càdlàg functions} with the local {Skorohod $J_1$} topology, where $B^1,B^2$ are independent standard Brownian motions.
\end{Prop}

\begin{proof}
\cref{lem: tight CD cts} implies that every sequence of integers has a further subsequence along which $((\cS_{\lfloor nt \rfloor}/\sqrt{n},\cD_{\lfloor nt \rfloor}/v_n), t\ge 0)$ converges in distribution. 
We need to show that the limit is of the form $(B^1_{\cdot},B^2_{\rho \cdot})$ as in the statement of the lemma. 
The strategy is to show that any subsequential limit $\beta=((\beta^1_t,\beta^2_t), t\ge 0)$ of $((\cS_{\lfloor nt \rfloor}/\sqrt{n},\cD_{\lfloor nt \rfloor}/v_n), t\ge 0)$ has stationary and independent increments, and satisfies Brownian scaling, i.e. $\beta(rt) \overset{\text{d}}{=} \sqrt{r}\beta(t)$ for all $r,t\geq 0$. 

{We first prove that $((\beta^1_t,\beta^2_t), t\ge 0)$ has stationary and independent increments. Let $k\in\mathbb{N}^*$ and $0=t_0 < t_1<\ldots<t_k=t$.}
Fix $k\geq 1$, and write $X(1, \lfloor nt \rfloor)$ as the concatenation of $X(\lfloor nt_{i-1}\rfloor + 1, \lfloor nt_i \rfloor)$ with $1\leq i\leq k$. Considering their burger counts and discrepancies, we {can write
\[
\mathcal{S}_{\lfloor nt \rfloor}
=
\sum_{i=1}^k \mathcal{S}\big(\lfloor nt_{i-1} \rfloor + 1, \lfloor nt_{i} \rfloor\big)
\quad \text{and} \quad
\mathcal{D}_{\lfloor nt \rfloor}
=
\sum_{i=1}^k \mathcal{D}\big(\lfloor nt_{i-1} \rfloor + 1, \lfloor nt_{i} \rfloor\big).
\]
The random variables $\cS^i_{\lfloor nt_{i} \rfloor - \lfloor nt_{i-1} \rfloor} := \mathcal{S}\big(\lfloor nt_{i-1} \rfloor + 1, \lfloor nt_{i} \rfloor\big)$, $1\leq i\leq k$, are independent with respective laws $\cS(1,\lfloor nt_{i} \rfloor - \lfloor nt_{i-1} \rfloor)$.
Moreover, for each term in the second sum above, we can apply \cref{prop:approx_markov_F}, starting with the last term.
Taking care of rounding errors, we end up with a coupling with a family $(\mathcal{S}_{\lfloor nt_{i} \rfloor - \lfloor nt_{i-1} \rfloor}^{i},\mathcal{D}_{\lfloor nt_{i} \rfloor - \lfloor nt_{i-1} \rfloor}^{i})$, $1\leq i\leq k$, of independent variables with respective laws  $(\mathcal{S}_{\lfloor nt_{i} \rfloor - \lfloor nt_{i-1} \rfloor}, \mathcal{D}_{\lfloor nt_{i} \rfloor - \lfloor nt_{i-1} \rfloor})$ and an error term $\Delta_{n,k}$ such that
\begin{equation} \label{eq:discrep_ID}
\cS_{\lfloor nt \rfloor}
=
\mathcal{S}_{\lfloor nt_{1} \rfloor - \lfloor nt_{0} \rfloor}^{1}
+ \ldots +
\mathcal{S}_{\lfloor nt_{k} \rfloor - \lfloor nt_{k-1} \rfloor}^{k}
\quad \text{and} \quad
\mathcal{D}_{\lfloor nt \rfloor}
=
\mathcal{D}_{\lfloor nt_{1} \rfloor - \lfloor nt_{0} \rfloor}^{1}
+ \ldots +
\mathcal{D}_{\lfloor nt_{k} \rfloor - \lfloor nt_{k-1} \rfloor}^{k}
+\Delta_{n,k}.
\end{equation}
In addition, under the same coupling, the error term is bounded by\footnote{We stress that the first term $k$ comes from the possible rounding errors, while the sum comes from \cref{prop:approx_markov_F}.} 
\[
|\Delta_{n,k}|
\leq 
k + 2\sum_{i=1}^k \Nright(\lfloor nt_{i-1} \rfloor + 1, \lfloor nt_i \rfloor),
\]
where $\Nright(\lfloor nt_{i-1} \rfloor + 1, \lfloor nt_i \rfloor)$ denotes the number of $\rF$ symbols in $\{\lfloor nt_{i-1} \rfloor + 1, \ldots, \lfloor nt_i \rfloor \}$ that are matched to the left of $\lfloor nt_{i-1} \rfloor$. Each term $\Nright(\lfloor nt_{i-1} \rfloor + 1, \lfloor nt_{i} \rfloor)$ has the marginal law of $\Nright_{\lfloor n(t_{i}-t_{i-1}) \rfloor}$ up to rounding error, so \cref{prop:tightness_N_F} entails that $\frac{1}{v_n}|\Delta_{n, k}| \to 0$ as $n\to\infty$.
}
Going back to \eqref{eq:discrep_ID}, we conclude by the tightness in \cref{lem: tight CD cts} (possibly by extracting along a further subsequence) that as $n\to\infty$, the subsequential limit $\beta$ has stationary and independent increments.

Next, we show that $\beta$ has Brownian scaling. Indeed, for all fixed $r, t\geq 0$, setting $m=nr$ we have
\[
\bigg(\frac{\cS_{\lfloor nrt \rfloor}}{\sqrt{n}},\frac{\cD_{\lfloor nrt \rfloor}}{v_n}\bigg) 
=
\bigg(\frac{\cS_{\lfloor mt \rfloor}}{\sqrt{m/r}},\frac{\cD_{\lfloor mt \rfloor}}{v_{m/r}}\bigg).
\]
Since $v_{m/r}=\frac{\sqrt{m}}{\sqrt{r} \log(m/r)}$, sending $n\to \infty$ on the left-hand side and $m\to\infty$ on the right-hand side (along the subsequence), we end up with
\[
\beta(rt) \overset{\text{d}}{=} \sqrt{r}\beta(t).
\]

Now, the only Lévy process which has Brownian scaling is Brownian motion, as can be seen from the Lévy-Khintchine formula for stable processes \cite[Equation 3.7.11]{DurrettRick2019PTaE}. We conclude that $\beta$ has the law of (possibly correlated) planar Brownian motion $((B^1_{\alpha t}, B^2_{\rho t}), t\geq 0)$ with some variance. We also know that $\alpha=1$ since $\cS$ is a simple symmetric random walk. Finally, for any $s,t\ge 0$
$\mathbb{E}[B_t^1B_{\rho s}^2]=-\mathbb{E}[B_t^1B_{\rho s}^2]$ by symmetry of hamburgers and cheeseburgers (i.e.\ if the convergence holds along a subsequence for $(\cS,\cD)$, then it also holds along the subsequence for $(\cS,-\cD)$ and the limits coincide). Thus $B^1$ and $B^2$ have covariance zero, so are independent. This proves \cref{lem: conv CD cts}.
\end{proof}

Next, we show that the variance $\rho$ from \cref{lem: conv CD cts} is unique and non-zero, thus upgrading the statement of \cref{lem: conv CD cts} to convergence in distribution (not only along subsequences). 

\begin{Prop}[Characterisation of the limiting variance]
\label{lem: id rho}
Let $\rho \ge 0$ be such that
\[
\bigg(\frac{\cS_{\lfloor nt \rfloor}}{\sqrt{n}}, \frac{\cD_{\lfloor nt \rfloor}}{v_n}\bigg)_{t\ge 0}
\overset{\textnormal{d}}{\longrightarrow} (B^1_t,B^2_{\rho t})_{t\ge 0}
\]
along a subsequence with respect to the local Skorohod $J_1$ topology, where $B^1,B^2$ are independent standard Brownian motions. Then $\rho=4\pi^2$.
\end{Prop}

\begin{proof}
Consider a subsequence $(n_k)_{k\ge 0}$ and a $\rho\ge 0$ such that 
\[
\Big(\frac{\cS_{\lfloor n_kt \rfloor}}{\sqrt{n_k}}, \frac{\cD_{\lfloor n_kt \rfloor}}{v_{n_k}}\Big)_{t\ge 0}
\overset{\text{d}}{\longrightarrow}(B^1_t,B^2_{\rho t})_{t\ge 0}.
\]
Introduce $m_k=\inf\{m\in\mathbb{N},\, m^2/\log^4(m)\le n_k\}$, so that $m_k^2/\log^4(m_k)\sim n_k$ as $k\to \infty$.
Applying \cref{lem:concl_sigma} and passing to a further subsequence if necessary (we keep the notation the same for compactness) we have that
\[
\Big(\frac{\cS_{\lfloor n_k\cdot \rfloor}}{\sqrt{n_k}}, \frac{\sigma(m_k)}{n_k},\frac{\cD_{\lfloor n_k\cdot \rfloor}}{v_{n_k}},\frac{\cD_{\sigma(m_k)}}{m_k/\log^3(m_k)}\Big)\overset{\text{d}}{\longrightarrow}(B^1_\cdot,\sigma, B^2_{\rho \cdot},\zeta-\zeta')
\]
where the joint law of $(B^1_\cdot, \sigma)$ is that of a standard Brownian motion together with the first time that it hits $-1/a_1$, the joint law of $(B^1,B^2)$ is a pair of independent Brownian motions, and $\zeta$, $\zeta'$ are independent and have the marginal law given in \cref{thm:lazy_scaling}. 
On the one hand, $\sigma$ is measurable with respect to $B^1$, so $\sigma$ and $B^2$ are independent. On the other hand, since $m_k/\log^3(m_k)\sim \tfrac12v_{n_k}$ as $k\to \infty$, we have that 
\[B_{\rho \sigma} \overset{\text{d}}{=} (\zeta-\zeta').\]
Now, the right-hand side has the law of a centred Cauchy distribution with parameter {$\pi^3$}
and the left-hand side has the law of $\sqrt{\rho}$ times a centred Cauchy distribution with parameter {$\pi^2/2$} {(either by direct calculations or using the fact that Brownian motion subordinated at an independent Brownian hitting time is Cauchy distributed, see e.g.\ \cite[Proposition III.3.11]{revuz2013continuous})}.
This implies that $\rho=4\pi^2$.
\end{proof}

We are now ready to finish the proof of \cref{thm: main intro}.

\begin{proof}[Proof of \cref{thm: main intro}]
The statement of the theorem follows from \cref{lem: conv CD cts} and \cref{lem: id rho} if we restrict the processes to $t\ge 0$. We now prove the  statement also holds for the processes restricted to $t\ge -T$, for any $T>0$.
For all $t\geq 0$, we introduce
\[Z_t := \bigg(\frac{\cS_{\lfloor n(t-T) \rfloor}}{\sqrt{n}}, \frac{\cD_{\lfloor n(t-T) \rfloor}}{4\pi v_n}\bigg)
\quad \text{and} \quad 
\tilde{Z}_t := Z_{t-T}-Z_{-T}.
\]
By translation invariance, the process $(\tilde{Z}_t)_{t\geq 0}$ has the same law as $\big(\cS_{\lfloor nt \rfloor}/\sqrt{n}, \cD_{\lfloor nt \rfloor}/(4\pi v_n)\big)_{t\geq 0}$. We deduce from the convergence on $\R_+$ that $(\tilde{Z}_t,t\geq 0)$ converges in distribution, in the local Skorohod $J_1$ topology, to a pair $(\tilde{B}^1,\tilde{B}^2)$ of independent standard Brownian motions. Since $Z_t = \tilde{Z}_t - \tilde{Z}_{T}$ for all $t\geq 0$, this implies that the process $Z$ converges in the same functional sense to $(\tilde{B}^1_t-\tilde{B}^1_T,\tilde{B}^2_t-\tilde{B}^2_T)_{t\geq 0}$ which has the law of $(B^1_{t-T},B^2_{t-T})_{t\geq 0}$ where $B^1$ and $B^2$ are independent standard Brownian motions.
Since $T$ was arbitrary this concludes the proof of convergence as bi-infinite processes in the local Skorohod $J^1$ topology.
\end{proof}

\subsection{Proof of \cref{thm:variance}}
First, since 
$\frac{\cD_n}{v_n}$
is convergent (by \cref{thm: main intro}), {with $\mathbb{E}[\lim_{n\to \infty} (\frac{\cD_n}{v_n})^2]=4\pi^2$ we have that 
\begin{equation}\label{var:lower} 
4\pi^2\le \liminf_{n\to \infty} \, \mathbb{E}\bigg[\bigg(\frac{\cD_n}{v_n}\bigg)^2\bigg] =\liminf_{n\to \infty} \, \mathbb{E}\bigg[\bigg(\frac{\cD_n}{\sqrt{\text{Var}(\cD_n)}}\bigg)^2\bigg] \frac{\text{Var}(\cD_n)}{v_n^2}=\liminf_{n\to \infty}  \frac{\text{Var}(\cD_n)}{v_n^2}\end{equation}
by Fatou's lemma.} The next proposition gives a bound on the $\limsup$, which together with \eqref{var:lower} completes the proof of \cref{thm:variance}.
\begin{Prop}\label{lem:varupper}
We have
\begin{equation}\label{var:upper} \limsup_{n\to \infty} \frac{\textnormal{Var}(\cD_n)}{v_n^2}\le {8\pi^2}.
\end{equation}
\end{Prop}

\begin{remark}
{As the proof will show, the estimate is actually uniform over the stack at time $0$, i.e.\ it holds conditional on the $\sigma$-field $\Gcal_0$ introduced in \cref{prop:approx_markov_F}. One could therefore use the variance estimate to prove the tightness of the discrepancy as a process, following the proof of \cref{lem: tight CD cts}. However, we find the strategy in \cref{sec:final_cvg} more instructive (and it requires the same amount of work).}
\end{remark}

\begin{proof}
We first recall some notation and arguments from \cite[Section 3.1]{SheffieldScott2016QGAI}. Let $-{J}$ be the index of the first symbol in $(-\infty,0)$ whose match is in $[0,+\infty)$. 
    Notice that $X(-J)$ is either $\rh$ or $\rc$, and that when $X(0)=\rF$, we have $J=\varphi(0)$. Moreover, the reduced word $\overline{X(-J,-1)}$ can only consist of one burger production and a string of orders with the opposite type (and so importantly, no $\rF$ symbols). Therefore, $-\cS(-J,-1)+2=|\cD(-J,-1)|$.
    Furthermore, on the event that $X(0)=\rF$, we have $-\cS(-J,-1)+1 = \xi$ where $\xi$ is the step distribution of the reduced walk (\cref{sec: step_rw}).
    By \eqref{eq:xi_centred}, we deduce that  
    \[
    0=
    \mathbb{E}[\xi]
    =-\frac12+\frac14+\frac14\left(\mathbb{E}[-\cS(-J,-1)+1 \mid X(0)=\rF]+1\right)
    =-\frac14-\frac14\Eb[|\cD(-J,-1)|-1 \mid X(0)=\rF].
    \]
    In addition, we can remove the conditioning by independence, and so we find 
    \begin{equation} \label{eq:exp_abs_discr}
         2-\Eb[\cS(-J,-1)] = \Eb[|\mathcal{D}(-{J},-1)|] = 2.
    \end{equation}
    
    Let $n\geq 1$. We know that the discrepancy $\mathcal{D}_n$ is centred, so that $\Var(\mathcal{D}_n) = \Eb[\mathcal{D}_n^2]$.
    Write $\mathcal{D}(i):=\mathcal{D}(X(i))$ for the contribution to the discrepancy of symbol $i$, so that $\mathcal{D}_n = \sum_{i=0}^n \mathcal{D}(i)$ with $\mathcal{D}(i)\in \{-1,1\}$. Then we have
    \begin{align*}
  \mathbb{E}[\cD_{n+1}^2-\cD_n^2 ]
    & =
   \Eb \big[\mathcal{D}(n+1)^2 \big]
    +
    2 \Eb[\mathcal{D}(n+1)\mathcal{D}(1,n)] \\
   & =
   1+ 2 \Eb[\mathcal{D}(n+1)\mathcal{D}(1,n)]
    \end{align*}
    By translation invariance, we can rewrite this as 
    \begin{equation} \label{eq:var_discr_diff}
       \mathbb{E}[\cD_{n+1}^2-\cD_n^2 ]=
    1 +
    2 \Eb[\mathcal{D}(0)\mathcal{D}(-n,-1) ].
    \end{equation}
    Now, conditional on $\{X(0)\neq \rF\}$, the discrepancy $\mathcal{D}(0)$ is independent of $\mathcal{D}(-n,-1)$, which has expectation $0$.
    So, that part of the expectation vanishes and we end up with
    \[
    \mathbb{E}[\cD_{n+1}^2-\cD_n^2]
    =
    1 +
    2 \Eb[\mathcal{D}(0)\mathcal{D}(-n,-1) \mathds{1}_{X(0)=\rF} ].
    \]
    We split the expectation on the right-hand side according to whether ${J}\leq n$ or ${J}>n$.

    Let us start with the event that ${J} \le n$. On the event that $J=j$ with $j<n$, the reduced word $\overline{X(-j,0)}$ contains no $\rF$ symbols (by definition of ${J}$), so that $(J, X(-j,0))$ is independent of $X(-n,-j-1)$. 
    Hence, 
    \[
    \Eb[\mathcal{D}(0)\mathcal{D}(-n,-J-1) \mathds{1}_{X(0)=\rF} \mathds{1}_{{J}<n}]
    =
    \sum_{j<n}\Eb[\mathcal{D}(-n,-j-1) ] \cdot \Eb[\mathcal{D}(0)\mathds{1}_{X(0)=\rF}\mathds{1}_{{J}=j} ]
    =
    0,
    \]
   {since the discrepancy is centred.}
    Thus 
    \begin{align*}
    &\Eb[\mathcal{D}(0)\mathcal{D}(-n,-1) \mathds{1}_{X(0)=\rF} \mathds{1}_{{J}<n}] \\
    &=
    \Eb[\mathcal{D}(0)\mathcal{D}(-n,-{J}-1) \mathds{1}_{X(0)=\rF} \mathds{1}_{{J}<n}]
    +
    \Eb[\mathcal{D}(0)\mathcal{D}(-{J},-1) \mathds{1}_{X(0)=\rF} \mathds{1}_{{J}<n}] \\
    &=
    \Eb[\mathcal{D}(0)\mathcal{D}(-{J},-1) \mathds{1}_{X(0)=\rF} \mathds{1}_{{J}<n}].
    \end{align*}
    The identity clearly extends when ${J}=n$, so that
    \begin{equation} \label{eq:discr_J<n}
    \Eb[\mathcal{D}(0)\mathcal{D}(-n,-1)\mathds{1}_{X(0)=\rF} \mathds{1}_{{J} \leq n}]
    =
    \Eb[\mathcal{D}(0)\mathcal{D}(-{J},-1) \mathds{1}_{X(0)=\rF} \mathds{1}_{{J} \leq n}].
    \end{equation}
    
We also have  that
    \begin{equation} \label{eq:exp_D(0)_D(-J,-1)}
    \Eb[\mathcal{D}(0)\mathcal{D}(-{J},-1)\mathds{1}_{X(0)=\rF}]
    =
    -\frac12.
    \end{equation}
    Indeed, on the event that $X(0)=\rF$ (which has probability $p/2=1/4$), $\mathcal{D}(0)$ has sign opposite that of $\mathcal{D}(-{J},-1)$, so that by \eqref{eq:exp_abs_discr},
        \[\Eb[\mathcal{D}(0)\mathcal{D}(-{J},-1)\mathds{1}_{X(0)= \rF}] = 
        - {\frac14}\Eb[|\mathcal{D}(-{J},-1)|]
        =- \frac12,\]
        which is \eqref{eq:exp_D(0)_D(-J,-1)}.
    Using \eqref{eq:exp_D(0)_D(-J,-1)}, we can express \eqref{eq:discr_J<n} as 
    \[
    \Eb[\mathcal{D}(0)\mathcal{D}(-n,-1) \mathds{1}_{X(0)=\rF}\mathds{1}_{{J} \leq n} ]
    =
    -\frac12 - \Eb[\mathcal{D}(0)\mathcal{D}(-{J},-1) \mathds{1}_{X(0)=\rF} \mathds{1}_{{J} > n}].
    \]

    We use this identity to rewrite \eqref{eq:var_discr_diff} as
    \begin{align} \label{eq:var_diff_J>n}
    \mathbb{E}[\cD_{n+1}^2-\cD_n^2]
    &=
     2\Eb[\mathcal{D}(0)\big(\mathcal{D}(-n,-1) -\mathcal{D}(-J,-1)\big)\mathds{1}_{X(0)=\rF}\mathds{1}_{{J} > n} ].
    \end{align}
    Now:
    \begin{itemize}
        \item For reasons already mentioned below \eqref{eq:exp_D(0)_D(-J,-1)}, $\mathcal{D}(0)\mathcal{D}(-{J},-1) = -|\mathcal{D}(-{J},-1)|$ on the event $X(0)=\rF$. Moreover, recall from the beginning of the proof that $|\cD(-{J},-1)| = -\cS(-{J},-1)+2$. Therefore
        \[-2\mathbb{E}[\cD(0)\mathcal{D}(-{J},-1)\mathds{1}_{X(0)=\rF}\mathds{1}_{{J}>n}]=2\mathbb{E}[(-\cS(-{J},-1)+2)\mathds{1}_{X(0)=\rF}\mathds{1}_{{J}>n}].\]
        \item 
        If ${J}>n$, the reduced word $X(-n,-1)$ can consist only of orders, which implies $|\mathcal{D}(-n,-1)| \leq -\mathcal{S}(-n,-1)$.
        Therefore 
        \begin{align*}
        2\mathbb{E}[\mathcal{D}(0)\mathcal{D}(-n,-1)\mathds{1}_{X(0)=\rF}\mathds{1}_{{J}>n} ]
        & \le 2\mathbb{E}[-\mathcal{S}(-n,-1)\mathds{1}_{X(0)=\rF}\mathds{1}_{{J}>n} ].
         \end{align*}
Furthermore, $(\cS(-j,-1), j\ge 1)$ is independent of $X(0)$ and the process $(\cS(-j,-1), j\geq 1)$ is a centred martingale (for the filtration $\sigma(X(-j), j\geq 1)$). So by optional stopping, 
        \begin{align*}    
      0 & = \Eb[\cS(-({J}\wedge n),-1)] \\
      & = 
     \Eb[\cS(-n,-1) \mathds{1}_{{J}>n}]+\Eb[\cS(-{J},-1) \mathds{1}_{{J} \le  n}]\\
     &=\Eb[\cS(-n,-1) \mathds{1}_{{J}>n}]+\Eb[\cS(-{J},-1)]- \Eb[\cS(-{J},-1) \mathds{1}_{{J} > n}] \\
      &=\Eb[\cS(-n,-1) \mathds{1}_{{J}>n}]- \Eb[\cS(-{J},-1) \mathds{1}_{{J} > n}],
     \end{align*}
     where the last line follows since $\mathbb{E}[\cS(-{J},-1)]=0$ {by \eqref{eq:exp_abs_discr}}. 
    \end{itemize}
    We can therefore bound the difference in \eqref{eq:var_diff_J>n} by
    \begin{align*} \label{eq:var_diff_count}
   \mathbb{E}[\cD_{n+1}^2-\cD_n^2]
    &\leq 
    4 \Eb[(-\cS(-{J},-1)+1) \mathds{1}_{X(0)=\rF} \mathds{1}_{{J}>n}]. 
    \end{align*} 
{Observe that, conditional on $X(0)=\rF$, the random variable $(J,-\cS(-{J},-1)+1)$ has the law of $(\eta-1,\xi)$ where $\eta$ is as in \cref{sec: joint laws}.}
But now if $n \ge 1$, recalling the definition of $(\eta,\xi)$, we have $\eta>n$ if this step in the reduced walk corresponds to an $\rF$-excursion. This means that 
\[\mathbb{E}[\xi\mathds{1}_{\eta{-1}>n}]=\Eb[(-\cS(-{J},-1)+1)\mathds{1}_{X(0)=\rF} \mathds{1}_{{J}>n}]\]
and so for any $M\ge 0$,
\begin{align*}
    \limsup_{n\ge 0}\log^2(n)\mathbb{E}[\cD_{n+1}^2-\cD_n^2]
& \leq 4 \limsup_{n\ge 0} \log^2(n) \mathbb{E}[\xi \mathbf{1}_{\eta>n{+1}}] \\
& \le \limsup_{n\ge 0} \log^2(n) \left(4 \mathbb{E}[\xi1_{\xi>\frac{\sqrt{n}}{M}}]+4\frac{\sqrt{n}}{M}\mathbb{P}(\eta>n{+1})\right)\\
& \le {8\pi^2} + \frac{{(2\pi)^{3/2}}}{M},
\end{align*}
where the last line follows by {\eqref{eq:Exik} and \eqref{eq:asympeta}}. Since $M$ was arbitrary we see that 
\[\limsup_{n\ge 0} \log^2(n)\mathbb{E}[\cD_{n+1}^2-\cD_n^2] 
\le 
{8\pi^2}.\]
Summing over $n$, this implies that 
\[
\limsup_{n\to \infty} \frac{\log^2(n)}{n}\mathrm{Var}(\mathcal{D}_n)
\leq 
8\pi^2.
\]
This concludes the proof of \cref{lem:varupper}.
\end{proof}

\appendix
\section{Asymptotics of integrals}
\label{appendix_asymp}
In this appendix we prove some asymptotics for integrals that are useful in various places where we analyse Laplace transforms. 

\begin{lemma}\label{lem: integral expansion}
The following integral estimate holds:
\begin{equation}\label{eq: integral expansion}
\int^1_0\log\bigg(\frac{1+\sqrt{1-u^2}}{u}\bigg)\frac{\mathrm{d}u}{\varepsilon + u} = \frac{1}{2}\log^2\Big(\frac{1}{\varepsilon}\Big) + \log 2\cdot \log\Big(\frac{1}{\varepsilon}\Big) + C + o(1)
\quad \text{as } \varepsilon\to 0. 
\end{equation}
\end{lemma}

\begin{proof}
We split the logarithm as 
\begin{equation} \label{eq:log_sqrt_split}
\log\left(\frac{1+\sqrt{1-u^2}}{u}\right) = \log\left(\frac{1+\sqrt{1-u^2}}{2}\right) + \log\left(\frac{2}{u}\right), 
\end{equation}
and estimate the corresponding integrals separately.
Note that $\log((1+\sqrt{1-u^2})/2)\sim -u^2/4$ as $u\to 0$.  Then by dominated convergence, 
\begin{equation} \label{eq:log_1st_term_dom}
\int_{0}^{1} \log\left(\frac{1+\sqrt{1-u^2}}{2}\right)\frac{\mathrm{d}u}{\varepsilon + u} \mathrm{d}u \longrightarrow
\int_{0}^{1} \log\left(\frac{1+\sqrt{1-u^2}}{2}\right)\frac{{\mathrm{d}u}}{u} \qquad \mbox{as } \varepsilon \to 0. 
\end{equation}

For the other integral term, we use the change of variables $u = \frac{1-s}{s}v$ to obtain 
\begin{equation} \label{eq:split_log_pi}
\int^1_0 \log\Big(\frac{2}{u}\Big)\frac{\mathrm{d}u}{u+\varepsilon} = \int^{\frac{1}{\varepsilon}}_0 \bigg(\log\Big(\frac{1}{x}\Big) + \log\Big(\frac{2}{\varepsilon}\Big)\bigg)\frac{\mathrm{d}x}{1+x}. 
\end{equation}
The second term can be explicitly calculated as follows:
\begin{align}
\log\Big(\frac{2}{\varepsilon}\Big)\int_{0}^{\frac{1}{\varepsilon}}\frac{\mathrm{d}x}{1+x} = \log\Big(\frac{2}{\varepsilon}\Big)\cdot\log\Big(1 + \frac{1}{\varepsilon}\Big) = \log^{2}\Big(\frac{1}{\varepsilon}\Big) + \log 2\cdot \log\Big(\frac{1}{\varepsilon}\Big) + o(1)
\quad \text{as } \varepsilon\to 0. \label{eq:int_log_expansion_2nd_term}
\end{align}
In the first term we write the integral as 
\[\int^{\frac{1}{\varepsilon}}_{0}\log \Big(\frac{1}{x}\Big)\frac{\mathrm{d}w}{1+x}
= -\int^{\frac{1}{\varepsilon}}_{0}\log(1+x)\frac{\mathrm{d}w}{1+x} + \int^{\frac{1}{\varepsilon}}_{0}\log\left(\frac{1+x}{x}\right)\frac{\mathrm{d}w}{1+x}. \]
The second integral converges as $\varepsilon\to 0$ to 
\[\int^{\infty}_{0}\log\left(\frac{1+x}{x}\right)\frac{\mathrm{d}w}{1+x} = \frac{\pi^2}{6}. 
\]
The first integral is again explicit with 
\begin{equation}\label{eq:int_log_expansion_1st_term}
\int^{\frac{1}{\varepsilon}}_{0}\log(1+x)\frac{\mathrm{d}w}{1+x} = \frac{1}{2}\log^2\Big(1+\frac{1}{\varepsilon}\Big) = \frac{1}{2}\log^2\Big(\frac{1}{\varepsilon}\Big) + o(1). 
\end{equation}
Combining \eqref{eq:int_log_expansion_2nd_term} and \eqref{eq:int_log_expansion_1st_term}, we obtain that as $\varepsilon\to 0$,
\begin{equation*}
\int^1_0 \log\Big(\frac{2}{u}\Big)\frac{\mathrm{d}u}{u+\varepsilon} 
= \frac{1}{2}\log^2\Big(\frac{1}{\varepsilon}\Big) + \log 2\cdot \log\Big(\frac{1}{\varepsilon}\Big) + \frac{\pi^2}{6} + o(1).  
\end{equation*}
Together with \eqref{eq:log_1st_term_dom}, we conclude that
\begin{equation} \label{eq_log^2_term}
\int^1_0\log\bigg(\frac{1+\sqrt{1-u^2}}{u}\bigg)\frac{\mathrm{d}u}{\varepsilon + u} = \frac{1}{2}\log^2\Big(\frac{1}{\varepsilon}\Big) + \log 2\cdot \log\Big(\frac{1}{\varepsilon}\Big) + C + o(1), 
\end{equation}
where 
\[
C = \frac{\pi^2}{6} + \int_{0}^{1} \log\left(\frac{1+\sqrt{1-u^2}}{2}\right)\frac{{\mathrm{d}u}}{u}. 
\]
\end{proof}

\begin{lemma}
\label{appendix:Laplace_H_PF}
For $\lambda \ge 0$, write $F(\lambda)$ for the Laplace transform of $\xi$, see \eqref{eq:hat_F_xi}. Set $\varepsilon = \varepsilon(\lambda) := \frac{2}{\pi}(F(\lambda) - 1)\ge 0$. Then, as  $\lambda\to 0$,
\begin{multline*}
\frac{2}{\pi^3} \mathrm{e}^{\lambda}\int^1_0\int^1_0 \log\bigg(\frac{1+\sqrt{1-u^2}}{u}\bigg) \log\bigg(\frac{1+\sqrt{1-v^2}}{v}\bigg)\frac{u(1-\pi v/2)}{u+\varepsilon} \frac{\mathrm{d}u\mathrm{d}v}{u+v} \\
=  1-c_1 \lambda\log^2\Big(\frac{1}{\lambda}\Big) {-} c_2 \lambda\log\Big(\frac{1}{\lambda}\Big)\cdot \log\log\Big(\frac{1}{\lambda}\Big) {+} c_3 \lambda\log\Big(\frac{1}{\lambda}\Big) + o\Big(\lambda\log\Big(\frac{1}{\lambda}\Big)\Big), 
\end{multline*}
where $c_1$, $c_2$ and $c_3$ are positive and given by:
\[
{
c_1
=
\frac{1}{4\pi^2},
\quad 
c_2
=
\frac{2}{\pi^2} 
\quad \text{and} \quad 
c_3
=
\frac{\log(\pi/2)}{\pi^2}.}
\]
\end{lemma}

\begin{proof}
Note that by the identity 
\[
\mathbb{E}[\mathrm{e}^{-\lambda\mathcal{H}^*(P_{\rF})}]
=
\frac{2}{\pi^3} \mathrm{e}^{\lambda}\int^1_0\int^1_0 \log\bigg(\frac{1+\sqrt{1-u^2}}{u}\bigg) \log\bigg(\frac{1+\sqrt{1-v^2}}{v}\bigg)\frac{u(1-\pi v/2)}{u+\varepsilon} \frac{\mathrm{d}u\mathrm{d}v}{u+v}
\]
in \cref{prop:Laplace_H_PF}, setting $\lambda = 0$ (then $\varepsilon = 0$), we get the following equation:
\begin{equation}\label{eq: normalisation double integral}
1 = 
\frac{{2}}{\pi^3} \int^1_0\int^1_0 \log\bigg(\frac{1+\sqrt{1-u^2}}{u}\bigg) \log\bigg(\frac{1+\sqrt{1-v^2}}{v}\bigg) \frac{1-\pi v/2}{u+v}\mathrm{d}u\mathrm{d}v. 
\end{equation}

We therefore obtain that 
\begin{multline}\label{eq: Laplace H PF - 1}
\frac{2}{\pi^3}\mathrm{e}^{\lambda}\int^1_0\int^1_0 \log\bigg(\frac{1+\sqrt{1-u^2}}{u}\bigg) \log\bigg(\frac{1+\sqrt{1-v^2}}{v}\bigg)\frac{u(1-\pi v/2)}{u+\varepsilon} \frac{\mathrm{d}u\mathrm{d}v}{u+v} - 1 \\
= (\mathrm{e}^\lambda - 1) 
-  \frac{{2}}{\pi^3} \varepsilon \mathrm{e}^{\lambda}\int^1_0\int^1_0 \log\bigg(\frac{1+\sqrt{1-u^2}}{u}\bigg) \log\bigg(\frac{1+\sqrt{1-v^2}}{v}\bigg)\frac{1-\pi v/2}{u+\varepsilon} \frac{\mathrm{d}u\mathrm{d}v}{u+v}.
\end{multline}
\noindent Note that $\mathrm{e}^{\lambda} -1\sim \lambda$ and  
$\varepsilon\sim \pi \lambda/\log^2(\lambda)$
as $\lambda \to 0$. At the precision of \cref{appendix:Laplace_H_PF} we only need to focus on the integral in \eqref{eq: Laplace H PF - 1}. Specifically, comparing with our desired asymptotic, we aim to see that the integral has leading term of order $\log^3(\lambda)$. First, we split the logarithm as 
\[
\log\bigg(\frac{1+\sqrt{1-u^2}}{u}\bigg) = \log\bigg(\frac{1+\sqrt{1-u^2}}{2}\bigg) + \log\Big(\frac{2}{u}\Big),
\quad u\in (0,1),
\]
and write, accordingly, the above integral as a sum of two terms:
\[\int^1_0\int^1_0 \log\bigg(\frac{1+\sqrt{1-u^2}}{u}\bigg) \log\bigg(\frac{1+\sqrt{1-v^2}}{v}\bigg)\frac{1-\pi v/2}{u+\varepsilon} \frac{\mathrm{d}u\mathrm{d}v}{u+v} = I_1 + I_2.\]

\noindent By Lebesgue's dominated convergence theorem, there is a constant $C>0$ such that 
\begin{equation}\label{eq: Laplace H PF I_1}
I_1 = \int^1_0\int^1_0 \log\bigg(\frac{1+\sqrt{1-u^2}}{2}\bigg) \log\bigg(\frac{1+\sqrt{1-v^2}}{v}\bigg)\frac{1-\pi v/2}{u+\varepsilon} \frac{\mathrm{d}u\mathrm{d}v}{u+v} = C+o(1)
\quad \text{as } \varepsilon\to0. 
\end{equation}
This leaves us with estimating 
\begin{equation}\label{eq: Laplace H PF I_2 display}
I_2 = \int^1_0 \log\Big(\frac{2}{u}\Big) \frac{\mathrm{d}u}{u+\varepsilon}\Bigg(\int^1_0\log\bigg(\frac{1+\sqrt{1-v^2}}{v}\bigg)\frac{1-\pi v/2}{u+v}\mathrm{d}v\Bigg).
\end{equation}

{
We first take a look at the inner integral, writing it as
\begin{multline*}
\int^1_0 \log\bigg(\frac{1+\sqrt{1-v^2}}{v}\bigg)\frac{1-\pi v/2}{u+v}\mathrm{d}v \\
= 
(1+\pi u/2)\int^1_0 \log\bigg(\frac{1+\sqrt{1-v^2}}{v}\bigg)\frac{\mathrm{d}v}{u+v} - \frac{\pi}{2}\int^1_0 \log\bigg(\frac{1+\sqrt{1-v^2}}{v}\bigg)\mathrm{d}v.
\end{multline*}
By Lebesgue's dominated convergence theorem, the first integral above is a continuous function of $u$, as long as $u>0$. On the other hand, as $u\to 0$, it has the blow-up prescribed by \cref{lem: integral expansion}. 
}
In conclusion, there exists a continuous function $u\mapsto C(u)$ on $[0, 1]$, such that
\[
\int^1_0 \log\bigg(\frac{1+\sqrt{1-v^2}}{v}\bigg)\frac{1-\pi v/2}{u+v}\mathrm{d}v 
= \frac{1}{2}\log^2\Big(\frac{1}{u}\Big) + \log(2)\log\Big(\frac{1}{u}\Big) + C(u). 
\]
Then \eqref{eq: Laplace H PF I_2 display} boils down to
\[
I_2 = \int^1_0 \bigg(
\frac{1}{2}\log^3\Big(\frac{1}{u}\Big)
+\frac{3}{2}\log(2)\log^2\Big(\frac{1}{u}\Big) 
+\Big(C(u)+\log^2(2)\Big) \log\Big(\frac{1}{u}\Big)
+ \log(2) C(u)
\bigg) \frac{\mathrm{d}u}{u+\varepsilon}. 
\]
{We deal with all these terms using \cref{lem:int_log^k} below. Bounding the function $C$ by a constant, it yields the expansion:} 

\[I_2 = \frac{1}{8}\log^4\Big(\frac{1}{\varepsilon}\Big) + \frac{1}{2}\log(2) \log^3\Big(\frac{1}{\varepsilon}\Big) + O\Big(\log^2\Big(\frac{1}{\varepsilon}\Big)\Big). \]

{
\noindent Recalling that $\varepsilon = \frac{2}{\pi}(F(\lambda) - 1)$,} we get by \cref{prop:xi} that
\begin{equation} \label{eq:eps_lambd_asympt}
\varepsilon = \frac{2}{a_1\pi}\frac{\lambda}{\log^2(\frac{1}{\lambda})}\bigg(1 -\frac{4\log\log(\frac{1}{\lambda})}{\log(\frac{1}{\lambda})} {-} \frac{2a_1\log a_1 {+} a_2}{{a_1}}\frac{1}{\log(\frac{1}{\lambda})} + o\left(\frac{{1}}{\log(\frac{1}{\lambda})}\right)\bigg)
\quad {\text{as } \lambda\to 0.}
\end{equation}
In particular
\[
\log\Big(\frac{1}{\varepsilon}\Big) = \log\Big(\frac{1}{\lambda}\Big)+2\log\log\Big(\frac{1}{\lambda}\Big) + \log \frac{\pi a_1}{2} + o(1)
\quad \text{as } \lambda\to 0, 
\]
with $a_1$, $a_2$ as in \cref{L:r(s)}.
It entails that
\[
I_2 = \frac{1}{8}\log^4\Big(\frac{1}{\lambda}\Big) + \log^3\Big(\frac{1}{\lambda}\Big) \cdot \log\log\Big(\frac{1}{\lambda}\Big) + \frac{1}{2}\log(\pi a_1) \log^3\Big(\frac{1}{\lambda}\Big) + o\Big(\log^3\Big(\frac{1}{\lambda}\Big)\Big). 
\]
Together with the asymptotics $\mathrm{e}^{\lambda} = 1+\lambda+o(\lambda)$ {and \eqref{eq:eps_lambd_asympt}}, we conclude from \eqref{eq: Laplace H PF - 1} that there are positive constants $c_1,c_2,c_3$ such that, as $\lambda\to 0$,
\[
\mathbb{E}[\mathrm{e}^{-\lambda\mathcal{H}^*(P_{\rF})}] = 1-c_1 \lambda\log^2\Big(\frac{1}{\lambda}\Big) {-} c_2 \lambda\log\Big(\frac{1}{\lambda}\Big)\cdot \log\log\Big(\frac{1}{\lambda}\Big) {+} c_3 \lambda\log\Big(\frac{1}{\lambda}\Big) + o\Big(\lambda\log\Big(\frac{1}{\lambda}\Big)\Big).  
\]
Tracing the constants, this completes the proof of \cref{appendix:Laplace_H_PF}.
\end{proof}

{
\begin{lemma}
\label{lem:int_log^k}
The following integral estimate holds: for all $k\in \mathbb{N}$,
\[
\int^1_0 \log^k\Big(\frac{1}{u}\Big)\frac{\mathrm{d}u}{u+\varepsilon} = \frac{1}{k+1}\log^{k+1}\Big(\frac{1}{\varepsilon}\Big) + o(1)
\quad \text{as } \varepsilon\to 0.
\]
\end{lemma}

\begin{proof}
    Let $k\in\mathbb{N}$.
    We first replace the integral with a more symmetric one:
    \[
    I(\varepsilon)
    :=
    \int^1_0 \log^k\Big(\frac{1}{u+\varepsilon}\Big)\frac{\mathrm{d}u}{u+\varepsilon}
    =
    \bigg[ - \frac{1}{k+1} \log^{k+1}\Big(\frac{1}{u+\varepsilon}\Big)\bigg]_{0}^1
    =
    \frac{1}{k+1}\log^{k+1}\Big(\frac{1}{\varepsilon}\Big) + o(1)
    \quad \text{as } \varepsilon\to 0.
    \]
    It remains to show that the difference
    \[
    \int^1_0 \log^k\Big(\frac{1}{u}\Big)\frac{\mathrm{d}u}{u+\varepsilon}
    - I(\varepsilon)
    =
    \int^1_0 \Big(\log^k\Big(\frac{1}{u}\Big)- \log^k\Big(\frac{1}{u+\varepsilon}\Big)\Big) \frac{\mathrm{d}u}{u+\varepsilon}
    \]
    goes to a constant as $\varepsilon\to 0$. But, for all $u\in (0,1)$, one can check that 
    $
    \log^k(1/u)- \log^k(1/({u+\varepsilon}))
    =
    \varepsilon k u \log(1/u) +o(\varepsilon)
    \quad \text{as } \varepsilon\to 0.
   $
    We conclude the proof by an application of Lebesgue's dominated convergence theorem. 
\end{proof}
}

%
%

\section{Numerical simulations}
\label{appendix:sim}

\begin{figure*}[h!]
    \centering
    \begin{subfigure}[t]{0.45\textwidth}
        \centering
        \includegraphics[width=0.85\textwidth]{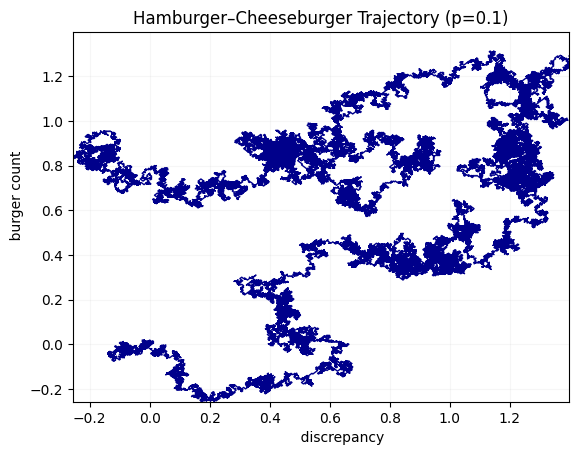}
        \caption{}
    \end{subfigure}%
    ~ 
    \begin{subfigure}[t]{0.45\textwidth}
        \centering
        \includegraphics[width=0.85\textwidth]{Figures/100000_sim_0_4.png}
        \caption{}
    \end{subfigure}%
    
    \begin{subfigure}[t]{0.45\textwidth}
        \centering
        \includegraphics[width=0.85\textwidth]{Figures/100000_sim_0_5.png}
        \caption{}
    \end{subfigure}%
    
    \begin{subfigure}[t]{0.45\textwidth}
        \centering
        \includegraphics[width=0.85\textwidth]{Figures/100000_sim_0_6.png}
        \caption{}
    \end{subfigure}%
    ~ 
    \begin{subfigure}[t]{0.45\textwidth}
        \centering
        \includegraphics[width=0.85\textwidth]{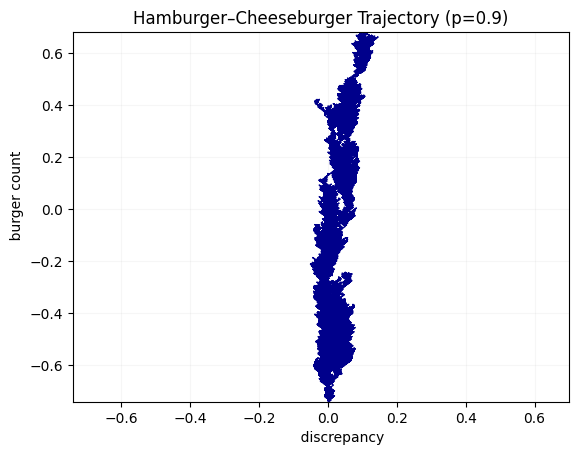}
        \caption{}
    \end{subfigure}
    \caption{Numerical simulations of hamburger-cheeseburger trajectories for $p\in\{0.1,0.4,0.5,0.6,0.9\}$. We took $N=10^6$ steps and represented the discrepancy and burger count, both rescaled by $\sqrt{N}$. We see that the trajectory becomes thinner and thinner as $p$ increases. Indeed, the rescaled discrepancy collapses when $p\geq 1/2$ by Sheffield's result \cite{SheffieldScott2016QGAI}. We stress that there is no \textit{a priori} monotonicity in $p$ in the inventory accumulation model. At the critical value $p=1/2$, which is the focus of this paper, the $\sqrt{N}$ scaling is only off by a logarithmic factor, which is very subtle to detect from simulations.}
\end{figure*}

\bibliographystyle{alpha}
\bibliography{biblio}

@book{kallenberg,
  author    = {O. Kallenberg},
  title     = {Foundations of {M}odern {P}robability},
  edition   = {3rd},
  publisher = {Springer},
  series    = {Probability {T}heory and {S}tochastic {M}odelling},
  volume    = {99},
  year      = {2021}
}

@book{VaartA.W.vander2023Wcae,
  series = {Springer {S}eries in {S}tatistics},
  publisher = {Springer},
  year = {2023},
  title = {Weak convergence and empirical processes: with applications to statistics},
  edition = {Second},
  author = {Vaart, A.W. van der and Wellner, J.A.}
}

@book{MandalB.N2011Asie,
  publisher = {{CRC} {P}ress},
  year = {2011},
  title = {Applied singular integral equations},
  author = {Mandal, B. N. and Chakrabarti, A.}
}

@article{borot2012recursive,
year = {2011},
volume = {45},
number = {4},
author = {G. Borot and J. Bouttier and E. Guitter},
title = {A recursive approach to the ${O}(n)$ model on random maps via nested loops},
journal = {Journal of {P}hysics {A}: {M}athematical and {T}heoretical}
}

@article{borot2012loop,
  title={Loop models on random maps via nested loops: the case of domain symmetry breaking and application to the {P}otts model},
  author={Borot, G. and Bouttier, J. and Guitter, E.},
  journal={Journal of {P}hysics {A}: {M}athematical and {T}heoretical},
  volume={45},
  number={49},
  year={2012}
}

@article{borot2012more,
  title={More on the ${O}(n)$ model on random maps via nested loops: loops with bending energy},
  author={Borot, G. and Bouttier, J. and Guitter, E.},
  journal={Journal of {P}hysics {A}: {M}athematical and {T}heoretical},
  volume={45},
  number={27},
  year={2012}
}

@article{budd2018peeling,
  title={The peeling process on random planar maps coupled to an ${O}(n)$ loop model (with an appendix by {L}inxiao {C}hen)},
  author={Budd, T.},
  journal={arXiv:1809.02012},
  year={2018}
}

@article {HoldenSun,
    AUTHOR = {Holden, N. and Sun, X.},
     TITLE = {Convergence of uniform triangulations under the {C}ardy
              embedding},
   JOURNAL = {Acta Math.},
  FJOURNAL = {Acta Mathematica},
    VOLUME = {230},
      YEAR = {2023},
    NUMBER = {1},
     PAGES = {93--203}
}

@article {LQGTBM,
	AUTHOR = {Miller, J. and Sheffield, S.},
	TITLE = {Liouville quantum gravity and the {B}rownian map {III}: the
	conformal structure is determined},
	JOURNAL = {Probab. {T}heory {R}elated {F}ields},
	FJOURNAL = {Probability {T}heory and {R}elated {F}ields},
	VOLUME = {179},
	YEAR = {2021},
	NUMBER = {3-4},
	PAGES = {1183--1211}
}

@article {DKRV,
	AUTHOR = {David, F. and Kupiainen, A. and Rhodes, R.
	and Vargas, V.},
	TITLE = {Liouville quantum gravity on the {R}iemann sphere},
	JOURNAL = {Comm. {M}ath. {P}hys.},
	FJOURNAL = {Communications in {M}athematical {P}hysics},
	VOLUME = {342},
	YEAR = {2016},
	NUMBER = {3},
	PAGES = {869--907}
}

@article {DuplantierSheffield,
	AUTHOR = {Duplantier, B. and Sheffield, S.},
	TITLE = {Liouville quantum gravity and {KPZ}},
	JOURNAL = {Invent. {M}ath.},
	FJOURNAL = {Inventiones {M}athematicae},
	VOLUME = {185},
	YEAR = {2011},
	NUMBER = {2},
	PAGES = {333--393}
}

@article{lqgcle4,
AUTHOR = {Ang, M. and Gwynne, E.},
TITLE = {Critical {L}iouville quantum gravity and {CLE}$_4$},
JOURNAL = {{Annales de l'{I.H.P.} {P}robabilités et statistiques} },
VOLUME = {(to appear)},
YEAR = {2025}}

@article{polyakovstrings,
	AUTHOR = {Polyakov, A. M.},
	TITLE = {Quantum geometry of bosonic strings},
	JOURNAL = {Phys. {L}ett. {B}},
	FJOURNAL = {Physics {L}etters. {B}. Particle Physics, Nuclear Physics and
	Cosmology},
	VOLUME = {103},
	YEAR = {1981},
	NUMBER = {3},
	PAGES = {207--210}
}

@article{knizhnik1988fractal,
  title={Fractal structure of 2d-quantum gravity},
  author={Knizhnik, V.G. and Polyakov, A. M. and Zamolodchikov, A. B.},
  journal={Modern {P}hysics {L}etters {A}},
  volume={3},
  number={08},
  pages={819--826},
  year={1988},
  publisher={World {S}cientific}
}

@article{gaudin1989n,
  title={O(n) model on a fluctuating planar lattice. {S}ome exact results},
  author={Gaudin, M. and Kostov, I.},
  journal={Physics {L}etters {B}},
  volume={220},
  number={1-2},
  pages={200--206},
  year={1989}
}

@article{doney1982exact,
  title={On the exact asymptotic behaviour of the distribution of ladder epochs},
  author={Doney, R. A.},
  journal={Stochastic {P}rocesses and their {A}pplications},
  volume={12},
  number={2},
  pages={203--214},
  year={1982},
  publisher={Elsevier}
}

@article{GwynneEwain2019Slft,
  issn = {0246-0203},
  journal = {Annales de l'{I.H.P.} {P}robabilités et statistiques},
  volume = {55},
  number = {1},
  year = {2019},
  title = {Scaling limits for the critical {F}ortuin–{K}asteleyn model on a random planar map {I}: {C}one times},
  language = {eng},
  author = {Gwynne, E. and Mao, C. and Sun, X.},
}

@article {tutteLQG,
    AUTHOR = {Gwynne, E. and Miller, J. and Sheffield, S.},
     TITLE = {The {T}utte embedding of the mated-{CRT} map converges to
              {L}iouville quantum gravity},
   JOURNAL = {Ann. {P}robab.},
  FJOURNAL = {The {A}nnals of {P}robability},
    VOLUME = {49},
      YEAR = {2021},
    NUMBER = {4},
     PAGES = {1677--1717}
}

@article{berestycki2017critical,
  title={Critical exponents on {F}ortuin--{K}asteleyn weighted planar maps},
  author={Berestycki, N. and Laslier, B. and Ray, G.},
  journal={Communications in {M}athematical {P}hysics},
  volume={355},
  pages={427--462},
  year={2017},
  publisher={Springer}
}

@book{bingham1989regular,
  title={Regular variation},
  author={Bingham, N. H. and Goldie, C. M. and Teugels, J. L.},
  series={Encyclopedia of {M}athematics and its {A}pplications},
  number={27},
  year={1989},
  publisher={Cambridge university press}
}

@article{bingham1974asymptotic,
  title={Asymptotic properties of supercritical branching processes {I}: The {G}alton-{W}atson process},
  author={Bingham, N. H. and Doney, R. A.},
  journal={Advances in {A}pplied {P}robability},
  volume={6},
  number={4},
  pages={711--731},
  year={1974},
  publisher={Cambridge University Press}
}

@book{sato1999levy,
  title={L{\'e}vy processes and infinitely divisible distributions},
  author={Sato, K.},
  volume={68},
  year={1999},
  publisher={Cambridge university press}
}

@article{SheffieldScott2016QGAI,
journal = {The {A}nnals of {P}robability},
pages = {3804--3848},
volume = {44}, 
number = {6}, 
year = {2016}, 
title = {QUANTUM GRAVITY AND INVENTORY ACCUMULATION}, 
author = {Sheffield, S.}, 
}

@book{BerestyckiNathanaël2024Gffa, 
year = {2024}, 
title = {Gaussian free field and {L}iouville quantum gravity}, 
author = {Berestycki, N. and Powell, E.}, 
publisher = {Cambridge university press}
}

@article{ChenLinxiao2017Bpot,
journal = {Annales de l'{I}nstitut {H}enri {P}oincaré. {D}. {C}ombinatorics, physics and their interactions},
pages = {245--271},
volume = {4}, 
number = {3},
year = {2017},
title = {Basic properties of the infinite critical-{FK} random map},
author = {Chen, L.},
}

@book{curien2019peeling,
  title={Peeling random planar maps},
  author={Curien, N.},
  year={2019},
  publisher={Springer}
}

@article{bernardi08,
	AUTHOR = {Bernardi, O.},
	TITLE = {A characterization of the {T}utte polynomial via combinatorial
	embeddings},
	JOURNAL = {Ann. {C}omb.},
	FJOURNAL = {Annals of {C}ombinatorics},
	VOLUME = {12},
	YEAR = {2008},
	NUMBER = {2},
	PAGES = {139--153}
}

@article{mullin67,
	AUTHOR = {Mullin, R. C.},
	TITLE = {On the enumeration of tree-rooted maps},
	JOURNAL = {Canadian {J}. {M}ath.},
	FJOURNAL = {Canadian {J}ournal of {M}athematics.},
	VOLUME = {19},
	YEAR = {1967},
	PAGES = {174--183}
}

@article{DuplantierMillerSheffield,
	AUTHOR = {Duplantier, B. and Miller, J. and Sheffield, S.},
	TITLE = {Liouville quantum gravity as a mating of trees},
	JOURNAL = {Ast\'{e}risque},
	FJOURNAL = {Ast\'{e}risque},
	VOLUME = {427},
	YEAR = {2021}
}

@article{AHPS,
  title={Brownian half-plane excursion and critical {L}iouville quantum gravity},
  author={Aru, J. and Holden, N. and Powell, E. and Sun, X.},
  journal={Journal of the {L}ondon {M}athematical {S}ociety},
  volume={107},
  number={1},
  pages={441--509},
  year={2023},
  publisher={Wiley {O}nline {L}ibrary}
}

@book{DurrettRick2019PTaE,
series = {Cambridge series in statistical and probabilistic mathematics},
volume = {49},
year = {2019},
title = {Probability: {T}heory and {E}xamples},
edition = {Fifth edition.},
author = {Durrett, R.},
publisher = {Cambridge university press}
}

@article{le2013uniqueness,
  title={Uniqueness and universality of the {B}rownian map},
  author={Le Gall, J.-F.},
  journal={The {A}nnals of {P}robability},
  pages={2880--2960},
  year={2013}
}

@article{miermont2013brownian,
  title={The {B}rownian map is the scaling limit of uniform random plane quadrangulations},
  author={Miermont, G.},
  journal={Acta {M}athematica},
  volume={210},
  number={2},
  pages={319--401},
  year={2013}
}

@article{curien2025scaling,
  title={The scaling limit of planar maps with large faces},
  author={Curien, N. and Miermont, G. and Riera, A.},
  journal={arXiv:2501.18566},
  year={2025}
}

@article{duminil2013parafermionic,
  title={Parafermionic observables and their applications to planar statistical physics models},
  author={Duminil-Copin, H.},
  journal={Ensaios {M}atematicos},
  volume={25},
  pages={1--371},
  year={2013}
}

@article{duminil2021discontinuity,
  title={Discontinuity of the phase transition for the planar random-cluster and {P}otts models with $q>4$},
  author={Duminil-Copin, H. and Gagnebin, M. and Harel, M. and Manolescu, I. and Tassion, V.},
  journal={Annales {S}cientifiques de l'{\'E}cole {N}ormale {S}up{\'e}rieure},
  volume={54},
  number={6},
  pages={1363--1413},
  year={2021}
}

@article{feng2023triviality,
  title={Triviality of critical {F}ortuin-{K}asteleyn decorated planar maps for $q> 4$},
  author={Feng, Y.},
  journal={Probability {T}heory and {R}elated {F}ields},
  pages={1--33},
  year={2025},
  publisher={Springer}
}

@article{aidekon2024scaling,
  title={The scaling limit of the volume of loop ${O}(n)$ quadrangulations},
  author={A{\"\i}d{\'e}kon, E. and Da Silva, W. and Hu, X.},
  journal={Communications in Mathematical Physics},
  volume={407},
  number={42},
  year={2024}
}

@article{kammerer2024gaskets,
  title={Gaskets of ${O}(2)$ loop-decorated random planar maps},
  author={Kammerer, E.},
  journal={arXiv:2411.05541},
  year={2024}
}

@book{revuz2013continuous,
  title={Continuous martingales and {B}rownian motion},
  author={Revuz, D. and Yor, M.},
  volume={293},
  year={2013},
  publisher={Springer {S}cience \& {B}usiness {M}edia}
}

@article{berestycki2025critical,
  title={Critical behaviour of the fully packed loop-${O}(n)$ model on planar triangulations},
  author={Berestycki, N. and Da Silva, W.},
  journal={arXiv:2512.05867},
  year={2025}
}

\end{document}